\documentclass[10pt,reqno]{amsart}
\usepackage[applemac]{inputenc}
\usepackage[english]{babel}
\usepackage{caption}
\usepackage{amsmath}
\usepackage{bm}
\usepackage{bbm}
\usepackage{xcolor}
\usepackage{graphicx}
\usepackage{amssymb}
\usepackage{braket}
\usepackage{enumitem}
\usepackage{dsfont}
\usepackage{stmaryrd}
\usepackage{url}
\usepackage[left=3.25cm,right=3.25cm,top=3.25cm,bottom=3.25cm]{geometry}

\usepackage{color}
\usepackage{color}

\newcommand{\Ban}{E} 
\newcommand{\Dom}{{\mathcal D}}

\newcommand{\JDif}{\mathrm{Dif}(\Dom,\Jcal)}
\newcommand{\Dcal}{{\mathcal D}}

\newcommand{\Jcal}{{\mathcal J}}
\newcommand{\Lcal}{{\mathcal L}}
\newcommand{\Ocal}{{\mathcal O}}
\newcommand{\calO}{{\mathcal O}}
\newcommand{\Rcal}{{\mathcal R}}
\newcommand{\Xcal}{{\mathcal X}}
\newcommand{\Ycal}{{\mathcal Y}}
\newcommand{\N}{\mathbb{N}}
\newcommand{\Z}{\mathbb{Z}}
\newcommand{\R}{\mathbb{R}}

\newcommand{\C}{\mathbb{C}}

\newcommand{\bbone}{\mathbb{I}}

\newcommand{\rd}{\mathrm{d}}

\newcommand{\T}{\mathbb{T}}

\newcommand\HessL{D^2_{u_\xi}\Lcal}
\newcommand\Ucalxi{\mathcal U_\xi}

\newcommand\Space{U} 
\newcommand{\Ker}{{\rm Ker}\,}
\newcommand\restr[2]{{
  \left.\kern\nulldelimiterspace 
  #1 
  \vphantom{|} 
  \right|_{#2} 
  }}
  
\newcommand{\scalar}{\langle\cdot,\cdot\rangle}

\newcommand{\scal}[2]{{\langle#1,#2\rangle}}

\newcommand{\scalbb}[2]{{\langle#1,#2\rangle}_{\Ban}}
\newcommand{\nablasq}{\nabla^2{\mathcal L}_{\xi_*}}

\newcommand{\vertiii}[1]{{\left\vert\kern-0.25ex\left\vert\kern-0.25ex\left\vert #1 
    \right\vert\kern-0.25ex\right\vert\kern-0.25ex\right\vert}}

\newcommand{\diff}{\,\mathrm{d}}

\renewcommand\emptyset{\mbox{\Large \o}}

\newcommand{\calC}{\mathcal{C}}

\newcommand{\calL}{\mathcal{L}}
\newcommand{\Ucal}{\mathcal{U}}

\newcommand\DomH{\Dcal(\nabla^2\calL_\xi (u_\xi))}

\newcommand{\sech}{\mathrm{sech}}
\newcommand{\en}{\xi}

\newcommand{\re}{\mathrm{Re}}
\newcommand{\im}{\mathrm{Im}}

\newtheorem{theorem}{Theorem}[section]
\newtheorem{proposition}[theorem]{Proposition}
\newtheorem{corollary}[theorem]{Corollary}
\newtheorem{lemma}[theorem]{Lemma}
\theoremstyle{definition}
\newtheorem{definition}[theorem]{Definition}

\theoremstyle{remark}
\newtheorem{remark}[theorem]{Remark}
\numberwithin{equation}{section}

 \newtheorem*{rem*}{Remark}
 \theoremstyle{definition}
  
 \numberwithin{equation}{section}

\newcommand\vb{|}

 \newcommand{\be}[1]{\begin{equation}\label{#1}}
\newcommand{\ee}{\end{equation}}

\def\squarebox#1{\hbox to #1{\hfill\vbox to #1{\vfill}}}

\title[Orbital stability via the energy-momentum method]{Orbital stability via the energy-momentum method: the case of higher dimensional symmetry groups}
\author[S. {De Bi\`{e}vre}]{Stephan {De Bi\`{e}vre}$^{1,2}$}
\address{$^1$Laboratoire Paul Painlev\'e, CNRS, UMR 8524 et UFR de Math\'ematiques,
Universit\'e de Lille,
F-59655 Villeneuve d'Ascq Cedex, France.}
\email{Stephan.De-Bievre@math.univ-lille1.fr}

\address{$^2$
Equipe-Projet MEPHYSTO,
Centre de Recherche INRIA Futurs,
Parc Scientifique de la Haute Borne, 40, avenue Halley B.P. 70478,
F-59658 Villeneuve d'Ascq cedex, France.}

\author[S. {Rota Nodari}]{Simona {Rota Nodari}$^3$}
\address{$^3$ Institut de Mathématiques de Bourgogne (IMB), CNRS, UMR 5584,
Universit\'e Bourgogne Franche-Comt\'e,
F-21000 Dijon, France.}
\email{simona.rota-nodari@u-bourgogne.fr}

\begin{document}
\begin{abstract} We consider the orbital stability of relative equilibria of Hamiltonian dynamical systems on Banach spaces, in the presence of a multi-dimen\-sional invariance group for the dynamics. We prove a persistence result for such relative equilibria, present a generalization of the Vakhitov-Kolokolov slope condition to this higher dimensional setting, and show how it allows to prove the local coercivity of the Lyapunov function, which in turn implies orbital stability. The method is applied to study the orbital stability of relative equilibria of nonlinear Schr\"odinger and Manakov equations. We provide a comparison of our approach to the one by Grillakis-Shatah-Strauss. 
\end{abstract}
\maketitle       

\section{Introduction}
The \emph{Vakhitov-Kolokolov slope condition}~\cite{gssI, stuart08, vk} is an often used ingredient in the proof of orbital stability of relative equilibria via the energy-momentum method for Hamiltonian systems with a one-dimensional symmetry group. For example, it has been applied in the proofs of stability of stationary or traveling waves of a variety of nonlinear partial differential  equations~(see \cite{pava, bonsoustr1987, ans, eect, ghilecsqu2013, ianlec2009, stuart08, waywei2015} and references therein). It is our goal here to present a natural generalization of this condition to the case where the Hamiltonian system admits a higher dimensional invariance group and to show how to obtain orbital stability from it.

The overall strategy underlying the energy-momentum method is well understood. Simply stated, it is a generalization of the standard Lyapunov method for proving the stability of fixed points to Hamiltonian systems having a Lie symmetry group $G$. Indeed, relative equilibria can be seen as fixed points ``modulo symmetry": they are fixed points of the dynamics induced on the space obtained by quotienting the phase space by the action of an appropriate subgroup of the invariance group.  For finite dimensional systems, the theory goes back to the nineteenth century. It is concisely explained in~\cite{am, ma}, in the modern language of Hamiltonian systems with symmetry, through the use of the properties of the momentum map, notably. This theory first of all gives a simple geometric characterization of all relative equilibria. It also naturally provides a candidate Lyapunov function as well as subgroup of $G$ with respect to which the relative equilibria can be hoped to be relatively stable. More recent developments in the finite dimensional setting can be found in~\cite{mon, monrod, pat, pat95, lermansinger98, or, patrobwulff, RSS}. 

When the Hamiltonian system is infinite dimensional, such as is the case for nonlinear Hamiltonian PDE's, the general philosophy of the energy-momentum method remains the same, but many technical complications arise, as expected. In~\cite{gssI} and~\cite{stuart08}, the theory is worked out in a Hilbert space setting, and when the symmetry group $G$ is a \emph{one-dimensional} Lie group. More recently, in~\cite{debgenrot15}, a version of the energy-momentum method has been presented for Hamiltonian dynamical systems on a Banach space $E$ having as invariance group a Lie group $G$ of \emph{arbitrary finite dimension}. What is shown there is that the proof of orbital stability can be reduced to a ``local coercivity estimate'' on an appropriately constructed Lyapunov function $\Lcal$. It is shown in~\cite{debgenrot15} that, in the above infinite dimensional setting, the construction of the latter follows naturally from the Hamiltonian structure and basic properties of the momentum map, in complete analogy with the finite dimensional situation. In specific models, it then remains to show the appropriate local coercivity estimate on~$\Lcal$ which amounts to a lower bound on its Hessian restricted to an appropriate subspace of $E$ (see~\eqref{eq:hessiancondition}). 

When the invariance group $G$ of the system is one-dimensional, one way to obtain such an estimate is via the aforementioned \emph{Vakhitov-Kolokolov slope condition}. Our main result here is a generalization of this condition to situations with a higher dimensional invariance group $G$, and the proof that it implies the desired coercivity of the Lyapunov function (Theorem~\ref{thm:localcoercivity} and Theorem~\ref{thm:localcoercivitygen}). Using this property, orbital stability can then for example be obtained using the techniques described in~\cite{debgenrot15}.

The rest of this paper is organized as follows. The elements of~\cite{debgenrot15} needed here are summarized in Section~\ref{s:enmommethod}. In Section~\ref{s:maintheoremhil}, we state our main result (Theorem~\ref{thm:localcoercivity}) in the simplest setting, when the phase space of the system is a Hilbert space. Some preliminary lemmas are proven in Section~\ref{s:lemmas} and the proof of Theorem~\ref{thm:localcoercivity} is given in Section~\ref{s:proofmaintheorem}.  In Section~\ref{s:generalcase}, we generalize our result to the Banach space setting, see Theorem~\ref{thm:localcoercivitygen}. Section~\ref{s:persistence} deals with the important question of ``persistence'' of relative equilibra. In Section~\ref{s:examples}, we use our approach to study several applications of our results to the stability study for relative equilibria of the nonlinear Schr\"odinger and Manakov equations.
In Section~\ref{s:comparegss}, finally, a detailed analysis of the differences between our work here and the approach of~\cite{gssII} is provided. In the latter paper, the energy-momentum method had previously been adapted to the case of higher symmetry groups, and a generalization of the \emph{Vakhitov-Kolokolov slope condition} proposed. We have in particular amended, completed and generalized some of the results and arguments of this seminal work.
 
\medskip
\noindent{\bf Acknowledgments.}  This work was supported in part by the Labex CEMPI (ANR-11-LABX-0007-01) and by FEDER (PIA-LABEX-CEMPI 42527). The authors are grateful to Dr. M. Conforti, Prof. F.~Genoud, Prof. S.~Keraani, Prof. S.~Mehdi, Prof S. Trillo and Prof. G.~Tuynman for helpful discussions on the subject matter of this paper. They also thank an anonymous referee for instructive comments and a careful reading of the manuscript.

\section{The energy-momentum method}\label{s:enmommethod}
To make this paper self-contained and to fix our notation, we summarize in this section the energy-momentum method as described in~\cite{debgenrot15}. We refer there for more details and for examples of the structures introduced here.
\subsection{Hamiltonian systems with symmetry}
Let $\Ban$ be a Banach space, $\Dcal$ a domain in $\Ban$ (\emph{i.e.} a dense subset of $\Ban$) and $\Jcal$ a symplector, that is an injective continuous linear map $\Jcal:\Ban\to \Ban^*$ such that $(\Jcal u)(v)=-(\Jcal v)(u)$. We will refer to $\left(\Ban, \Dcal, \Jcal\right)$ as a \emph{symplectic Banach triple}. Next, let $H:\Ban\to \R$ be differentiable on $\Dcal\subset E$. In other words, $H$ is globally defined on $\Ban$, and differentiable at each point $u\in \Dcal$. We say that the function $H$ has a $\Jcal$-compatible derivative if, for all $u\in\Dom$, $D_uH\in\Rcal_{\!\Jcal}$, where $\Rcal_{\!\Jcal}$ is the range of $\Jcal$. In that case we write $H\in\JDif$.

We define a \emph{Hamiltonian flow} for $H\in\JDif$ as a separately continuous map $\Phi^H:\R\times \Ban\to\Ban$ with the following properties:
\begin{enumerate}[label=({\roman*})]
\item For all $t,s\in\R$, $\Phi_{t+s}^H=\Phi_t^H\circ\Phi_s^H,\, \Phi_0^H=\mathrm{Id}$;
\item For all $t\in\R$, $\Phi_t^H(\Dom)=\Dom$;
\item For all $u\in\Dom$, the curve $t\in\R\to u(t):=\Phi_t^H(u)\in\Dcal\subset \Ban$ is differentiable and is the unique solution of 
\begin{equation}\label{eq:hameqmotionJbis}
\Jcal\dot u(t)=D_{u(t)}H,\quad u(0)=u.
\end{equation}
\end{enumerate}

Note that here and below, $D_uH\in\Ban^*$ is our notation for the derivative of $H$ at $u$. We refer to~\eqref{eq:hameqmotionJbis} as the Hamiltonian differential equation associated to $H$ and to its solutions as Hamiltonian flow lines. 

Next, let $G$ be a Lie group, $\frak{g}$ the Lie algebra of $G$ and $\Phi: (g,x)\in G\times \Ban\to \Phi_g(x)\in \Ban,$ an action of $G$ on $\Ban$. In what follows we will suppose all Lie groups are connected. We will say $\Phi$ is a \emph{globally Hamiltonian action} if the following conditions are satisfied:
\begin{enumerate}[label=({\roman*})]
\item For all $g\in G$, $\Phi_g\in C^1(E,E)$ is symplectic.
\item For all $g\in G$, $\Phi_g(\Dcal)=\Dcal$.
\item For all $\xi\in\frak{g}$, there exists $F_\xi\in C^1(\Ban, \R)\cap\JDif$ such that $\Phi_{\exp(t\xi)}=\Phi_t^{F_\xi}$  is a Hamiltonian flow as defined above and the map $\xi\to F_\xi$ is linear. 
\end{enumerate} 
Here and in what follows, we say
$\Psi\in C^0(E,E)\cap C^1(\Dcal, E)$ is a symplectic transformation if 
\begin{equation}\label{eq:sympltransfbis}
\forall u\in \Dcal, \forall v, w\in \Ban, (\Jcal D_u\Psi (v))(D_u\Psi(w)) = (\Jcal v)(w).
\end{equation}
Note that, in the above definition of globally Hamiltonian action, $\Psi=\Phi_g\in C^1(E, E)$. 
For further reference, we introduce, for all $u\in\Dcal$ and for all $\xi\in\frak g$,
\begin{equation}\label{eq:xxi}
X_\xi(u)=\frac{\rd }{\rd t}\Phi_{\exp(\xi t)}(u)_{\mid t=0}.
\end{equation}
It follows from the preceding definitions that
\begin{equation}\label{eq:xxiDF}
X_\xi(u)=\Jcal^{-1} D_{u}F_\xi.
\end{equation}
We will always suppose $G$ is a matrix group, in fact, a subgroup of GL$(\R^N)$. We can then think of the Lie algebra $\frak g$ as a sub-algebra of the $N\times N$ matrices ${\mathcal M}(N,\R)$ and define the adjoint action of $G$ on $\frak g$ via
$$
\mathrm{Ad}_g\xi=g\xi g^{-1},
$$
where in the right hand side we have a product of matrices. We will write $m=\mathrm{dim}\frak g=\mathrm{dim}\frak g^*$, where $\frak g^*$ designates the vector space dual of the Lie algebra $\frak g$. For details, we refer to Section~A.2 of~\cite{{debgenrot15}}, \cite{am} or~\cite{ma}. Note that, for each $u\in E$ fixed, one can think of $\xi\in\frak g\to F_\xi(u)\in\R$ as an element of $\frak g^*$. 
Hence, if we identify (as we always will) $
 \frak g$ and $\frak g^*$ with $\R^m$ and view $F$ as a map $F:E\to\R^m\simeq \frak g^*$, we can write
$$
F_\xi=\xi\cdot F,
$$
where $\cdot$ refers to the canonical inner product on $\R^m$. The map $F$ is called the momentum map of the symplectic group action and, in what follows, we will suppose that $F$ is  $\mathrm{Ad}^*$-equivariant which means that for all $g\in G$, for all $\xi\in \frak{g}$
\[
F_\xi\circ\Phi_g=F_{\mathrm{Ad}_{g^{-1}}\xi},
\]
or equivalently, $F\circ\Phi_g=\mathrm{Ad}_g^*F$.
Here $\mathrm{Ad}^*$ is the co-adjoint action of $G$ on $\frak g^*$.

Now, for all $\mu\in \frak{g}^*$, we define the isotropy group or stabilizer of $\mu$ as
$$
G_\mu=\{g\in G\mid \mathrm{Ad}_g^* \mu=\mu\};
$$ 
 $\frak g_\mu$ is the Lie algebra of $G_\mu$, and $\frak g^*_\mu$ its dual. Finally, for all $\mu\in\frak{g}^*\simeq\R^m$, let
$$
\Sigma_\mu=\{ u\in \Ban\mid F(u)=\mu\}.
$$
We will say $\mu$ is a \emph{regular value of $F$} if $\Sigma_\mu\not=\emptyset$ and if, for all $u\in\Sigma_\mu$, $D_uF$ is surjective (maximal rank). Then $\Sigma_\mu$ is a codimension $m$ sub-manifold of $E$ and its tangent space at $u\in\Sigma_\mu$ is
\begin{equation}
T_u\Sigma_\mu=\mathrm{Ker} D_uF.
\end{equation}
Finally, since the momentum map is Ad$^*$-equivariant, it is easy to see $G_\mu=G_{\Sigma_\mu}$, where $G_{\Sigma_\mu}$ is the subgroup of $G$ leaving $\Sigma_\mu$ invariant. 

Below, $G$ will be an invariance group of $H$, in the sense that $H\circ \Phi_g=H$, for all $g\in G$. This implies $G$ is an invariance group for the dynamics generated by $H$, meaning that for all $g\in G, t\in\R$, $\Phi_g\circ\Phi_t^H=\Phi_t^H\circ \Phi_g$ (See Theorem~\ref{thm:caracterizereleq}~(i) below). Noether's Theorem then implies that the components $F_i$ of the moment map are constants of the motion (See Theorem~\ref{thm:caracterizereleq}~(ii)) and hence that, for any $\mu\in\R^m\simeq \mathfrak g^*$, the level set $\Sigma_\mu$ is invariant under the dynamics $\Phi_t^H$. We refer to Sections~\ref{s:examples} and~\ref{s:example} for examples; see also~\cite{debgenrot15}.

\subsection{Relative equilibria and orbital stability}\label{s:releqorbstab} We now recall the definition of a relative equilibrium. Let $G$ be an invariance group for the dynamics $\Phi_t^H$, as above, and let $\tilde G$ be a subgroup of $G$. Let $u\in\Ban$ and let $\mathcal{O}_u^{\tilde G}=\Phi_{\tilde G}(u)$ be the $\tilde G$-orbit of $u$. We say $u$ is a  relative $\tilde G$-equilibrium of the dynamics if, for all $t\in\R$, $\Phi_t^H(u)\in\mathcal{O}_u^{\tilde G}$. In other words, if the dynamical trajectory through $u$ lies in the group orbit $\mathcal O_u^{\tilde G}$. 

The goal is to investigate under which circumstances these relative equilibria are orbitally stable. Recall that a relative $\tilde G$-equilibrium $u\in \Ban$ is orbitally stable if
\[
\forall \varepsilon>0,\exists \delta>0,\forall v\in \Ban, \left(\rd(v,u)\le \delta \Rightarrow \forall t \in \R,\  \rd(v(t),\Ocal_{u}^{\tilde G})\le \varepsilon\right),
\]
with $v(t)$ the solution of the Hamiltonian equation of motion with initial condition $v(0)=v$.  Here $d(\cdot, \cdot)$ is the distance function induced by the norm on $\Ban$. Note that the definitions of relative equilibrium and of orbital stability are increasingly restrictive as the subgroup $\tilde G$ is taken smaller. Sharper statements are therefore obtained by choosing smaller subgroups $\tilde G$. 

It turns out that, if $H$ is $G$ invariant and the action of $G$ is Ad$^*$-equivariant, then $u$ is a $G$-relative equilibrium if and only if $u$ is a $G_\mu$-relative equilibrium, where $\mu=F(u)$ (See Theorem~7 in~\cite{debgenrot15}). This observation, familiar from the finite dimensional theory (See for instance~\cite{am, ma}), explains why it is natural to try to prove $G_\mu$-orbital stability. This is the approach we shall adopt here. It differs  from the one in~\cite{gssII}, where orbital stability is studied with respect to an a priori different subgroup, as we will explain in detail in Section~\ref{s:comparegss}.  We will also show there that in many situations of interest, the two subgroups actually coincide.

We will write 
\begin{equation}\label{eq:Gmuorbit}
\mathcal O_u=\Phi_{G_\mu}(u),
\end{equation}
where $\mu=F(u)$. And, for all $u\in\Dcal$,
\begin{equation}
	\label{eq:orbits}
	T_{u}\Ocal_{u}=\{X_\xi (u)\mid \xi\in \frak g_\mu\}\subset \Ban.
\end{equation}

For later reference, we recall the following definitions.
\begin{definition}\label{def:regular}
We say $F$ is regular at $u\in\Ban$ if $D_uF$ is of maximal rank. We say $\mu$ is a regular value of $F$, if for all $u\in\Sigma_\mu$, $D_uF$ is of maximal rank. We will refer to relative equilibria $u$ for which $D_uF$ is of maximal rank, as  \emph{regular relative equilibria}.  
\end{definition}

To understand what follows, it is helpful to keep in mind that in practice, the action of the invariance group $G$ is well known explicitly, and typically linear and isometric. The dynamical flow $\Phi_t^H$, on the other hand, is a complex object one tries to better understand using the invariance properties of $H$. 

We now collect some results from~ \cite{debgenrot15} which give a characterization of the relative equilibria of Hamiltonian systems with symmetry and which also yield the candidate Lyapunov function that can be used to study their stability. 

\begin{theorem}\label{thm:caracterizereleq} Let $(\Ban, \Dcal, \Jcal)$ be a symplectic Banach triple. Let $H\in C^1(\Ban, \R)\cap \JDif$ and suppose $H$ has a Hamiltonian flow $\Phi_t^H$. Let furthermore $G$ be a Lie group, and $\Phi$ a globally Hamiltonian action on $\Ban$ with Ad$^*$-equivariant momentum map $F$. Suppose that,
\begin{equation}
\forall g\in G,\quad H\circ \Phi_g=H.\label{eq:HGinv}
\end{equation}
\begin{enumerate}[label=({\roman*})]
\item Then $G$ is an invariance group for $\Phi_t^H$.
\item For all $t\in\R$, $F\circ \Phi_t^H=F$.
\item $u$ is a relative $G$-equilibrium if and only if $u$ is a relative $G_\mu$-equilibrium.
\item Let $u\in \Dcal\subset \Ban$. If there exists $\xi\in \frak g$ so that 
\begin{equation}\label{eq:defrelequilibria}
D_uH-\xi\cdot D_u F=0,
\end{equation}
then $u$ is a relative $G_\mu$-equilibrium. Let $\mu= F(u)\in\R^m\simeq\frak g^*$; if $\mu$ is a regular value of $F$, then $u$ is a critical point of $H_\mu$ on $\Sigma_\mu$, where $H_\mu=H_{\mid_{\Sigma_\mu}}$. 
\end{enumerate}
\end{theorem}
Equation~\eqref{eq:defrelequilibria} is referred to as the stationary equation in the PDE literature. The theorem states that its solutions determine relative $G$- and hence relative $G_\mu$-equilibria. 

We now turn to the stability analysis of those relative equilibria. Suppose we are given $\xi\in\frak g$ and $u_\xi$, solution of~\eqref{eq:defrelequilibria}. We first note that the fact that $u_\xi$ is a critical point of the restriction $H_{\mu_\xi}$ of the Hamiltonian $H$ to $\Sigma_{\mu_\xi}$ ($\mu_\xi=F(u_\xi)$)  is an immediate consequence of the observation that  $u_{\xi}$ is a critical point of the Lagrange function 
\begin{equation}
	\label{eq:lyapfunction}
\Lcal_{\xi}=H-\xi\cdot F:\Ban\to \R.
\end{equation} 
 The goal is now to prove that these relative equilibria are orbitally stable. As pointed out in \cite{debgenrot15}, the basic idea underlying the energy-momentum method is that, modulo technical problems, a relative equilibrium is expected to be stable if it is not only a critical point but actually a local minimum of $H_{\mu_\xi}$. To establish such a result, it is natural to use the second variation of the Lagrange multiplier theory and to establish that the Hessian of $\Lcal_\xi$ is positive definite when restricted to $T_{u_\xi}\Sigma_{\mu_\xi}\cap T_{u_\xi}\Ocal_{u_\xi}^\perp$. The precise statement is given in Proposition~\ref{prop:hessianestimate} below.

Let $\scalar$ be a scalar product on $\Ban$, which is continuous in the sense that 
\begin{equation}\label{eq:scalprodbis}
\forall v, w\in \Ban, \quad |\langle v, w\rangle|\leq \|v\|_\Ban \|w\|_\Ban,
\end{equation}
where $\|\cdot\|_\Ban$ is our notation for the Banach norm on $\Ban$. Note that $\Ban$ is not necessarily a Hilbert space for this inner product. In addition, even if $(\Ban, \|\cdot\|_E)$ is in fact a Hilbert space with the norm $\|\cdot\|_E$ coming from an inner product $\langle\cdot,\cdot\rangle_E$, the second inner product $\scalar$ is not necessarily equal to $\langle\cdot,\cdot\rangle_E$. 

Let $\|\cdot\|$ be the norm associated to the scalar product $\scalar$ and define $\hat\Ban$ to be the closure of $\Ban$ with respect to the $\|\cdot\|$-norm, that is 
\begin{equation}\label{eq:hatban}
\|\cdot\|=\sqrt{\scalar},\quad \hat \Ban = \overline\Ban^{\|\cdot\|}.
\end{equation} 
Note that $\hat \Ban$ is a Hilbert space and $\Ban \subset \hat \Ban$. As an example, one can think of $\Ban=H^1(\R^n)$ and $\scalar=\scalar_{L^2(\R^n)}$ so that $\hat\Ban=L^2(\R^n)$ in that case. This is the typical situation for the nonlinear Schr\"odinger equation; we refer to~\cite{debgenrot15} and Section~\ref{s:examples} for details and further examples.

For further reference, we collect our main structural assumptions in the following hypotheses:
\vskip 0.2cm
\noindent{\bf Hypothesis A.} Let $(E,\mathcal J, \mathcal D, \langle\cdot, \cdot\rangle, H, G, \Phi, F)$ and suppose:\\
(i) $(\Ban, \mathcal J, \mathcal D)$ is a symplectic Banach triple and $\scalar$ a continuous scalar product on $(\Ban, \|\cdot\|_E)$ as in~\eqref{eq:scalprodbis}. \\
(ii) $H\in C^2(\Ban,\R)\cap \JDif$ \\ 
(iii) $G$ is  a Lie group, and $\Phi$ a globally Hamiltonian $G$-action on $\Ban$ with Ad$^*$-equivariant momentum map $F\in C^2(\Ban, \R^m)$.\\
(iv) $H\circ \Phi_g=H$ for all $g\in G$.

\vskip0.2cm
\noindent{\bf Hypothesis B.}  $\Phi_g$ is linear and preserves both the structure $\scalar$ and the norm $\|\cdot\|_\Ban$ for all $g\in G$.
\vskip0.2cm
One then has:
\begin{proposition}\label{prop:hessianestimate} Suppose Hypotheses A and B hold. 
Let $\xi\in \frak g$ and suppose $u_{\xi}\in \Dom$ satisfies~\eqref{eq:defrelequilibria}, \emph{i.e.} $D_{u_{\xi}}\Lcal_{\xi}=0$, with $\Lcal_{\xi}=H-\xi\cdot F$. Let $\mu_\xi=F(u_\xi)\in \R^m\simeq \frak{g}^*$ and suppose $\mu_\xi$ is a regular value of $F$. 
Suppose in addition that 
\begin{enumerate}[label=({\roman*})]
\item\label{regularityPhi_g} $g\in G_{\mu_\xi}\to \Phi_g(u_\xi)\in \Ban$ is $C^2$.
\item $\forall j=1,\ldots,m$,
\begin{equation}
	\label{eq:hypXbis}
	 \exists \nabla F_j(u_{\xi})\in \Ban \text{ such that }  D_{u_{\xi}}F_j(w)=\scal{ \nabla F_j(u_{\xi})}{w}\ \forall w\in \Ban;
\end{equation}
\item There exists $C>0$ so that 
\begin{equation*}
\forall w\in \Ban,\  D^2_{u_\xi}\Lcal_{\xi}(w,w)\leq C\|w\|_\Ban^2;
\end{equation*}
\item There exists $c>0$ so that 
\begin{equation}\label{eq:hessiancondition}
\forall w\in T_{u_\xi}\Sigma_{\mu_\xi}\cap (T_{u_\xi}\Ocal_{u_{\mu_\xi}})^\perp, \ 
D^2_{u_\xi}\Lcal_{\xi}(w,w)\geq c\|w\|_\Ban^2,
\end{equation}
where
\begin{equation}\label{eq:perp}
\left(T_{u_\xi}\Ocal_{u_\xi}\right)^\perp=\{z\in \Ban\mid \scal{z}{y}=0, \forall y\in T_{u_\xi}\Ocal_{u_\xi}\}.
\end{equation}
\end{enumerate}
Then there exist $\epsilon>0$, $c>0$ so that
\begin{equation}
\label{eq:coercivityglobal}
\forall u\in\Ocal_{u_{\xi}}, \forall u'\in \Sigma_{\mu_\xi}, \quad \rd(u,u')\leq \epsilon\Rightarrow H(u')-H(u)\geq c\rd^2(u', \Ocal_{u_{\mu_\xi}}).
\end{equation}
\end{proposition}
This result constitutes a generalization of Proposition~5 in~\cite{debgenrot15}. In fact, if $G_\mu$ is commutative, the latter result applies immediately. If not, the mild regularity condition~(i) of the proposition suffices to obtain the result. We will give the proof of Proposition~\ref{prop:hessianestimate} in the next subsection. The basic message of this result is the following. If $G$ is an invariance group for $H$ that has a globally Hamiltonian action on $\Ban$ and if $u_\xi$ satisfies the stationary equation $D_{u_\xi}\mathcal L_\xi=0$ for some $\xi\in\frak g$, then,  modulo the technical conditions of the proposition, the coercive estimate~\eqref{eq:hessiancondition} on the Hessian of $\mathcal L_\xi$ implies that the restriction of the Hamiltonian $H$ to the constraint surface $\Sigma_{\mu_\xi}$ attains a local minimum on the $G_{\mu_\xi}$-orbit $\Ocal_{u_\xi}$.  As explained in Section 8 of~\cite{debgenrot15}, modulo some further technical conditions,~\eqref{eq:coercivityglobal} allows one to show that $u_\xi$ is $G_{\mu_\xi}$-orbitally stable. (See in particular Theorem~10 and Theorem~11 in~\cite{debgenrot15}). We will give details in the examples of Section~\ref{s:examples} below.

The difficulty in proving~\eqref{eq:hessiancondition} comes from the fact that, in general, the bilinear symmetric form $\HessL_\xi$ is not positive on $\Ban$, but has instead a non-trivial negative cone
$$
\calC_-=\{v\in \Ban\mid \HessL_\xi(v,v)<0\}.
$$
The estimate~\eqref{eq:hessiancondition} implies that $T_{u_\xi}\Sigma_{\mu_\xi}$ does not intersect $\calC_-$. To prove this, we shall show that there exists a maximally negative subspace of $\Ban$ for $\HessL_\xi$ which is $\HessL_\xi$-orthogonal to $T_{u_\xi}\Sigma_{\mu_\xi}$. 

The main goal of this paper is to give a condition (see Theorem~\ref{thm:localcoercivity} $(iv)$ and Theorem~\ref{thm:localcoercivitygen} $(iv)$), which is a generalization to the \emph{Vakhitov-Kolokolov slope condition}, that implies the coercivity estimate \eqref{eq:hessiancondition}. This condition is in general easier to verify than the coercivity estimate itself and allows one  to prove the orbital stability of relative equilibria of general Hamiltonian system. As an example of this method we study in Section~\ref{s:examples} several applications of our results to the stability analysis of relative equilibria of nonlinear Schr\"odinger and Manakov equations.

 \subsection{Proof of Proposition~\ref{prop:hessianestimate}}
   
 The general strategy of the proof is identical to the one of Proposition~5 in~\cite{debgenrot15}. First, we need some simple preliminary results.
 
 \begin{lemma}\label{lem:stuff}
Suppose the hypotheses of Proposition~\ref{prop:hessianestimate} hold.
Then $\forall u\in \Ocal_{u_{\xi}}$, for all $g\in G_{{\mu_\xi}}$, we have
\begin{enumerate}[label=({\alph*})]
\item $
\Phi_g(T_u\Ocal_{u_{\xi}})=T_v\Ocal_{u_{\xi}}, \quad \Phi_g\left(\left(T_u\Ocal_{u_{\xi}}\right)^\perp\right)=\left(T_v\Ocal_{u_{\xi}}\right)^\perp,
$
\item $
\Phi_g(T_u\Sigma_{\mu_\xi})=T_v\Sigma_{\mu_\xi}, \quad \Phi_g\left(\left(T_u\Sigma_{\mu_\xi}\right)^\perp\right)=\left(T_v\Sigma_{\mu_\xi}\right)^\perp,
$
\end{enumerate}
where $v=\Phi_g(u)$. In addition, defining
$V_u={\textrm span}\{\nabla_j F(u)\mid j=1\dots m\}$, we have
$$
V_u^\perp= T_u\Sigma_{\mu_\xi}\quad\text{and}\quad \Ban = T_u\Sigma_{\mu_\xi}\oplus  V_u.
$$
\end{lemma}

\begin{proof}
(a), respectively (b), follows from the observation that $\Phi_g(\Ocal_{u_{\xi}})=\Ocal_{u_{\xi}}$, respectively  $\Phi_g(\Sigma_{\mu})=\Sigma_{\mu}$, and the fact that $\Phi_g$ preserves the inner product $\langle\cdot, \cdot\rangle$. That $V_u^\perp= T_u\Sigma_{\mu}$ follows from the definitions and the second statement is easily verified. 
\end{proof}
\begin{proof}\emph{(of Proposition~\ref{prop:hessianestimate})}
We start with some preliminaries. First of all, we prove that there exists $\delta>0$ so that, for all $u'\in \Sigma_{\mu_\xi}$ for which $\rd(u', \Ocal_{u_\xi})< \delta$, there exists $g\in G_{\mu_{\xi}}$ so that $\Phi_g(u')-u_\xi\in \left(T_{u_\xi}\Ocal_{u_\xi}\right)^\perp$.

Let $\delta >0$ to be fixed later. We start by remarking that, since $\rd(u', \Ocal_{u_\xi})< \delta$, there exists $v'\in \Ocal_{u_\xi}$ such that $\|u'-v'\|_{\Ban}<\delta$. Moreover, since $v'\in \Ocal_{u_\xi}$, there exists $g'\in G_{\mu_\xi}$ such that $u_\xi=\Phi_{g'}(v')$. It follows that $\|\Phi_{g'}(u')-u_\xi\|_{\Ban}<\delta$. Next, let $\eta_1,\ldots,\eta_{m}$ be a basis of $\frak g_{\mu_\xi}$ and $X_{\eta_i}(u_\xi)$ defined as in~\eqref{eq:xxi}. We define
\begin{align}\label{defFcalift}
	\mathcal F : &\,\Ban\times G_{\mu_{\xi}}\to \R^m\nonumber\\
	&(w,g)\mapsto \begin{pmatrix} <\Phi_{g}(w)-u_\xi,X_{\eta_1}(u_{\xi})>\\\cdots\\<\Phi_{g}(w)-u_\xi,X_{\eta_m}(u_{\xi})>\end{pmatrix}.
\end{align}
Using Hypothesis~B, we can write:
$$
<\Phi_{g}(w)-u_\xi,X_{\eta_i}(u_{\xi})>=<w-\Phi_{g^-1}(u_\xi),\Phi_{g^-1}\left(X_{\eta_i}(u_{\xi})\right)>,
$$
for all $i=1,\ldots,m$. It follows from hypothesis~\ref{regularityPhi_g} of Proposition~\ref{prop:hessianestimate} and the fact that $u_{\xi}\in \Dom$, that $\mathcal F$ is of class $\mathcal C^1$.

Next, remark that $\mathcal F(u_{\xi}, e)=0$. 

We now compute the partial derivative of $\mathcal F$ along $G_{\mu_\xi}$ at the point $(u_\xi,e)$ denoted by 
$$
\partial_g \mathcal F(u_\xi, e) : \frak g_{\mu_{\xi}}\to \R^m.
$$
Note that $\partial_g \mathcal F(u_\xi, e)$ is a $m\times m$ matrix and the ${i,j}$-coefficient is obtained by writing $g=\exp(t\eta_j)$ and computing
\begin{align*}
\frac{d}{dt}\mathcal F_i(u_\xi,\exp(t\eta_j))_{\mid t=0}&=\frac{d}{dt}<u_\xi-\Phi_{\exp(-t\eta_j)}(u_\xi), \Phi_{\exp(-t\eta_j)}\left(X_{\eta_i}(u_{\xi})\right)>_{\mid t=0}\\
&=-<u_\xi, X_{\eta_j}\left(X_{\eta_i}(u_{\xi})\right)>=<X_{\eta_j}(u_\xi), X_{\eta_i}(u_{\xi})>. 
\end{align*}
Here we used Hypothesis~B again.  Since $\mu_\xi$ is a regular value of $F$, it follows that the $X_{\eta_i}$ are linearly independent. Hence, $\partial_g \mathcal F(u_\xi, e)$ is invertible and we can apply the implicit function theorem to $\mathcal F$. As consequence, there exist $\mathcal V_{u_\xi}$ a neighbourhood of $u_\xi$ in $\Ban$, $\mathcal V_{e}$ a neighbourhood of $e$ in $G_{\mu_\xi}$ and a function $\Lambda: \mathcal V_{u_\xi}\to \mathcal V_{e}$ such that if $v\in \mathcal V_{u_\xi}$ then there exists a unique $g_v=\Lambda(v)\in \mathcal V_{e}\subset G_{\mu_\xi}$ such that $\mathcal F(v, g_v)=0$ which means $\Phi_{g_v}(v)-u_{\xi}\in \left(T_{u_\xi}\Ocal_{u_\xi}\right)^\perp$.

Hence, by taking $\delta$ sufficiently small, we have $\Phi_{g'}(u')\in \mathcal V_{u_\xi}$ and we can conclude that there exists $\tilde g=\Lambda(\Phi_{g'}(u'))$ such that $\Phi_{\tilde g}\left(\Phi_{g'}(u')\right)-u_{\xi}\in \left(T_{u_\xi}\Ocal_{u_\xi}\right)^\perp$. Finally, by defining $g=\tilde g g'$, we obtain  $\Phi_{g}(u')-u_{\xi}\in \left(T_{u_\xi}\Ocal_{u_\xi}\right)^\perp$.

Let $u''=\Phi_{g}(u')$. Then $u''-u_\xi\in \left(T_{u_\xi}\Ocal_{u_\xi}\right)^\perp$ and, thanks to Lemma~\ref{lem:stuff},  we can write
$$
u''-u_\xi=(u''-u_\xi)_1+(u''-u_\xi)_2,
$$
where $(u''-u_\xi)_1\in T_{u_\xi}\Sigma_{\mu}\cap \left(T_{u_\xi}\Ocal_{u_\xi}\right)^\perp$ and $(u''-u_\xi)_2\in  V_{u_\xi}$.
Since $u'', u_\xi\in \Sigma_{\mu}$, one has
$$
0=F(u'')-F(u_\xi)=D_{u_\xi}F((u''-u_\xi)_2)+{O}(\|u''-u_\xi\|_\Ban^2).
$$
As $D_{u_\xi}F$ is of maximal rank, it has no kernel in $V_{u_\xi}$, and we conclude there exists $c_0$ so that
$$
\|(u''-u_\xi)_2\|_\Ban\leq {O}(\|u''-u_\xi\|_\Ban^2).
$$
Hence,  
$$
\|(u''-u_\xi)_1\|_\Ban\geq C\|u''-u_\xi\|_\Ban
$$
since $\|u''-u_\xi\|_\Ban$ is small by construction.

We can then conclude the proof as follows. Let $\epsilon>0$ be small enough so that the previous inequalities hold. Then compute
 \begin{eqnarray*}
\Lcal_{\xi}(u')-\Lcal_{\xi}(u_\xi)&=&\Lcal_{\xi}(u'')-\Lcal_{\xi}(u_\xi)\\
&=&D_{u_\xi}\Lcal_{\xi}(u''-u_\xi)+\frac12 D_{u_\xi}^2\Lcal_{\xi}(u''-u_\xi, u''-u_\xi)\\
&\ &\qquad+{o}(\|u''-u_\xi\|_\Ban^2)\\
&=&\frac12 D_{u_\xi}^2\Lcal_{\xi}((u''-u_\xi)_1,(u'-u_\xi)_1)+ {O}(\|u''-u_\xi\|_\Ban^3) \\
&\ &\qquad+{o}(\|u''-u_\xi\|_\Ban^2)\\
&=&\frac12 D_{u_\xi}^2\Lcal_{\xi}((u''-u_\xi)_1,(u''-u_\xi)_1)+{o}(\|u''-u_\xi\|_\Ban^2)\\
&\geq& \frac{c}{2}\|(u''-u_\xi)_1\|_\Ban^2+{o}(\|u''-u_\xi\|_\Ban^2)\\
&\geq &\tilde c\|u''-u_\xi\|_\Ban^2\ge\tilde c\rd^2(u'',\Ocal_{u_\xi})=\tilde c\rd^2(u',\Ocal_{u_\xi}).
\end{eqnarray*}
Note that, in the first equality above, we used the observation that, for all $g\in G_{\mu_\xi}$, for all $u'\in\Sigma_{\mu_\xi}$, one has
$$
\Lcal_\xi(\Phi_g(u'))=\Lcal_\xi(u').
$$ 
This follows from the $G$-invariance of $H$ and from the fact that
$$
\xi\cdot F(\Phi_g(u'))=\xi\cdot \mu_\xi=\xi\cdot F(u')
$$
since both $\Phi_g(u')$ and $u'$ belong to $\Sigma_{\mu_\xi}$. 
\end{proof}

\section{Main result: the Hilbert space setting}\label{s:maintheoremhil}

In this section, we state our main result (Theorem~\ref{thm:localcoercivity}) in the setting where $\Ban$ is a Hilbert space, with inner product $\langle\cdot,\cdot\rangle_E$, and $\| \cdot\|_E=\sqrt{\scalar_E}$. In particular, the inner product $\langle\cdot, \cdot\rangle$ in~\eqref{eq:scalprodbis} and in Hypothesis~A is, in this section, equal to $\langle\cdot,\cdot\rangle_E$.  The Hilbert space structure will be used mainly to control the Hessian of $\mathcal L_\xi$ through the spectral analysis of the associated bounded self-adjoint operator $\nabla^2\mathcal L_\xi$ (see below).  This makes for a simpler statement and proof than in the more general setting of Section~\ref{s:generalcase}, and allows for an easier understanding  of the philosophy of the result.  We point out that the result we obtain in Theorem~\ref{thm:localcoercivity} may be of interest also in finite dimensional problems (dim$\Ban<+\infty$). Indeed, the usual orbital stability results in the literature on finite dimensional Hamiltonian dynamical systems reduce their proof to the coercivity estimate~\eqref{eq:hessiancondition} on the Hessian of $\Lcal_\xi$, which is of dimension dim$\Ban$.  We reduce the problem here to a control on the Hessian of the function $W$ (see~\eqref{eq:WhatF}), which is of dimension $m=\mathrm{dim}G$, typically much lower than dim$\Ban$, especially when the latter is high-dimensional. 

We start with some preliminaries. We will make use of the following hypothesis:
\vskip 0.2cm
\noindent{\bf Hypothesis C.} There exists an open set $\Omega\subset\frak g \simeq \R^m$ and a map $\tilde u\in C^1(\Omega\subset\frak g, \Ban)$
\begin{equation}\label{eq:deftildeu}
\begin{aligned}
\tilde u :\,& \xi\in\Omega\subset \frak{g} \to u_\xi\in\Dcal\subset \Ban\\
\end{aligned} 
\end{equation}
satisfying, for all $\xi\in\Omega$,
\begin{equation}\label{eq:statequation}
D_{u_\xi}H-\xi\cdot D_{u_\xi}F=0.
\end{equation}
\vskip0.2cm
As recalled in section \ref{s:enmommethod}, if $u_\xi$ is a solution to \eqref{eq:statequation}, then $u_\xi$ is $G_\mu$-relative equilibrium with $\mu=\mu_\xi=F(u_\xi)$.

So our starting point is equation~\eqref{eq:statequation}, which in PDE applications is often an elliptic partial differential equation and we suppose we have an $m$-parameter family of solutions, indexed by $\xi$. One of the major difficulties to apply the theory is of course to find such families of solutions. In many cases, one has one single such solution for $\xi_*\in\frak g$ and one needs to ensure there exists a neighbourhood $\Omega$ of $\xi_*$ for which such solutions exist. We will come back to this property of ``persistence of relative equilibria'' in Section~\ref{s:persistence} and present results ensuring Hypothesis~C is satisfied. For now, we will suppose this is the case. 
Next, consider the Lyapunov function $\Lcal_\xi$ defined by \eqref{eq:lyapfunction} and remark that each $u_\xi$ solution to \eqref{eq:statequation} is a critical point of $\Lcal_\xi$. Moreover, define for all $\xi \in \Omega\subset\frak{g}$, 
the map
\begin{equation}\label{eq:WhatF}
\begin{aligned}
W:&\,\xi\in\Omega\subset\frak{g}\to \Lcal_\xi(u_\xi)\in \R.
\end{aligned}
\end{equation}
Note that
$$
W(\xi)=H(u_\xi)-\xi\cdot \hat F(\xi),
$$
where 
\begin{equation}
\label{eq:fhat}
\hat F :\, \xi\in\Omega\subset \frak{g} \to F(u_\xi)\in \frak{g}^*\simeq \R^m.
\end{equation}
For each $\xi\in\Omega$, the Hessian $D^2_\xi W$ of $W$ is a bilinear form on $\R^m$. Hence, we can consider the following decomposition 
$$
\R^m=W_-\oplus W_0\oplus W_+,
$$
where $W_0$ is the kernel of $D^2_\xi W$ and where $D^2_\xi W$ is positive (negative) definite on $W_+$ ($W_-$).  Let $d_0(D^2_\xi W), p(D^2_\xi W), n(D^2_\xi W)$ be the dimensions of these spaces. Note that the decomposition is not unique, but the respective dimensions are. In other words, $W_\pm$ are maximal positive/negative definite spaces for $D_\xi^2W$. Also, in order not to burden the notation, we have not made the $\xi$-dependence of the spaces $W_0, W_\pm$ explicit. Recall that, given a symmetric bilinear form $B$ on a Banach space $\Ban$, a subspace $\mathcal{X}$ of $\Ban$ is said to be a positive (negative) definite subspace for $B$ on $\Ban$ if $B_{\mid_{\mathcal{X}\times \mathcal{X}}}$ is positive (negative) definite. A subspace is maximally positive (negative) definite if it is positive (negative) definite and it is not contained in any other positive (negative) definite subspace. 

Similarly, the Hessian $D^2_{u}\calL_{\xi}$ of $\calL_\xi$ is a symmetric bilinear form on $\Ban$. For each $u\in\Ban$, we define as usual the Morse index $n(D^2_{u}\Lcal_\xi)$  of $u$ for $\mathcal L_\xi$ as the dimension of a maximally negative definite subspace for $D^2_{u}\Lcal_\xi$ in $\Ban$. 

Finally, when $\Ban$ is a Hilbert space, we can define for each $u\in \Ban$ a bounded self-adjoint operator $\nabla^2\calL_\xi (u)$ by 
\begin{equation}
	\label{eq:defHessian}
	\scal{v}{\nabla^2\calL_\xi (u)w}_E=D^2_u\calL_{\xi}(v,w).
\end{equation}
As a consequence, we can consider the spectral decomposition of $\Ban$ for  $\nabla^2\calL_\xi (u_\xi)$
\begin{equation}	
	\label{eq:spectraldec}
	E=E_-\oplus E_0\oplus E_+
\end{equation}
with $E_0=\mathrm{Ker} \nabla^2\calL_\xi (u_\xi)=\mathrm{Ker}D^2_{u_\xi}\calL_\xi$, and $E_\pm$ the positive and negative spectral subspaces of $\nabla^2\calL_\xi (u)$. Clearly 
$E_\pm$ are maximally positive/negative subspaces for $\nabla^2_{u_\xi}\calL_\xi$ so that $n(D^2_{u_\xi}\Lcal_\xi)=\dim\,E_-$. 
We can now state our main result. 
\begin{theorem}
	\label{thm:localcoercivity}
	Suppose $(\Ban, \langle\cdot,\cdot\rangle_E)$ is a Hilbert space and that Hypotheses A and C hold. Let $\xi\in \Omega$ and suppose
	\begin{enumerate}
		\item[(i)] $D^2_\xi W$ is non-degenerate,
		\item[(ii)]  $\mathrm{Ker}D^2_{u_\xi}\calL_\xi= T_{u_\xi}\Ocal_{u_\xi}$, \label{hyp:kernel}
		\item[(iii)]  $\inf (\sigma(\nabla^2\calL_\xi (u_\xi))\cap (0,+\infty))>0$,
		\item[(iv)]  $p(D^2_\xi W)=n(D^2_{u_\xi}\Lcal_\xi)$.
	\end{enumerate}
	Then there exists $\delta>0$ such that
	\begin{equation}
		\label{eq:localcoercivity}
		\forall v\in T_{u_\xi}\Sigma_{\mu_\xi}\cap \left(T_{u_\xi}\Ocal_{u_\xi}\right)^{\perp},\ D^2_u\calL_{\xi}(v,v)\ge \delta \| v\|_E^2.
	\end{equation}
\end{theorem}
We will say a relative equilibrium is \emph{non-degenerate} when $D^2_\xi W$ is non-degenerate. Since~\eqref{eq:localcoercivity} is the same as~\eqref{eq:hessiancondition}, one can then use Proposition~\ref{prop:hessianestimate} together with Theorems~10 and~11 of~\cite{debgenrot15} to show that $u_\xi$ is orbitally stable.  It is the fourth condition of the above theorem that generalizes the Vakhitov-Kolokolov slope condition, as we now explain. Suppose the group $G$ is $1$-dimensional, so that $m=1$. Then $W$ is a scalar function of $\xi\in\R\simeq \frak g$. We will see below (See~\eqref{eq:slopetrick}) that
$$
W''(\xi)=-\hat F'(\xi).
$$
Hence the proof of orbital stablity for $u_\xi$ reduces to verifying that the spectral conditions on $\nabla^2_\xi \mathcal{L}_\xi$ are satisfied and notably that dim$E_-=1$, and that
\begin{equation}
 \hat F'(\xi)<0.
\end{equation}
This is the Vakhitov-Kolokov slope condition. In applications to the Schr\"odinger equation, where $F(u)=\frac12\langle u, u\rangle$, it says that the norm of $u_\xi$ decreases as a function of $\xi$. In the case $m=1$, the above result is proven in~\cite{gssI} (Corollary~3.3.1) and in~\cite{stuart08} (Proposition~5.2). 

The setup of the Hamiltonian dynamics with higher dimensional symmetry on a Hilbert space we used in this section is similar to the one proposed in~\cite{gssII} where the decomposition~\eqref{eq:spectraldec} of the bounded self-adjoint operator $\nabla^2\Lcal_{\xi}(u_\xi)$ as well as condition~(iii) of Theorem~\ref{thm:localcoercivity} are also used to obtain a coercivity result of the type~\eqref{eq:localcoercivity}. Nevertheless, in~\cite{gssII} a different constraint surface and orbit are used and some of the arguments provided are incomplete: for a complete comparison between Theorem~\ref{thm:localcoercivity} and the coercivity results of~\cite{gssII}, we refer to Section~\ref{s:comparegss}. 

We finally note that, when $\Ban$ is infinite dimensional, and the equation under study a PDE, the more general formulation of Section~\ref{s:generalcase} is often considerably more pertinent than the simpler Hilbert space formulation proposed here.  Indeed, the operator 
$\nabla^2\mathcal L_\xi(u_\xi)$ introduced in~Theorem~\ref{thm:localcoercivity} is not a partial differential operator (it is bounded) making the analysis of its spectrum generally less convenient than for the operator $\nabla^2\mathcal L_\xi(u_\xi)$ in Theorem~\ref{eq:localcoercivitygen}, which in applications is a self-adjoint partial differential operator on a suitable auxiliary Hilbert space. We will come back to this point when treating examples in Section~\ref{s:examples}. 


\section{Useful lemmas}\label{s:lemmas}
The following lemmas collect some basic properties of the objects introduced above, that are essential in the further analysis of the Hessian of the Lyapunov function. We define, for $\xi\in\Omega$,
\begin{equation}\label{eq:Y}
\Ucalxi=\{\eta\cdot\nabla_\xi u_\xi\in E\mid \eta\in\R^m\},
\end{equation}
where we used the notation
\begin{equation}\label{eq:immersion}
\eta\cdot\nabla_\xi u_\xi:=D_\xi\tilde u(\eta).
\end{equation}
\begin{lemma}\label{lem:transversality}
Let $\Ban$ be a Banach space, $\Omega$ an open set in $\frak g$. Let $\tilde u\in C^1(\Omega\subset\frak g,  \Ban)$.  Let $\xi\in\Omega$ and consider the following statements:
\begin{enumerate}
\item $\hat F$ is a local diffeomorphism;
\item $D_\xi\tilde u$ is injective.
\item
$
\Ucalxi\cap \mathrm{Ker}D_{u_\xi}F=\{0\}.
$
\item  There is a neighbourhood of $u_\xi$ where  the moment map $F$ is regular (\emph{i.e.} $D_{u_\xi}F$ has maximal rank).
\item 
\begin{equation}\label{eq:transversality}
\Ucalxi\oplus \mathrm{Ker}D_{u_\xi}F=\Ban;
\end{equation}
\end{enumerate}
Then $(1) \Leftrightarrow \left( (2) \text{ and } (3) \right) \Leftrightarrow \left((4) \text{ and } (5)\right) $.   
\end{lemma}
Note that the lemma does not use the fact that the $u_\xi$ are solutions to the stationary equation: $\tilde u$ takes values in $\Ban$, without further condition. The lemma therefore strings together some useful facts on compositions of maps. 

It is easy to see that, whenever $u_\xi$ is a solution to  \eqref{eq:statequation} for every $\xi\in\Omega$, the map $\tilde u$ is injective provided the $u_\xi$ are regular relative equilibria.  Indeed, if $u_{\xi_1}=u_{\xi_2}$ are both solutions of~\eqref{eq:statequation}, then
$$
(\xi_1-\xi_2)\cdot D_{u_{\xi_1}}F=0.
$$
Hence, if the $u_\xi$ are regular relative equilibria (see Definition~\ref{def:regular}), one has $\xi_1=\xi_2$. 
It is natural in that context to assume it is in fact an immersion, meaning that its derivative is injective, as in condition~(2) of Lemma~\ref{lem:transversality}.  One can then think of $\tilde u(\Omega)$ as an $m$-parameter surface in $\Ban$. In applications, this additional condition often arises naturally from the construction of $\tilde u$, as seen in Section~\ref{s:persistence}. It is also a consequence of assumption (i) in Theorem~3.1, as a result of Lemma~\ref{lem:prophessWhessL}~(1) together with  Lemma~\ref{lem:transversality}~(1).
\begin{proof}  First note that, for all $\eta_1, \eta_2\in\frak g$, 
\begin{equation}\label{eq:diffFdiffFhat}
 D_\xi (\eta_2\cdot \hat F)(\eta_1)= D_{u_\xi}(\eta_2\cdot F) (\eta_1\cdot \nabla_\xi u_\xi).
\end{equation}
{\boldmath $(1) \Rightarrow \left((2) \text{\bf\,and\,} (3)\right)$} Let $\eta_1\in\frak g$ and suppose $\eta_1\cdot\nabla_\xi u_\xi=0$. It follows from~\eqref{eq:diffFdiffFhat} that $D_\xi\hat F(\eta_1)=0$. But since $\hat F$ is a local diffeomorphism at $\xi$, this implies $\eta_1=0$. Hence $D_\xi \tilde u$ is injective, which shows (2). 
To show (3), let $\eta_1\in\frak g$ and suppose $\eta_1\cdot\nabla_\xi u_\xi\in \mathrm{Ker}D_{u_\xi}F$. Then, by definition, $D_{u_\xi}(\eta_2\cdot F) (\eta_1\cdot \nabla_\xi u_\xi)=0$ for all $\eta_2\in\frak g$.  It follows from~\eqref{eq:diffFdiffFhat} that $\eta_1\in \mathrm{Ker} D_\xi\hat F$ so that, by (1), $\eta_1=0$. This proves~(3).

\noindent{\boldmath $\left((2) \text{\bf\,and\,} (3)\right) \Rightarrow (1) $} Let $\eta_1\in \mathrm{Ker}D_\xi\hat F$. Then according to the above equality, $\eta_1\cdot \nabla_\xi u_\xi\in \mathrm{Ker} D_{u_\xi} F$. So, by~(3), $\eta_1\cdot\nabla_\xi u_\xi=0$ and by~(2), $\eta_1=0$. This proves $D_\xi\hat F$ is injective, hence surjective, which proves~(1). \\
{\boldmath $\left((2) \text{\bf\,and\,} (3)\right)\Rightarrow \left((4) \text{\bf\,and\,} (5)\right)$}
According to~(3), the map
$$
D_{u_\xi}F: \Ucal_\xi \to \frak g^*\simeq\R^m
$$ 
is injective. But since by~(2), $D_{\xi}\tilde u$ is injective, the dimension of $\Ucal_\xi$ is $m$. Hence this map is a bijection. The rank of $D_{u_\xi}F$ is therefore maximal. By continuity of $D_uF$ in $u$, this remains true in a neigbhourhood of $u_\xi$, which proves~(4).
It follows from~(4) that locally, $\Sigma_{\mu_\xi}$ is a co-dimension $m$ submanifold of $\Ban$. Since, by definition, $T_{u_\xi}\Sigma_{\mu_\xi}=\mathrm{Ker}D_{u_\xi}F$,  we know from~(3) that  $\Ucalxi\cap T_{u_\xi}\Sigma_{\mu_\xi}=\{0\}$. Since, by~(2), the dimension of $\Ucalxi$ is $m$, (5) follows. \\
{\boldmath $\left((4) \text{\bf\,and\,} (5)\right)\Rightarrow \left((2) \text{\bf\,and\,} (3)\right)$}. This is obvious and concludes the proof of the lemma.
\end{proof}
We introduce
\begin{equation}\label{eq:gxi}
G_\xi=\{g\in G\mid \mathrm{Ad}_g\xi=\xi\},
\end{equation}
which is the subgroup of $G$ for which $\xi$ is a fixed point under the adjoint action. We will write $\frak g_\xi$ for its Lie-algebra. We furthermore need ($u_\xi\in\Dcal$)
\begin{equation}\label{eq:Zksi}
Z_\xi=\{X_\eta(u_\xi)\mid \eta\in\frak g_\xi\}\subset E
\end{equation}
where $X_\eta(u_\xi)$ is defined in~\eqref{eq:xxi}.

\begin{lemma}\label{lem:prophessWhessL}  Suppose Hypotheses A and~C hold.
Let $\xi\in\Omega$. 
Then, one has:
\begin{enumerate}[label=(\arabic*),ref=(\arabic*)]
\item For all $\eta\in\R^m$,
\begin{equation}\label{eq:equivker}
\eta\in W_0=\Ker D_\xi^2 W \Leftrightarrow \eta\in \mathrm{Ker}D_\xi\hat F \Leftrightarrow\eta\cdot \nabla_\xi u_\xi\in \Ker D_{u_\xi}F.
\end{equation}
In particular, $D_\xi^2 W$ is non-degenerate if and only if $\hat F$ is a local diffeomorphism at $\xi$ \label{eq:nondegeneracyW}. 
\item For all $\eta_1, \eta_2\in\R^m$,
\begin{equation}
\label{eq:hessWhessL}
D_{u_\xi}^2\Lcal_\xi(\eta_1\cdot \nabla_\xi u_\xi, \eta_2\cdot\nabla_\xi u_\xi)=-D^2_\xi W(\eta_1, \eta_2).
\end{equation}
\item For all $v\in \Ker D_{u_\xi}F$, for all $\eta\in\frak g$,
\begin{equation}
\label{eq:ortho}
D_{u_\xi}^2\Lcal(\eta\cdot\nabla_\xi u_\xi, v)=0.
\end{equation}
\item $T_{u_\xi}\Ocal_{u_\xi}$ is a subspace of the kernel of $(D^2_{u_\xi}\Lcal_\xi\mid \Ker D_{u_\xi}F)$, which is the restriction of $D^2_{u_\xi}\Lcal_\xi$ to $\Ker D_{u_\xi}F \times \Ker D_{u_\xi}F$ \label{prop:orbitinkernel}.
\item $Z_\xi\subset \mathrm{Ker}(D_{u_\xi}^2\Lcal_\xi)\subset \Ker(D_{u_\xi}F).$
\item $\frak g_\xi\subset \frak g_{\mu_\xi}$.
\end{enumerate}
\end{lemma}

Note that, combining~\eqref{eq:equivker} with Lemma~\ref{lem:transversality}, we can conclude that the directions $\eta\cdot\nabla_\xi u_\xi$ form a complementary subspace to $ \Ker D_{u_\xi}F$ when $D^2_\xi W$ is non-degenerate. Also, $u_\xi$ is a regular relative equilibrium, and the subspace  $\Ucal_\xi$ is complementary to the tangent space $T_{u_\xi}\Sigma_{\mu_\xi}$.

Expression~\eqref{eq:hessWhessL} is of interest since it identifies part of the Hessian of the Lyapunov function $\calL_\xi$ in terms of the Hessian of $W$. More precisely, it is useful to determine a subspace of negative directions of 
$D^2_{u_\xi}\Lcal_\xi$. Indeed, if $n_+=p(D^2_\xi W)$ and if $\{\eta_1,...,\eta_{n_+}\}$ is a family of linearly independent elements of $\R^m$ such that $\mathrm{span\,}\{\eta_1,...,\eta_{n_+}\}$ is a positive definite subspace for $D^2_\xi W$, then $\mathrm{span\,}\{\eta_1\cdot \nabla_\xi u_\xi,...,\eta_{n_+}\cdot \nabla_\xi u_\xi\}$ is a negative definite subspace for $D^2_{u_\xi}\Lcal_\xi$ (see~\eqref{eq:hessWhessL}). Thus, the dimension of a maximally negative definite subspace for $D^2_{u_\xi}\Lcal_\xi$ in $\Ban$ is at least $p(D^2_\xi W)$:
\begin{equation}\label{eq:nLpW}
n(\HessL_\xi)\geq p(D^2_\xi W).
\end{equation}
 
Expression~\eqref{eq:ortho} turns out to be crucial in what follows: it expresses the fact that $\mathcal U_\xi=
\{\eta\cdot\nabla_\xi u_\xi\mid \eta\in\R^m\}$
is $D^2_{u_\xi}\Lcal_\xi$-orthogonal to $\Ker D_{u_\xi}F$.
\begin{proof}
First of all, note that, since $u_{\xi'}$ is a solution to the stationary equation \eqref{eq:statequation} for all $\xi'\in \Omega$, for all $\eta \in \R^m$
\begin{equation}
D_\xi W(\eta)=-\eta\cdot F(u_\xi).
\end{equation}
Then a straightforward calculation gives, for all $\eta_1, \eta_2\in\R^m$, 
\begin{equation}\label{eq:hessWdiffF}
D_\xi^2W(\eta_1,\eta_2)=-D_{u_\xi}\eta_1\cdot F(\eta_2\cdot\nabla_\xi u_\xi).
\end{equation}
In other words, 
\begin{equation}\label{eq:slopetrick}
D_\xi^2W=-D_\xi\hat F.
\end{equation}
Note that, as $\hat F$ is a map from $\R^m\simeq \frak g$ to $\R^m\simeq\frak g^*$, $D_\xi\hat F$ is linear from $\R^m=\frak g$ to $\R^{m}=\frak g^*$. It therefore naturally defines a bilinear map on $\R^m\simeq \frak g$. In our notation here, we identify $\frak g$ with $\frak g^*$ using an Euclidean structure, but even without the latter, the above is natural. 

The first statement of \eqref{eq:equivker} is now obvious and for the second, note that $\eta\cdot\nabla_\xi u_\xi\in \mathrm{Ker}\,D_{u_\xi}F$ if and only if, for all $\eta'\in \R^m$,
$D_{u_\xi}\eta'\cdot F(\eta\cdot\nabla_\xi u_\xi)=0$, which yields the conclusion, thanks to \eqref{eq:diffFdiffFhat} and \eqref{eq:hessWdiffF}.

To obtain~\eqref{eq:hessWhessL}, it is sufficient to take the derivative of the stationary equation~\eqref{eq:statequation} with respect to $\xi\in\frak g$ and to use~\eqref{eq:hessWdiffF}. More precisely, by taking this derivative with respect to $\xi$ in the direction $\eta$, we obtain for all $\eta\in\frak g$, 
\begin{equation}\label{eq:hessX*}
D_{u_\xi}^2\Lcal_\xi(\eta\cdot \nabla_\xi u_\xi)=D_{u_\xi} \eta\cdot  F\in \Ban^*.
\end{equation}
Hence, using~\eqref{eq:hessWdiffF},
\[
D_{u_\xi}^2\Lcal_\xi(\eta_1\cdot \nabla_\xi u_\xi,\eta_2\cdot \nabla_\xi u_\xi)=D_{u_\xi} \eta_1\cdot  F(\eta_2\cdot \nabla_\xi u_\xi)=-D_\xi^2W(\eta_1,\eta_2).
\] 
Next, \eqref{eq:ortho} follows directly from \eqref{eq:hessX*}. Indeed, for all $v\in \Ker D_{u_\xi}F$ and for all $\eta\in\frak g$, $D_{u_\xi}^2\Lcal_\xi(\eta\cdot \nabla_\xi u_\xi,v)=D_{u_\xi} \eta\cdot  F(v)=0$.

From $F=(F\circ \Phi_{g^{-1}})\circ \Phi_g$ and $H=H\circ \Phi_g$ one finds, for all $u\in \Ban, g\in G$,
$$
D_u F=\left(D_{\Phi_g(u)}(F\circ\Phi_{g^{-1}})\right)D_u\Phi_g, \quad D_uH=D_{\Phi_g(u)}H D_u\Phi_g.
$$
Hence, by~\eqref{eq:statequation},
\begin{align*}
D_{\Phi_g(u_\xi)}H =&\, \xi\cdot D_{u_\xi}F\left(D_{u_\xi}\Phi_g\right)^{-1}=D_{\Phi_g(u_\xi)}(\xi\cdot F\circ\Phi_{g^{-1}})\\
=&\,D_{\Phi_g(u_\xi)}(Ad_g\xi\cdot F)
\end{align*}
and therefore
\begin{equation*}
D_{\Phi_g(u_\xi)}\left(H-\xi\cdot F\right) = D_{\Phi_g(u_\xi)}((Ad_g\xi-\xi)\cdot F).
\end{equation*}
Now let $\eta\in\frak g$, consider $g=\exp(t\eta)$ and take the derivative at $t=0$ in the previous relation. One finds, for all $v\in \Ban$,
\begin{equation}
D^2_{u_\xi}(H-\xi\cdot F)(X_\eta (u_\xi), v)=D_{u_\xi}[\eta,\xi]\cdot F(v).
\end{equation}
Taking $v\in \Ker D_{u_\xi}F$, the right hand side above vanishes for any $\eta\in \frak g$, and one finds \ref{prop:orbitinkernel} follows. 
 
To prove~(5), note that, taking $\eta\in \frak g_\xi$ so that $[\eta, \xi]=0$, we see that $Z_\xi\subset \Ker\HessL_\xi$. Finally, let $v\in\Ban$, then~\eqref{eq:hessX*} yields, for all $\eta\in\frak g$
$$
D_{u_\xi}^2\Lcal_\xi(\eta\cdot \nabla_\xi u_\xi, v)=\eta\cdot D_{u_\xi}F(v).
$$
Hence, if $v\in\Ker(\HessL_\xi)$, it follows that $v\in \Ker(D_{u_\xi}F)$. 

To prove~(6) note first that~(5) implies that, if $\eta\in \frak g_\xi$, we have 
$$
0=D_{u_\xi}F(X_\eta(u_\xi))=\frac{\rd}{\rd t} F\circ\Phi_{\exp(\eta t)}(u_\xi)_{\mid t=0}=\frac{\rd}{\rd t} {\mathrm Ad}^*_{\exp(\eta t)}F(u_\xi)_{\mid t=0}=\mathrm{ad}^*_\eta\mu_\xi.
$$
It follows that $\eta \in \frak g_{\mu_\xi}$. Hence $\frak g_\xi\subset \frak g_{\mu_\xi}$.
\end{proof}

Finally, we state some properties of symmetric bilinear forms and their associated quadratic forms in the form of a short lemma. In what follows, if $B$ is a bilinear form on some vectorspace $\Ban$, and $\mathcal Y$ is a subspace of $\Ban$, then we write $(B\mid \mathcal Y)$ for the restriction of $B$ to $\mathcal Y\times\mathcal Y$.
\begin{lemma}\label{lem:Borthopositive} Let $\Ban$ be a vector space and $B$ a symmetric bilinear form on $\Ban$. 
\begin{enumerate}[label=(\roman*)]
\item Let $\mathcal X_-$ be a maximally negative definite subspace for $B$ in $\Ban$. Suppose $\mathcal Y\subset \Ban$ is a subspace of $\Ban$ with the property that $\mathcal X_-\cap \mathcal Y=\{0\}$ and such that $B(\mathcal X_-, \mathcal Y)=0$. Then $\mathcal Y$ is a positive subspace for $B$. 
\item Let $\Ycal_1, \Ycal_2$ be two subspaces of $\Ycal$, such that $B(\mathcal Y_1, \Ycal_2)=0$. Then
\begin{equation}\label{eq:kernelsum}
\Ker (B\mid \Ycal_1+\Ycal_2)=\Ker (B\mid \Ycal_1) +\Ker(B\mid \Ycal_2).
\end{equation}
\item Let $\mathcal Y$ be a positive subspace for $B$. Suppose $u\in\mathcal Y$ satisfies $B(u,u)=0$. Then $u\in \mathrm{Ker} (B\mid \mathcal Y)$. 
\end{enumerate}
\end{lemma}
We say $\mathcal Y$ is a positive subspace for $B$ if for all $y\in\mathcal Y$, $B(y,y)\geq 0$. Note that the $B$-orthogonality of the subspaces is crucial in parts~(i) and~(ii). 
\begin{proof} (i) Suppose the statement is false, then there exists $y\in\mathcal Y$, so that $B(y,y)<0$. Clearly, $y\not=0$ and hence, by assumption, $y\not\in\mathcal X_-$. Now consider $\mathcal Z=\mathrm{span}\{y, \mathcal X_-\}$. Let $0\not=z\in \mathcal Z$. Then, there exist $\lambda\in\R$ and $z_-\in\mathcal X_-$, not both zero, so that  $z=\lambda y+z_-$. It follows from the $B$ orthogonality of $\mathcal X_-$ and $\mathcal Y$ that
$$
B(z,z)=\lambda^2B(y,y)+B(z_-, z_-)<0.
$$
Hence $B$ is negative definite on $\mathcal Z$. Since dim$\mathcal X_- \subsetneq \mathcal Z$ this is a contradiction. \\
(ii) Immediate.\\
(iii) One has, for all $v\in \mathcal Y$, and for all $\lambda\in\R$, 
$$
0\leq B(\lambda u+ v, \lambda u+v)=2\lambda B(u,v) +B(v,v).
$$
If $B(u,v)\not=0$, this is a contradiction. 
\end{proof}

\section{Proof of Theorem~\ref{thm:localcoercivity}}\label{s:proofmaintheorem}
Let $n_-=\dim\Ban_-$ and $\{\eta_1,...,\eta_{n_-}\}\subset \R^m$ a family of linearly independent elements of $\R^m$ such that $D^2_\xi W(\eta,\eta)>0$ for all $\eta\in \mathrm{span}\{\eta_1,...,\eta_{n_-}\}$. As a consequence of \eqref{eq:hessWhessL}, $\Xcal_-:={\rm span}\,\left\{\eta_1\cdot \nabla_\xi u_\xi,..., \eta_{n_-}\cdot \nabla_\xi u_\xi\right\}$ is a negative definite subspace for $D^2_{u_\xi}\Lcal_\xi$ in $\Ban$. Moreover, since $\dim \Xcal_-=n_-$, $\Xcal_-$ is a maximally negative definite subspace. 

Next, 
since $D^2_\xi W$ is non-degenerate by hypothesis, Lemma~\ref{lem:prophessWhessL}~\ref{eq:nondegeneracyW} implies that $\hat F$ is a local diffeomorphism. Hence, thanks to Lemma~\ref{lem:transversality}, $\Xcal_-\cap  T_{u_\xi}\Sigma_{\mu_\xi}=\{0\}$. Moreover, thanks to \eqref{eq:ortho}, $\Xcal_-$ and $T_{u_\xi}\Sigma_{\mu_\xi}$ are $D^2_{u_\xi}\Lcal_{\xi}$-orthogonal. As a consequence, we can apply Lemma~\ref{lem:Borthopositive}~(i) and conclude that $T_{u_\xi}\Sigma_{\mu_\xi}$ is a positive subspace for $D^2_{u_\xi}\Lcal_\xi$. 

Furthermore, since by hypothesis~(ii) of the theorem, $E_0=\Ker D^2_{u_\xi}\Lcal_\xi= T_{u_\xi}\Ocal_{u_\xi}$, it follows from Lemma~\ref{lem:Borthopositive}~(ii)-(iii) that
\begin{equation}\label{eq:Yspacesimple}
\Ycal:=T_{u_\xi}\Sigma_{\mu_\xi} \cap \left(T_{u_\xi}\Ocal_{u_\xi}\right)^{\perp}
\end{equation} 
is a positive definite subspace for $D^2_{u_\xi}\Lcal_\xi$, meaning that
\begin{equation}\label{eq:strictpos}
D^2_{u_\xi}\Lcal_\xi(v,v)>0, \quad \forall v\in \mathcal Y= T_{u_\xi}\Sigma_{\mu_\xi} \cap \left(T_{u_\xi}\Ocal_{u_\xi}\right)^{\perp}, v\neq 0.
\end{equation}
Here, and in the rest of the proof, the $\perp$ denotes orthogonality with respect to the inner product $\scalar_E$.

To obtain the desired coercive bound, we now use the spectral information on $\nabla^2\Lcal_{\xi}(u_{\xi})$ provided by the hypotheses of the theorem. Note first that, since $\nabla^2\Lcal_{\xi}(u_\xi)$ is self-adjoint, $E_0^{\perp}=\left(T_{u_\xi}\Ocal_{u_\xi}\right)^{\perp}\subset \Ban$ is an invariant subspace for $\nabla^2\Lcal_{\xi}(u_\xi)$: it is the spectral subspace of $\nabla^2\Lcal_{\xi}(u_\xi)$ corresponding to $\R\setminus\{0\}$. 

Let $\nabla F_j(u_\xi)\in\Ban$ be defined by $D_{u_\xi} F_j(v)=\scalbb{\nabla F_j(u_\xi)}{v}$ for $j=1,...,m$. Then, since $\mu_\xi$ is a regular value,
$$
\left(T_{u_\xi}\Sigma_{\mu_\xi}\right)^{\perp}={\rm span}\,\{\nabla F_j(u_\xi)\}_{j=1,...,m}.
$$
and $\dim \left(T_{u_\xi}\Sigma_{\mu_\xi}\right)^{\perp}=m$. Moreover, since $T_{u_\xi}\Ocal_{u_\xi}\subset T_{u_\xi}\Sigma_{\mu_\xi}$, one obtains the following orthogonal decomposition of $E_0^\perp$: 
$$
E_0^{\perp}=\left(T_{u_\xi}\Ocal_{u_\xi}\right)^{\perp}=\Ycal\oplus \left(T_{u_\xi}\Sigma_{\mu_\xi}\right)^{\perp}.
$$
Note that $\nabla^2\Lcal_{\xi}(u_{\xi})$ does not leave this decomposition invariant: we are interested in controlling it on $\mathcal Y$. 

For that purpose, let $P$ be the projection on $\Ycal\subset E_0^\perp$ and consider the following decomposition of the restriction of the operator $\nabla^2\Lcal_{\xi}(u_\xi)$ to $E_0^{\perp}$,
\begin{align*}
\nabla^2\Lcal_{\xi}(u_\xi)_{\mid E_0^{\perp}}=& P\nabla^2\Lcal_{\xi}(u_\xi)P+P\nabla^2\Lcal_{\xi}(u_\xi)(I_{E_0^\perp}-P)\\
&+(I_{E_0^\perp}-P)\nabla^2\Lcal_{\xi}(u_\xi)P+(I_{E_0^\perp}-P)\nabla^2\Lcal_{\xi}(u_\xi)(I_{E_0^\perp}-P).
\end{align*}
Since $\left(T_{u_\xi}\Sigma_{\mu_\xi}\right)^{\perp}$ is finite dimensional, the projector $I_{E_0^\perp}-P$ is finite rank. Hence $\nabla^2\Lcal_{\xi}(u_\xi)_{\mid E_0^{\perp}}= P\nabla^2\Lcal_{\xi}(u_\xi)P+K$ with $K$ a compact operator and it follows that $\sigma_{\rm ess}(P\nabla^2\Lcal_{\xi}(u_\xi)P)=\sigma_{\rm ess}\left(\nabla^2\Lcal_{\xi}(u_\xi)_{\mid E_0^{\perp}}\right)$. Here $\sigma_{\rm ess}(\cdot)$ designates the essential spectrum. 
In particular, $0\not\in\sigma_{\rm ess}(P\nabla^2\Lcal_{\xi}(u_\xi)P)$ by hypotheses $(iii)$ and $ (iv)$ of Theorem~\ref{thm:localcoercivity}.

Now, note that Ker$P\nabla^2\Lcal_{\xi}(u_\xi)P=\left(T_{u_\xi}\Sigma_{\mu_\xi}\right)^{\perp}\subset E_0^\perp$. Indeed, let $u\in E_0^\perp$ be such that $P\nabla^2\Lcal_{\xi}(u_\xi)Pu=0$. Then
$$
D^2_{u_\xi}\mathcal L(u_\xi) (Pu, Pu)=\langle u, P\nabla^2\Lcal_{\xi}(u_\xi)Pu\rangle_\Ban=0.
$$
Since $D^2_{u_\xi}\mathcal L(u_\xi)$ is strictly positive on $\mathcal Y$ (see~\eqref{eq:strictpos}), it follows that $Pu=0$, which means that $u\in \left(T_{u_\xi}\Sigma_{\mu_\xi}\right)^{\perp}$. 

We now consider $P\nabla^2\Lcal_{\xi}(u_\xi)P$ as an operator on $\mathcal Y$. We have just shown that $0$ is not an eigenvalue of $P\nabla^2\Lcal_{\xi}(u_\xi)P$, nor in its essential spectrum. It is therefore not in its spectrum. In addition, we showed $P\nabla^2\Lcal_{\xi}(u_\xi)P$ is a strictly positive operator on $\mathcal Y$ (see~\eqref{eq:strictpos}). It therefore has a spectral gap:
$$
\delta=\inf\limits_{v\in\Ycal\smallsetminus \{0\}}\frac{\scal{P\nabla^2\Lcal_{\xi}(u_\xi)Pv}{v}_\Ban}{\scal{v}{v}_\Ban}>0.
$$ 
Finally, for all $v\in\mathcal Y$, $v\not=0$, one finds
\[
D^2_{u_\xi}\Lcal_\xi(v,v)=\scal{\nabla^2\Lcal_{\xi}(u_\xi)v}{v}_\Ban=\scal{P\nabla^2\Lcal_{\xi}(u_\xi)Pv}{v}_\Ban\ge \delta \|v\|_\Ban^2
\]
which is the desired estimate.


\section{Main result: a more general setting}\label{s:generalcase}

In this section, we extend Theorem~\ref{thm:localcoercivity} to a more general setting that we now describe.
In order to state our main result, we first need to associate to $D^2_{u_\xi}\Lcal_\xi$ a (typically unbounded) self-adjoint operator on $\hat \Ban$. 
\begin{lemma}\label{lem:existenceHessian} 
Suppose Hypothesis~A holds. Let $\xi\in \frak g$ and $\calL_{\xi}$ as in~\eqref{eq:lyapfunction} and let $u\in\Ban$. If there exist $\varepsilon, C>0$ such that  
		\begin{equation}
			\label{eq:condGgen}
			D^2_{u}\calL_{\xi}(v,v)\ge \varepsilon \|v\|^2_{\Ban}-C\|v\|^2
		\end{equation}
		for all $v\in \Ban$, then there is a self-adjoint operator $\nabla^2\calL_\xi (u):\Dcal(\nabla^2\mathcal L_\xi(u))\subset \hat\Ban\to \hat\Ban$ defined by 
		\begin{equation}
			\label{eq:defHessiangen}
			\begin{aligned}
			&\Dcal(\nabla^2\mathcal L_\xi(u))=\{z\in \Ban\mid \exists w\in \hat\Ban\text{ such that } D^2_{u}\calL_{\xi}(z,v)=\scal{w}{v}\text{ for all }v\in \Ban\},\\
			&\nabla^2\calL_\xi (u)z=w\text{ for all } z\in \Dcal(\nabla^2\mathcal L_\xi(u)) . 
			\end{aligned}
		\end{equation}
		Moreover, $\DomH$ is a form core for $D^2_{u}\calL_{\xi}$.
\end{lemma}

\begin{remark} Note that 
\begin{enumerate}
\item $\Ban$ is the form domain of the operator $\nabla^2\calL_\xi (u)$,
\item Since $\Dcal(\nabla^2\mathcal L_\xi(u))$ is a form core for $D^2_{u}\calL_{\xi}$ and condition \eqref{eq:condGgen} holds,  $\Dcal(\nabla^2\mathcal L_\xi(u))$ is dense in $\Ban$ (see \cite[Chapter VI]{kato1980} for details). 
\end{enumerate}
\end{remark}
The existence and the uniqueness of the operator $\nabla^2\calL_\xi (u)$ is a consequence of the First Representation theorem in Kato \cite[Theorem 2.1 and 2.6 in Chapter VI]{kato1980}. Condition~\eqref{eq:condGgen}  ensures that the hypotheses of the First Representation theorem are satisfied (see \cite[Lemma 3.3]{stuart08}). See also \cite{reesim1980,teschl2009} for more details.

We can now state our main result. We define $p(D^2_\xi W)$, $n(D^2_\xi W)$, $W$, and $n(D^2_{u_\xi}\Lcal_\xi)$ as in Section~\ref{s:maintheoremhil}.
\begin{theorem}
	\label{thm:localcoercivitygen}
	Suppose Hypotheses~A,B,C hold. Let $\xi\in\Omega$ and suppose 
	\begin{equation}\label{eq:condGgen2}
	D^2_{u_\xi}\calL_{\xi}(v,v)\ge \varepsilon \|v\|^2_{\Ban}-C\|v\|^2, \forall v\in\Ban.
	\end{equation}
	Suppose in addition that
	\begin{enumerate}[label=(\roman*)]
		\item $D^2_\xi W$ is non-degenerate,
		\item $\Ker D^2_{u_\xi}\calL_\xi = T_{u_\xi}\Ocal_{u_\xi}$, \label{hyp:kernel}
		\item $\inf (\sigma(\nabla^2\calL_\xi (u_\xi))\cap (0,+\infty))>0$,
		\item $p(D^2_\xi W)=n(D^2_{u_\xi}\Lcal_\xi)$,
		\item for all $u\in \Ocal_{u_\xi}$ and for all $j=1,...,m$, there exists $\nabla F_j (u)\in \DomH\subset \Ban$ such that
\begin{equation}
	\label{eq:defgradient}
	D_{u} F_j(w)=\scal{\nabla F_j (u)}{w},\quad \forall w\in \Ban.
\end{equation}
	\end{enumerate}
	Then there exists $\delta>0$ such that
	\begin{equation}
		\label{eq:localcoercivitygen}
		\forall v\in T_{u_\xi}\Sigma_{\mu_\xi}\cap \left(T_{u_\xi}\Ocal_{u_\xi}\right)^{\perp},\ D^2_{u_\xi}\calL_{\xi}(v,v)\ge \delta \| v\|^2_\Ban
	\end{equation}
	with $ \left(T_{u_\xi}\Ocal_{u_\xi}\right)^{\perp}=\{v\in \hat\Ban\,|\,\scal{v}{w}=0\  \forall w\in T_{u_\xi}\Ocal_{u_\xi}\}$.
\end{theorem}
Note that here, and in the rest of this section, the orthogonality is with respect to the inner product $\langle\cdot,\cdot\rangle$.

Let us point out that the hypotheses on the bilinear form $D^2_{u_\xi}\calL_\xi $ in Theorem \ref{thm:localcoercivitygen} can be re-expressed in terms of spectral hypotheses on the (unbounded) self-adjoint operator $\nabla^2\calL_\xi (u_\xi)$, as shown in the following lemma. This is important in applications, since it allows one to use the tools of spectral analysis for partial differential operators to establish those conditions. 

\begin{lemma}\label{lem:morseindex} Under the hypotheses of Lemma \ref{lem:existenceHessian}, with $u=u_\xi$, $\Ker \nabla^2\calL_\xi (u_\xi)=\Ker D^2_{u_\xi}\calL_{\xi}$. 

If, in addition $\dim \Ker \nabla^2\calL_\xi (u_\xi)<+\infty$, the negative spectral subspace of $\nabla^2\calL_\xi (u_\xi)$ is finite dimensional, and hypothesis~(iii) of Theorem~\ref{thm:localcoercivitygen} is satisfied, 
then the dimension of the negative spectral subspace of $ \nabla^2\calL_\xi (u_\xi)$ in $\hat \Ban$ is equal to the Morse index $n(D^2_{u_\xi}\Lcal_\xi)$ of $u_\xi\in \Ban$ for $\Lcal_\xi$.
\end{lemma}

This lemma constitutes a slight generalization of   Lemma 5.4 in \cite{stuart08} and its proof follows along the same lines. We give it for completeness.

\begin{proof}
By definition  $\Ker \nabla^2\calL_\xi (u_\xi)=\{v\in \DomH \mid \nabla^2\calL_\xi (u_\xi)v=0\}$ and a straightforward calculation leads to $\Ker \nabla^2\calL_\xi (u_\xi)=\Ker D^2_{u_\xi}\calL_{\xi} \cap \DomH$. Moreover, using the definition \eqref{eq:defHessiangen}, it easy to see that $\Ker D^2_{u_\xi}\calL_{\xi} \subset \DomH$. As a consequence, $\Ker \nabla^2\calL_\xi (u_\xi)=\Ker D^2_{u_\xi}\calL_{\xi}$.

Now, we know that 
\begin{align*}
\dim \Ker \nabla^2\calL_\xi (u_\xi)=n_0<+\infty,\qquad\inf (\sigma(\nabla^2\calL_\xi (u_\xi))\cap (0,+\infty))>0 
\end{align*}
and we denote by $0\leq n_-<+\infty$ the dimension of the negative spectral subspace of $ \nabla^2\calL_\xi (u_\xi)$ in $\hat \Ban$. It follows that there exists $\Gamma>0$ such that $\sigma(\nabla^2\calL_\xi (u_\xi))\cap (0,\Gamma]=\emptyset$. Let $P_{(-\infty, 0]}:\hat E\to\hat E$
denote the orthogonal projection onto the finite dimensional span of all the eigenvectors of $\nabla^2\calL_\xi (u_\xi)$ corresponding to the eigenvalues in $(-\infty,0]$, and let $Q=I-P_{(-\infty, 0]}$. We have that $P_{(-\infty, 0]}(\hat \Ban)\subset\DomH$ and that $Qz\in \DomH$ if and only if $z\in \DomH$. Also $\dim P_{(-\infty, 0]}(\hat\Ban)=n_0+n_-$ and $\scal{\nabla^2\Lcal_{\xi}(u_\xi)Qz}{Qz}\ge \Gamma \|Qz\|^2$ for all $z\in \DomH$. 
Thus, for any $z\in \DomH$
\begin{align*}
	D^2_{u_\xi}\calL_{\xi} (z,z)=&\,\scal{\nabla^2\Lcal_{\xi}(u_\xi)z}{z}=\scal{Q\nabla^2\Lcal_{\xi}(u_\xi)z}{z}+\scal{P_{(-\infty, 0]}\nabla^2\Lcal_{\xi}(u_\xi)z}{z}\\
	=&\,\scal{\nabla^2\Lcal_{\xi}(u_\xi)Qz}{Qz}+\scal{P_{(-\infty, 0]}\nabla^2\Lcal_{\xi}(u_\xi)z}{z}\\
	\ge&\, \Gamma \|Qz\|^2+\scal{P_{(-\infty, 0]}\nabla^2\Lcal_{\xi}(u_\xi)z}{z}\\
	\geq&\, \Gamma \|z\|^2- \Gamma \|P_{(-\infty, 0]}z\|^2+\scal{\nabla^2\Lcal_{\xi}(u_\xi)P_{(-\infty, 0]}z}{z}\\
	\geq&\,\Gamma \|z\|^2+\scal{(\nabla^2\Lcal_{\xi}(u_\xi)-\Gamma I)P_{(-\infty, 0]}z}{z}.
\end{align*}
As a consequence, using \eqref{eq:condGgen2} and the fact that $\DomH$ is dense in $\Ban$, we obtain 
\begin{equation*}
	\left(1+\frac{\Gamma}{C}\right)D^2_{u_\xi}\calL_{\xi} (z,z)-\scal{(\nabla^2\Lcal_{\xi}(u_\xi)-\Gamma I)P_{(-\infty, 0]}z}{z}\ge \frac{\varepsilon\Gamma}{C} \|z\|_{\Ban}^2
\end{equation*}
for all $z\in\Ban$, which implies
\begin{equation*}
	D^2_{u_\xi}\calL_{\xi} (z,z)-\frac{C}{\Gamma+C}\scal{(\nabla^2\Lcal_{\xi}(u_\xi)-\Gamma I)P_{(-\infty, 0]}z}{z}\ge \frac{\varepsilon\Gamma}{\Gamma+C} \|z\|_{\Ban}^2
\end{equation*}
for all $z\in\Ban$. 
Moreover, for all $z\in Q(\Ban)$,
\begin{equation}\label{eq:Qpos}
	D^2_{u_\xi}\calL_{\xi} (z,z)\ge \frac{\varepsilon\Gamma}{\Gamma+C} \|z\|_{\Ban}^2
\end{equation}
since $P_{(-\infty, 0]}z=0$. But $Q(\Ban)\subset \Ban$ since $P_{(-\infty, 0]}(\Ban)\subset P_{(-\infty, 0]}(\hat \Ban)\subset \DomH\subset \Ban$. 
So we have shown that $Q(E)$ is a positive subspace of $E$, for $D^2_{u_\xi}\mathcal L$. Now consider the direct sum decomposition of $E$ given by
$$
E=Q(E)\oplus P_{(-\infty, 0)}(E)\oplus \mathrm{Ker}(D^2\mathcal L_\xi).
$$
Here $P_{(-\infty, 0)}$ is the projector onto the $n_-$-dimensional space spanned by the eigenvectors of $\nabla^2\mathcal L_\xi(u_\xi)$ with strictly negative eigenvalue. Clearly, $P_{(-\infty, 0)}(E)$ is a negative definite space for $D^2_{u_\xi}\mathcal L_\xi$. We now show it is maximal. For that purpose, suppose $z_*\in E, z_*\not\in P_{(-\infty, 0)}(E)$ and suppose span$\{z_*, P_{(-\infty, 0)}(E)\}$ is a negative definite subspace of $E$ for $D^2_{u_\xi}\mathcal L_\xi$ of dimension $n_-+1$. We can suppose, without loss of generality, that $z_*\in \mathrm{Ker}(D^2\mathcal L_\xi)\oplus Q(E)$.  Writing $z_*=z_{*,0}+z_{*, 1}$ with $z_{*,0}\in \mathrm{Ker}(D^2\mathcal L_\xi)$ and $z_{*, 1}\in Q(E)$, we see
$$
D^2_{u_\xi}\mathcal L_\xi(z_*, z_*)=D^2_{u_\xi}\mathcal L_\xi(z_{*,1}, z_{*,1})\geq 0,
$$
where we used~\eqref{eq:Qpos}.
This contradicts the fact that span$\{z_*, P_{(-\infty, 0)}(E)\}$ is negative definite subspace for $D^2_{u_\xi}\mathcal L_\xi$ and shows that $P_{(-\infty, 0)}(E)$ is a maximally negative definite subspace for $D^2_{u_\xi}\mathcal L_\xi$. Thus $n(D^2_{u_\xi}\Lcal_\xi)=n_-$.\\
\end{proof}

For the proof of Theorem \ref{thm:localcoercivitygen}, we will need the following two lemmas.

\begin{lemma}\label{lem:densityortho} Let $(\mathcal H,\|\cdot\|)$ be a Hilbert space and $\mathcal M\subset \mathcal K\subset \mathcal H$ with $\mathcal M$ a closed subspace of $\mathcal H$ and $\mathcal K$ a dense subspace of $\mathcal H$ ($\overline{ \mathcal K}^{\|\cdot\|}=\mathcal H$). Then
\begin{equation}\label{eq:densityortho}
\overline{ \mathcal M^{\perp}\cap\mathcal K}^{\|\cdot\|}= \mathcal M^{\perp}
\end{equation}
\end{lemma}

\begin{proof} Let $u\in \mathcal M^\perp$. There exists a sequence $k_n\in \mathcal K$ so that $k_n\to u$. Since $\mathcal M$ is closed, we can write
$k_n=w_n + v_n$, with $w_n\in \mathcal M, v_n\in \mathcal M^\perp$. Moreover, since $k_n\in \mathcal K$ and $w_n\in \mathcal M\subset \mathcal K$, $v_n\in \mathcal M^\perp\cap \mathcal K$.
Clearly both sequences $w_n$ and $v_n$ converge, respectively to $w\in \mathcal M, v\in \mathcal M^\perp$. Since $u=w+v\in \mathcal M^\perp$, we find $w=0$ and $v=u$. Hence $v_n\in \mathcal K \cap \mathcal M^\perp$ converges to $u$. 
\end{proof}
We introduce, as before $E_0=\Ker \nabla^2_{u_\xi}\mathcal L_\xi$. We know from Lemma~\ref{lem:morseindex} that $\Ker \nabla^2_{u_\xi}\mathcal L_\xi=\Ker D^2_{u_\xi}\mathcal L_\xi$. Hypothesis~(ii) of the theorem then implies
\begin{equation}
E_0=T_{u_\xi}\Ocal_{u_\xi}\subset \DomH.
\end{equation}
We define furthermore
\begin{equation}\label{eq:Vdef}
V_{u_\xi}=\mathrm{span}\{\nabla F_j(u_\xi), {j=1,...m}\}.
\end{equation}
Note that, by hypothesis $(v)$ of Theorem~\ref{thm:localcoercivitygen}, $ V_{u_\xi}\subset\DomH\subset E$. 
\begin{lemma}\label{lem:propdom} Under the hypotheses of Theorem~\ref{thm:localcoercivitygen}, we have\\
(i) $E_0^\perp \cap \overline{T_{u_\xi}\Sigma_{\mu_\xi}}^{\|\cdot\|}=\overline{E_0^\perp \cap T_{u_\xi}\Sigma_{\mu_\xi} }^{\|\cdot\|}.$\\
(ii) $\overline{T_{u_\xi}\Sigma_{\mu_\xi}\cap \DomH}^{\|\cdot\|}=\overline{T_{u_\xi}\Sigma_{\mu_\xi}}^{\|\cdot\|}.$\\
(iii) Define
\begin{equation}\label{eq:Yspace}
\Ycal= T_{u_\xi}\Sigma_{\mu_\xi}\cap \DomH\cap E_0^{\perp}.
\end{equation}
Then 
\begin{equation}\label{eq:Yhatspace}
\hat\Ycal:=\overline\Ycal^{\|\cdot\|} = E_0^\perp\cap V_{u_\xi}^\perp,
\end{equation}
where $V_{u_\xi}$ is defined in~\eqref{eq:Vdef}. Hence
\begin{equation}
E_0^\perp=T_{u_\xi}\Ocal_{u_\xi}^\perp = \hat{\mathcal Y} \oplus_\perp V_{u_\xi},\quad \hat E=E_0\oplus_\perp \hat{\mathcal Y} \oplus_\perp V_{u_\xi}.
\end{equation}
(iv) Let $P$ be the orthogonal projector onto $\hat{\mathcal Y}$. Let $u\in\DomH\cap E_0^\perp.$ Then $Pu\in\mathcal Y.$
\end{lemma}

We use the notation $\oplus_\perp$ to indicate a direct sum that is orthogonal for the inner product $\langle\cdot, \cdot\rangle$. 
\begin{proof} (i) Note that $E_0\subset T_{u_\xi}\Sigma_{\mu_\xi}\subset\overline{T_{u_\xi}\Sigma_{\mu_\xi}}^{\|\cdot\|}$. We now apply Lemma~\ref{lem:densityortho} with $\mathcal M=E_0, \mathcal K= T_{u_\xi}\Sigma_{\mu_\xi}, \mathcal H = \overline{T_{u_\xi}\Sigma_{\mu_\xi}}^{\|\cdot\|}$ to obtain
\begin{equation*}
E_0^{\perp_\sigma}=\overline{E_0^{\perp_\sigma}\cap T_{u_\xi}\Sigma_{\mu_\xi}}^{\|\cdot\|}.
\end{equation*}
 Here we wrote $E_0^{\perp_\sigma}$ for the orthogonal complement to $E_0$ in $\overline{T_{u_\xi}\Sigma_{\mu_\xi}}^{\|\cdot\|}$, \emph{i.e.}
\begin{equation}\label{eq:Ezeroperpsigma}
E_0^{\perp_\sigma}=E_0^\perp\cap \overline{T_{u_\xi}\Sigma_{\mu_\xi}}^{\|\cdot\|}.
\end{equation}
The last two equations imply the result. \\
(ii) Since $V_{u_\xi}$ is a closed finite dimensional subspace of $\hat \Ban$, we have 
$
\hat \Ban=V_{u_\xi}^\perp\oplus_\perp V_{u_\xi}
$
with 
$$
V_{u_\xi}^\perp=\{w\in \hat\Ban\,|\,\scal{v}{w}=0\  \forall v\in V_{u_\xi}\}. 
$$
Since by hypothesis~(v) of Theorem~\ref{thm:localcoercivitygen}, for all $w\in E$,  $D_{u_\xi} F_j(w)=\scal{\nabla F_j(u_\xi)}{w}$ for $j=1,...,m$, we see that
\begin{equation}\label{eq:tusigma}
T_{u_\xi}\Sigma_{\mu_\xi}=V_{u_\xi}^\perp \cap \Ban
\end{equation}
and hence
\begin{equation}\label{eq:Esum}
E=T_{u_\xi}\Sigma_{\mu_\xi}\oplus_\perp V_{u_\xi}.
\end{equation}

Using Lemma~\ref{lem:densityortho}, and the fact that $V_{u_\xi}\subset E\subset \hat E$, with $E$ dense in $\hat E$, \eqref{eq:tusigma} implies $\overline{T_{u_\xi}\Sigma_{\mu_\xi}}^{\|\cdot\|}=\overline{V_{u_\xi}^\perp\cap \Ban}^{\|\cdot\|}=V_{u_\xi}^\perp$ meaning that 
\begin{equation}\label{eq:hatEsum}
\hat \Ban= \overline{T_{u_\xi}\Sigma_{\mu_\xi}}^{\|\cdot\|} \oplus_\perp V_{u_\xi}.
\end{equation}

From~\eqref{eq:Esum} one concludes 
$$
\DomH = \left(\DomH\cap T_{u_\xi}\Sigma_{\mu_\xi}\right)\oplus_\perp V_{u_\xi},
$$
and hence
$$
\hat E = \overline{\DomH\cap T_{u_\xi}\Sigma_{\mu_\xi}}^{\|\cdot\|}\oplus_\perp V_{u_\xi}
$$
Comparing this to~\eqref{eq:hatEsum}, one concludes
\begin{equation}
\overline{T_{u_\xi}\Sigma_{\mu_\xi}}^{\|\cdot\|} =\overline{\DomH\cap T_{u_\xi}\Sigma_{\mu_\xi}}^{\|\cdot\|}.
\end{equation}
In other words, $\DomH\cap T_{u_\xi}\Sigma_{\mu_\xi}$ is dense in $\overline{T_{u_\xi}\Sigma_{\mu_\xi}}^{\|\cdot\|}.$ This proves~(ii).

(iii)  Note that $E_0\subset \DomH\cap T_{u_\xi}\Sigma_{\mu_\xi}\subset \overline{T_{u_\xi}\Sigma_{\mu_\xi}}^{\|\cdot\|}.$ Then we can, in view of part (ii) of the Lemma,  apply Lemma~\ref{lem:densityortho} with $\mathcal M=E_0, \mathcal K=\DomH\cap T_{u_\xi}\Sigma_{\mu_\xi}$, and $\mathcal H=\overline{T_{u_\xi}\Sigma_{\mu_\xi}}^{\|\cdot\|}$ to obtain:
$$
E_0^{\perp_\sigma}=\overline{E_0^{\perp_\sigma}\cap \DomH\cap T_{u_\xi}\Sigma_{\mu_\xi}}^{\|\cdot\|}=\overline{E_0^{\perp}\cap \DomH\cap T_{u_\xi}\Sigma_{\mu_\xi}}^{\|\cdot\|}=\hat{\mathcal Y},
$$
where $E_0^{\perp_\sigma}$ is defined in~\eqref{eq:Ezeroperpsigma}. Since 
$$
E_0^\perp = E_0^{\perp_\sigma}\oplus_\perp V_{u_\xi},
$$
part (iii) follows.\\
(iv) Let $u\in \DomH\cap E_0^\perp$. Since $E_0^\perp=\hat{\mathcal Y}\oplus_\perp V_{u_\xi}$, $u=Pu+v$ with $Pu\in \hat{\mathcal Y}$ and $v\in V_{u_\xi}$. Moreover, since $u\in \DomH$ and $V_{u_\xi}\subset \DomH$, it follows $Pu\in \DomH\cap \hat{\mathcal Y}$. To conclude, we observe that 
$$
\DomH\cap \hat{\mathcal Y}=\mathcal Y.
$$
Indeed, using \eqref{eq:Yhatspace} and \eqref{eq:tusigma},
\begin{align*}
\DomH\cap \hat{\mathcal Y}=&\,\DomH\cap E_0^\perp \cap V_{u_\xi}^\perp= \DomH \cap E_0^\perp \cap V_{u_\xi}^\perp\cap \Ban\\
=&\, \DomH \cap E_0^\perp \cap T_{u_\xi}\Sigma_{\mu_\xi}=\Ycal.
\end{align*}
Finally, $Pu\in \Ycal$.\\
\end{proof}

We can now proceed with the proof of Theorem \ref{thm:localcoercivitygen}.

\begin{proof}[Proof of Theorem \ref{thm:localcoercivitygen}] Let $n_-=n(D^2_{u_\xi}\Lcal_\xi)$. As in the proof of Theorem \ref{thm:localcoercivity}, $\Xcal_-={\rm span}\,\left\{\eta_1\cdot \nabla_\xi u_\xi,..., \eta_{n_-}\cdot \nabla_\xi u_\xi\right\}$ is a maximally negative definite subspace for $D^2_{u_\xi}\Lcal_\xi$, $\Xcal_-\cap T_{u_\xi}\Sigma_{\mu_\xi}=\{0\}$, and $T_{u_\xi}\Sigma_{\mu_\xi}$ is a positive subspace for $D^2_{u_\xi}\Lcal_\xi$. 
Furthermore, note that by hypothesis $\Ker D^2_{u_\xi}\Lcal_\xi= T_{u_\xi}\Ocal_{u_\xi}\subset T_{u_\xi}\Sigma_{\mu_\xi}$.  Hence, by Lemma~\ref{lem:Borthopositive}~(ii), $T_{u_\xi}\Sigma_{\mu_\xi} \cap \left(T_{u_\xi}\Ocal_{u_\xi}\right)^{\perp}$ is a positive definite subspace for $D^2_{u_\xi}\Lcal_\xi$, meaning 
\begin{equation}\label{eq:strctlyposop}
D^2_{u_\xi}\Lcal_\xi(v,v)>0, \quad \forall v\in T_{u_\xi}\Sigma_{\mu_\xi} \cap \left(T_{u_\xi}\Ocal_{u_\xi}\right)^{\perp}, v\neq 0.
\end{equation}
We note for further reference that, by Lemma~\ref{lem:morseindex}, $\Ker D^2_{u_\xi}\Lcal_\xi=\Ker \nabla^2\Lcal_\xi(u_\xi)$, so that $T_{u_\xi}\Ocal_{u_\xi}\subset \DomH$. 

Recall that, since $E_0$ is the kernel of $\nabla^2\Lcal_{\xi}(u_\xi)$, $E_0^\perp$ is the spectral space associated to $\R^*$, and hence invariant under $\nabla^2\Lcal_{\xi}(u_\xi)$. Now, let $P$ be the projection on $\hat\Ycal$ and consider the following decomposition of the operator $\nabla^2\Lcal_{\xi}(u_\xi)$ on $E_0^\perp=\hat{\mathcal Y} \oplus_\perp V_{u_\xi}$,
\begin{align*}
\nabla^2\Lcal_{\xi}(u_\xi)_{\mid E_0^\perp}=& P\nabla^2\Lcal_{\xi}(u_\xi)P+P\nabla^2\Lcal_{\xi}(u_\xi)(I_{E_0^\perp}-P)\\
&+(I_{E_0^\perp}-P)\nabla^2\Lcal_{\xi}(u_\xi)P+(I_{E_0^\perp}-P)\nabla^2\Lcal_{\xi}(u_\xi)(I_{E_0^\perp}-P).
\end{align*}
We claim that $P\nabla^2\Lcal_{\xi}(u_\xi)P$ is a self-adjoint operator on $E_0^\perp$ with domain $E_0^\perp\cap \DomH$. Indeed, since $V_{u_\xi}$ is finite dimensional, we can easily show that $P\nabla^2\Lcal_{\xi}(u_\xi)(I_{E_0^\perp}-P)+(I_{E_0^\perp}-P)\nabla^2\Lcal_{\xi}(u_\xi)P$ and $(I_{E_0^\perp}-P)\nabla^2\Lcal_{\xi}(u_\xi)(I_{E_0^\perp}-P)$ are bounded self-adjoint operators on $E_0^\perp$. Hence $P\nabla^2\Lcal_{\xi}(u_\xi)P$ is the sum of a self-adjoint operator with domain $E_0^\perp\cap \DomH$ and a bounded operator on $E_0^\perp$ and, by the Kato-Rellich theorem, it is self-adjoint on $E_0^\perp$ with domain $E_0^\perp\cap \DomH$. In particular,
 $\nabla^2\Lcal_{\xi}(u_\xi)= P\nabla^2\Lcal_{\xi}(u_\xi)P+K$ with $K$ a finite rank operator and $\sigma_{\rm ess}(P\nabla^2\Lcal_{\xi}(u_\xi)P)=\sigma_{\rm ess}(\nabla^2\Lcal_{\xi}(u_\xi))$. As a consequence, $0\notin \sigma_{\rm ess}(P\nabla^2\Lcal_{\xi}(u_\xi)P)$ by hypotheses $(iii)$ and $(iv)$ of Theorem~\ref{thm:localcoercivitygen}. 
 
Now note that $\ker P\nabla^2\Lcal_{\xi}(u_\xi)P=V_{u_\xi}\subset E_0^\perp\cap  \DomH$. Indeed, let $u\in E_0^\perp\cap \DomH$ be such that $P\nabla^2\Lcal_{\xi}(u_\xi)Pu=0$. Then
$$
D^2_{u_\xi}\mathcal L_{\xi} (Pu, Pu)=\langle u, P\nabla^2\Lcal_{\xi}(u_\xi)Pu\rangle=0
$$
with $Pu \in \mathcal Y$ by Lemma~\ref{lem:propdom}. 
Since $D^2_{u_\xi}\mathcal L_{\xi}$ is strictly positive on $\mathcal Y$ (see~\eqref{eq:strctlyposop}), it follows that $Pu=0$, which means that $u\in V_{u_\xi}$.  

We now consider $P\nabla^2\Lcal_{\xi}(u_\xi)P$ on $\mathcal Y$. We have just shown that $$0\not\in \sigma(P\nabla^2\Lcal_{\xi}(u_\xi)P)=\sigma_{\mathrm{ess}}(P\nabla^2\Lcal_{\xi}(u_\xi)P)\cup\sigma_{\mathrm{d}}(P\nabla^2\Lcal_{\xi}(u_\xi)P)$$ and that $P\nabla^2\Lcal_{\xi}(u_\xi)P$ is strictly positive on $\mathcal Y$ (see~\eqref{eq:strctlyposop}).
It therefore has a spectral gap: 
$$
\tilde\delta=\inf\limits_{v\in\Ycal \smallsetminus \{0\}}\frac{\scal{P\nabla^2\Lcal_{\xi}(u_\xi)Pv}{v}}{\scal{v}{v}}>0.
$$
Next, using the inequality \eqref{eq:condGgen}, we obtain for all $v\in \Ycal$
\[
\left(1+\frac{\tilde \delta}{C}\right)D^2_{u_\xi}\Lcal_\xi(v,v)\ge \frac{\epsilon\tilde\delta}{C} \|v\|_{\Ban}^2
\]
which implies
\[
D^2_{u_\xi}\Lcal_\xi(v,v)\ge \frac{\epsilon\tilde\delta}{\tilde \delta+C} \|v\|_{\Ban}^2.
\]
Finally, the density of $\DomH$ in $\Ban$ for $\|\cdot\|_\Ban$ yields \eqref{eq:localcoercivitygen}.
\end{proof}

\section{Persistence of relative equilibria}\label{s:persistence}
In this section we come back to the question of persistence of relative equilibria, which is the question of the existence of a family of relative equilibria as in~\eqref{eq:deftildeu}-\eqref{eq:statequation}. Three situations occur. In some cases, such a family can be explicitly exhibited. In others, its existence can be proven by ad hoc methods adapted to the specific situation at hand. We will give examples of both these cases in the following section. Finally, under suitable conditions, general structural theorems asserting its existence can be proven. We give below a theorem guaranteeing the existence of a family of relative equilibria as in~\eqref{eq:deftildeu}-\eqref{eq:statequation} in the infinite dimensional framework under study here, under a natural condition on the point $\mu_*=F(u_*)$ in $\frak g^*$, which is for example always satisfied when the symmetry group $G$ is commutative and which is satisfied on an open dense subset of $\frak g^*$ in all cases. We will comment on the relation with the situation for finite dimensional systems at the end of this section. Applications will be given in the following section.

We will make the following hypothesis throughout this section:

\vskip0.2cm
\noindent{\bf Hypothesis D.} $E$ is a Hilbert space with inner product  $\langle\cdot, \cdot\rangle_E$ and $\|\cdot\|_{\Ban}=\sqrt{\langle\cdot,\cdot\rangle_E}$.
\vskip 0.2cm

\noindent With this hypothesis, one can view $\langle\cdot, \cdot\rangle_E$ as a closed form on $\hat E$ (defined in~\eqref{eq:hatban}), with form domain $E$. It follows (Theorem~VIII.15 in~\cite{reesim1980}) that there exists a unique unbounded positive operator $T^2$ on $\hat E$, with domain $\mathcal D(T^2)$, so that, for all $u,v\in \mathcal D(T^2)$,
$$
\langle u,v\rangle_E=\langle u, T^2v\rangle,
$$ 
and so that, in addition $E=\Dcal(T)$ and, for all $u,v\in \mathcal D(T)$
$$
\langle u,v\rangle_E=\langle Tu, Tv\rangle.
$$ 
Here $T$ is the positive square root of $T^2$.  Note that
$$
\langle u, T^2 u\rangle \geq \langle u, u\rangle,
$$
so that $0$ is in the resolvent set of $T^2$ and hence $T^{-2}$ is a bounded operator on $\hat E$. 

Next, we introduce in the usual manner the scale of spaces $\Ban_\lambda=\overline{\mathcal D(T^{\lambda})}^{\|\cdot\|_\lambda}$, where $\|u\|_\lambda=\|T^\lambda u\|$ and $\lambda\in \R$. In particular, we have $\Ban=\Ban_1$ and $\hat{\Ban}=\Ban_0$. 

Our persistence result then reads as follows:
\begin{theorem}\label{thm:persistenceb}
Let Hypotheses A, B  and D hold and suppose there exists $\xi_*\in\frak g$ and $u_{*}\in \mathcal D\cap \Ban_2$ so that $D_{u_*}\Lcal_{\xi_*}=0$. Suppose in addition:
\begin{enumerate}[label=(\alph*)]
\item There exist $\varepsilon, C>0$ such that, for all $v\in\Ban$,
\begin{equation}\label{eq:condGgen3}
D^2_{u_{*}}\calL_{\xi_*}(v,v)\ge \varepsilon \|v\|^2_{\Ban}-C\|v\|^2.
\end{equation}
\item For all $u\in \Ban_2$ and for all $\xi \in \frak g$ there exists $\nabla \mathcal L_{\xi}(u)\in \hat \Ban$ such that
\begin{equation}
	\label{eq:defgradientL}
	D_{u} \mathcal L_\xi(v)=\scal{\nabla \mathcal L_{\xi}(u)}{v},\quad \forall v\in \Ban.
\end{equation}
\item \label{hyp:regularity} The function $(\xi, v)\in \frak g\times E_2\to \nabla\Lcal_\xi(v)\in\hat E$ belongs to $C^1(\frak g\times E_2; \hat E)$. 
\item \label{hyp:derivative} For all $v\in E_2$, $g\in G \to \Phi_g(v)\in E$ is $C^1$.
\item \label{hyp:regularpoint} The function $F$ is regular at $u_*$.
\item \label{hyp:mstarregular} For all $\mu$ in a neighbourhood of $\mu_*=F(u_*)$, dim$\frak g_\mu=$dim$\frak g_{\mu_*}$.
\end{enumerate}
If in addition,
\begin{enumerate}[label=(\roman*)]
\item $\mathcal D(\nablasq(u_{*}))=\mathcal D(T^2)$
\item \label{hyp:kernelper}$\Ker D^2_{u_{*}}\calL_{\xi_*} = T_{u_{*}}\Ocal_{u_{*}}$, 
\item \label{hyp:gap1}$\inf (\sigma(\nabla^2\calL_{\xi_*} (u_{*}))\cap (0,+\infty))>0$,
\item \label{hyp:gap2}$n(D^2_{u_{*}}\Lcal_{\xi_*})<+\infty$,
\end{enumerate}
Then there exists a neighbourhood $\Omega$ of $\xi_*$ in $\frak g$ and a $C^1$ map $\xi\in\Omega\to u_\xi\in E$ with $u_{\xi_*}=u_*$ so that for all $\xi$, \eqref{eq:statequation} holds.  The map $\xi\to u_\xi$ is an injective immersion.
\end{theorem}
The conditions that are central here are~(ii)-(iii)-(iv): they are to be compared to the identically numbered conditions of Theorem~\ref{thm:localcoercivitygen}. The other conditions, notably (a)-(e), are technical and usually readily verified in applications. They are virtually automatic in finite dimensional problems. Condition~(f) is of purely group-theoretic nature. It  is known to hold on an open dense set for any Lie group. In fact, on such a set, the orbits all have the same maximal dimension and the Lie algebra $\frak g_\mu$ of the isotropy group of $\mu$ is commutative~\cite{dufver1969}. 
\begin{proof} Let $\hat \Space$ be the subspace of $\hat \Ban$ defined by
\begin{equation}\label{eq:defUhat}
\hat \Space=\{v\in \hat E\mid \langle v, u\rangle=0, \forall u\in T_{u_*}\mathcal O_{u_*}\}=T_{u_*}\Ocal_{u_*}^\perp
\end{equation} 
and consider the map 
\begin{align*}
\mathcal F :&\, \frak g \times (\Ban_2\cap \hat \Space) \to \hat \Space\\
&\,(\xi, w)\mapsto Q \nabla \mathcal L_{\xi}(u_*+w)
\end{align*}
where $Q$ is the orthogonal projector onto $\hat \Space\subset \hat \Ban$. Note that $(\Ban_2\cap \hat \Space, \|\cdot\|_2)$ and $(\hat \Space, \|\cdot\|)$ are Banach spaces.

By hypothesis~\ref{hyp:regularity}, $\mathcal F$ is $\mathcal C^1$. Moreover, it is clear that $\mathcal F(\xi_*,0)=0$. Hence, to apply the implicit function theorem, we remark that the derivative of $\mathcal F$ along $\Ban_2\cap \hat \Space$ at the point $(\xi_*,0)$ 
$$
\partial_w \mathcal F(\xi_*,0)= Q\nabla^2\mathcal L_{\xi_*}(u_*) 
$$
is an isomorphism from $(\Ban_2\cap \hat \Space, \|\cdot\|_2)$ to $(\hat \Space, \|\cdot\|)$. Indeed, as a result of the hypotheses~\ref{hyp:kernelper}, \ref{hyp:gap1} and~\ref{hyp:gap2}, $ Q\nabla^2\mathcal L_{\xi_*}(u_*)Q$ is a self-adjoint operator with bounded inverse.

Therefore, there exist $\mathcal V_{\xi_*}$ a neighbourhood of $\xi_*$ in $\frak g$, $\mathcal V_{u_*}$ a neighbourhood of $u_*$ in $\Ban_2\cap \hat\Space$, and a function $\Lambda: \mathcal V_{\xi_*}\to \mathcal V_{u_*}$ such that $\mathcal F(\xi, \Lambda(\xi))=0$. In particular, setting $u_\xi=u_*+\Lambda(\xi)\in \Ban_2$, we have
$$
Q \nabla \mathcal L_{\xi}(u_\xi)=0.
$$
This implies that for all $v\in \Space=\{v\in E\mid \langle v, u\rangle=0, \forall u\in T_{u_*}\mathcal O_{u_*}\}\subset \hat\Space$, 
$$
0=<v,Q \nabla \mathcal L_{\xi}(u_\xi)>=D_{u_\xi}\mathcal L_{\xi}(v).
$$
On the other hand, we know that for all $g\in G_{\mu_\xi}$
$$
\mathcal L_{\xi}(\Phi_g(u_\xi))=\mathcal L_{\xi}(u_\xi).
$$
Now, let $\eta\in \frak g_{\mu_\xi}$ and define $t\to g(t)=\exp(t\eta)\in G_{\mu_\xi}$. Then  $\mathcal L_{\xi}(\Phi_{g(t)}(u_\xi))=\mathcal L_{\xi}(u_\xi)$, so that, taking the derivative with respect to $t$ (which is possible because of hypothesis~\ref{hyp:derivative}), it follows that 
$$
D_{u_\xi}\mathcal L_{\xi}\left(X_{\eta}(u_\xi)\right)=0.
$$ 
This means that $D_{u_\xi}\mathcal L_{\xi}\left(v\right)=0$ for all $v\in T_{u_\xi}\Ocal_{u_\xi}$.

To conclude it is sufficient to prove that $\Ban=\Space\oplus T_{u_\xi}\Ocal_{u_\xi}$ for all $\xi$ in a neighbourhood of $\xi_*$. 

First of all, using the fact that  $\hat E=T_{u_*}\mathcal O_{u_*}\oplus_\perp T_{u_*}\mathcal O_{u_*}^\perp= T_{u_*}\mathcal O_{u_*}\oplus_\perp \hat \Space$, we prove $\hat \Ban= \hat \Space \oplus T_{u_\xi}\Ocal_{u_\xi}$.
Since $\mu_*=F(u_*)$ is a regular point in $\frak g^*$, one can choose, for every $\mu$ in a neighbourhood $\tilde{\mathcal M}$ of $\mu_*$, a basis  $\ell_i(\mu)\in \frak g_\mu$, for $i=1\dots m'$, smoothly in $\mu$. One can then construct
$$
\Xi: (\eta,\mu)\in \R^{m'}\times \tilde{\mathcal M}\to \Xi(y,\mu)=\sum_{i=1}^{m'} \eta_i \ell_i(\mu)\in \frak g_\mu\subset \frak g.
$$
Let $e_1, \ldots,e_{m'}$ the canonical basis of $\R^{m'}$. For all $\xi\in \frak g$ such that $u_\xi$ is sufficiently close to $u_*$, we can define, using~\eqref{eq:xxi},
\begin{equation}\label{eq:xli}
X_i(u_\xi)=\frac{d}{dt}\Phi_{\exp(\Xi(te_i, F))}(u_\xi)_{\mid t=0}= X_{\ell_i(F(u_\xi))}(u_\xi)
\end{equation}
Note that $X_i(u_\xi)$ are linearly independent and $\mathrm{span}\{X_i(u_\xi)\}=T_{u_\xi}\Ocal_{u_\xi}$. Then, $\mathrm{dim}T_{u_\xi}\Ocal_{u_\xi}=m'=\mathrm{dim}T_{u_*}\Ocal_{u_*}$. 

Next, writing $P=I-Q$ the orthogonal projector onto $T_{u_*}\Ocal_{u_*}$, we prove $P: T_{u_\xi}\Ocal_{u_\xi}\to T_{u_*}\Ocal_{u_*}$ is a bijection. Noting that the matrix of $P: T_{u_\xi}\Ocal_{u_\xi}\to T_{u_*}\Ocal_{u_*}$, given by 
$$
\langle X_i(u_*), P X_i(u_\xi)\rangle
$$
is invertible for $\xi=\xi_*$, this remains true  by continuity for $\xi$ close to $\xi_*$. 
As a conclusion, $\hat \Ban= \hat \Space \oplus T_{u_\xi}\Ocal_{u_\xi}$.

Finally, it is clear that $\Ban=\Space\oplus T_{u_\xi}\Ocal_{u_\xi}$. Since $\hat \Ban= \hat \Space \oplus T_{u_\xi}\Ocal_{u_\xi}$, we have, for each $u\in E$,
$$
u=\ell+k,
$$
with $k\in T_{u_\xi}\mathcal O_{u_\xi}\subset E, \ell\in \hat \Space$. Hence $\ell\in \Ban\cap  \hat \Space=\Space$, which concludes the argument. Note that, here as before, the orthogonality is with respect to the inner product $\langle\cdot,\cdot \rangle$. 

\end{proof}
For systems with a finite dimensional phase space, the persistence problem is addressed in~\cite{pat95, mon, lermansinger98}. The theorems provided there use various conditions on the group $G$ and on its action $\Phi$, on the nature of the isotropy groups $G_\xi$ and $G_\mu$, and finally on the Hessian of $H$ and/or $\mathcal L_\xi$ restricted to a suitable space. Note however that in applications to PDE the symplectic structure is always weak rather than strong. As a result, the finite dimensional arguments do not readily transpose to the infinite dimensional setting.  Indeed, various topological complications manifest themselves essentially as domain questions for unbounded operators and lack of differentiability of the dynamical vector field $X_H=\mathcal J^{-1}DH$ and of the vector fields generating the symmetries ($X_\xi$ in~\eqref{eq:xxi}), as we have seen. On the other hand, the infinite dimensional setting offers a simplification over the usual finite dimensional one because the phase space $E$, rather than being a manifold, is a vector space, and because the action of the invariance group is usually linear, facts that we have very much exploited in the above proof and elsewhere in this paper.

\section{Examples}  \label{s:examples}

\subsection{Stability of solitons for the nonlinear Schrödinger equation}\label{ss:nlssoliton}

We consider the focusing nonlinear Schr\"odinger equation with a power nonlinearity given by
\begin{equation}
 	\label{eqnlscubic}
 	\left\{
 	\begin{aligned}
 		&i{\partial_t}u(t,x)+{\Delta}u(t,x)+\vb u(t,x)\vb^{p-1}u(t,x)=0&\text{in } \R^d\\
 		&u(0,x)=u(x)&
 	\end{aligned}
 	\right.
\end{equation} 
with $u(t,x)\in\C$, $1<p<1+\frac{4}{d}$ and $d=1,2,3$.
This choice of parameters guarantees the global existence of solution to~\eqref{eqnlscubic} in $H^1(\R^d)$ (see \cite{cazenave2003}).

Equation \eqref{eqnlscubic} is the Hamiltonian differential equation associated to the Hamiltonian 
\begin{equation}
	\label{eqenergynls}
	H(u)=\frac{1}{2}\int_{\R^d}\vb \nabla u(x)\vb^2\diff x-\frac{1}{p+1}\int_{\R^d}\vb u(x)\vb^{p+1}\diff x.
\end{equation}

Next, let $G=\R\times \R^d$ and define its action on $E=H^1(\R^d,\C)$ via
		\begin{equation}\label{eqnlsinvariance}
		\forall u\in H^1(\R^d),\quad \left(\Phi_{\gamma_1,\gamma_2}(u)\right)(x)=e^{-i\gamma_1} u(x-\gamma_2).
	\end{equation}
Clearly, $H\circ\Phi_g=H$ and the group $G$ is an invariance group for the dynamics and the quantities 
\begin{align}
	\label{eqL2norm}
	&F_1(u)=\frac{1}{2}\int_{\R^d}\vb u(x)\vb^2\diff x\\
	\label{eqmomentum}
	&F_{1+j}(u)=\frac12\int_{\R^d}u^*\left(\frac{1}{i}\partial_{x_j}\right) u\diff x
\end{align}
for $j=1,\dots,d$, are the corresponding constants of the motion (see \cite{cazenave2003}).

The family of solitary waves
\begin{equation}
 	\label{eqsolitarywaves}
 	u_{\omega,c}(x)=e^{i \frac{c}{2}\cdot x}u_\omega(x)
\end{equation} 
with $c\in \R^d$ and $u_\omega$ the unique positive solution (see~\cite{tao_book} for more details) to 
\begin{equation}
	\label{eqsoliton}
	\Delta u_\omega+|u_{\omega}|^{p-1}u_\omega=-\omega u_{\omega}
\end{equation}
with $\omega\in \R$, $\omega<0$, are $G$-relative equilibria of \eqref{eqnlscubic}. Indeed, if we define $\mathcal L_\xi$ by 
\begin{equation}
	\label{defL}
	\mathcal L_\xi(u)=H(u)-\xi_1 F_1(u)-\sum_{j=1}^d \xi_{j+1} F_j(u),
\end{equation} 
we can easily verify that $u_{c,\omega}$ is a solution to the stationary equation $D_{u}\mathcal L_\xi=0$ with $\xi=(\omega-\frac{|c|^2}{4},c)$. In other words, for each $\xi\in \Omega=\left\{(\xi_1,\hat \xi)\in \R\times \R^d, \xi_1+\frac{|\hat \xi|^2}{4}<0\right\}$, 
\begin{equation*}
 	u_{\xi}(x)=e^{i \frac{\hat \xi}{2}\cdot x}u_{\xi_1+\frac{|\hat \xi|^2}{4}}(x)
\end{equation*} 
is a $G$-relative equilibrium of \eqref{eqnlscubic} with $\mu_\xi=F(u_\xi)$. Note that, since $G$ is commutative, $G_{\mu_\xi}=G$. Here, we use the notation $\hat \xi=(\xi_2,\ldots\xi_{d+1})$.
Note that, if $d=1$ and $p=3$, the unique positive solution of \eqref{eqsoliton} is  explicit:
\begin{equation}
	\label{sold1}
	u_\omega(x)=\sqrt{-2\omega}\,\sech(\sqrt{-\omega}x).
\end{equation}  
The $G$-orbit of the initial condition $u_{\xi}(x)$ is given by 
\begin{equation}
	\label{eqGorbit}
	\Ocal_{u_{\xi}}=\big\{e^{-i\gamma_1}u_\xi(x-\hat \gamma), (\gamma_1,\hat \gamma)\in G_{\mu_\xi}\big\}.
\end{equation}
Our goal is to investigate the orbital stability of these relative equilibria and in particular to obtain the coercivity of $\mathcal L_\xi$ by means of Theorem~\ref{thm:localcoercivitygen}. This, together with Proposition~\ref{prop:hessianestimate} and the results of~\cite{debgenrot15}, leads to the orbital stability.  Hypotheses A, B and C are easily seen to be satisfied, with $\Dcal=H^3(\R^d)$. Note in particular that, since $p>1$, $H\in C^2(E)$. Also, we use for $\langle\cdot,\cdot\rangle$ in Hypothesis~B the   usual $L^2$-scalar product, so that $\hat E=L^2(\R^d,\C)$ (viewed as a real Hilbert space).
To check the further hypotheses of Theorem~\ref{thm:localcoercivitygen}, we start by computing $D^2_{u_\xi}\mathcal L_{\xi}(v,w)$. A straightforward calculation gives
\begin{align*}
D^2_{u_\xi}\mathcal L_{\xi}(v,w)=&\,\re\left[\int_{\R^d}(-\Delta w)v^*-\int_{\R^d}|u_\xi|^{p-1} wv^*-\frac{p-1}{2}\int_{\R^d}|u_{\xi}|^{p-3}(w^*u_\xi+w u_\xi^*)(u_\xi v^*)\right.\\
&\left.-\xi_1\int_{\R^d}w v^*-\sum_{j=1}^d\xi_{j+1}\int_{\R^d}\left(\frac{1}{i}\partial_{x_j}w \right) v^*\right].
\end{align*}
Hence, writing $v(x)=e^{i\frac{c}{2}\cdot x}\tilde v(x)$ and $w(x)=e^{i\frac{c}{2}\cdot x}\tilde w(x)$, we obtain 
\begin{align*}
D^2_{u_\xi}\mathcal L_{\xi}(v,w)=&\,\re\left[\int_{\R^d}(-\Delta \tilde w) \tilde v^*- \int_{\R^d}u_\omega^{p-1}\tilde w \tilde v^*-\frac{p-1}{2}\int_{\R^d}u_{\omega}^{p-1}(\tilde w^*+\tilde w) \tilde v^*-\omega\int_{\R^d}\tilde w \tilde v^*\right]\\
=&\, \langle \mathbbm{L} \tilde w, \tilde v\rangle=\langle \nabla^2 \mathcal L_\xi(u_\xi)  w,  v\rangle
\end{align*}
with
\begin{equation}
	\label{defnabla2} 
	\mathbbm{L} \tilde w= \begin{pmatrix} -\Delta -p u_\omega^{p-1} -\omega & 0 \\
	0 & -\Delta - u_\omega^{p-1}-\omega \end{pmatrix} 
	\begin{pmatrix} \re (\tilde w)\\ \im (\tilde w) \end{pmatrix}.
\end{equation}
It then follows that the operator $\nabla^2 \mathcal L_\xi(u_\xi)$ introduced in Lemma~\ref{lem:existenceHessian} is given by  
$$
\nabla^2 \mathcal L_\xi(u_\xi)= U^*\mathbbm{L}U \quad\mathrm{with} \quad U=\begin{pmatrix} \cos\left(\frac{c}{2}\cdot x\right)& \sin\left(\frac{c}{2}\cdot x\right) \\
	-\sin\left(\frac{c}{2}\cdot x\right)  &\cos\left(\frac{c}{2}\cdot x\right) \end{pmatrix}.
$$
Clearly, the estimate~\eqref{eq:condGgen2} is satisfied. Let $L_+$ and $L_-$ be defined by 
\begin{align*}
L_+=-\Delta -p u_\omega^{p-1} -\omega,\qquad
L_-= -\Delta - u_\omega^{p-1}-\omega.
\end{align*}
Since $u_\omega$ is the unique positive solution to \eqref{eqsoliton}, using a decomposition in spherical harmonics and proceeding as in \cite[Lemma 4.1]{wei-96}, one proves that $\Ker(L_+)=\mathrm{span}\{\partial_{x_1}u_\omega, \dots,\partial_{x_n} u_\omega\}$ and $\Ker(L_-)=\mathrm{span}\{u_\omega\}$. Moreover, since $u_\omega$ is strictly positive, one concludes that $0$ is the first eigenvalue of $L_-$. Similarly, one proves $L_+$ has exactly one negative eigenvalue. 

As a consequence, $\Ker(D^2_{u_\xi}\mathcal L_{\xi})=T_{u_\xi}\Ocal_{u_\xi}$ and $n(D^2_{u_\xi}\mathcal L_{\xi})=1$.

Next, we have to show that $1=n(D^2_{u_\xi}\mathcal L_{\xi})=p(D^2_\xi W)$. Since $p(D^2_\xi W)\le n(D^2_{u_\xi}\mathcal L_{\xi})$, we already know that $D^2_\xi W$, which is a $(d+1)\times (d+1)$ matrix, has at least $d$ negative eigenvalue $\lambda_1,\ldots,\lambda_d$. Let $\lambda_0$ the remaining eigenvalue, then
$$
(-1)^d\mathrm{sign}(\lambda_0)=\mathrm{sign}(\lambda_0\lambda_1\cdots\lambda_d)=\mathrm{sign}(\det(D^2_\xi W)).
$$ 
A straightforward calculation gives 
\begin{align}
	\label{functW}
	W(\xi)& =H(u_\xi)-\xi_1 F_1(u_\xi)-\sum_{j=1}^d \xi_{j+1}F_{j+1}(u_\xi)\nonumber\\
	&=\frac{1}{2}\int_{\R^d}\vb \nabla u_\omega(x)\vb^2\diff x-\frac{1}{p+1}\int_{\R^d}\vb u_\omega(x)\vb^{p+1}\diff x -\frac{\omega}{2}\int_{\R^d}\vb u_\omega(x)\vb^2\diff x.
\end{align}
Therefore, $W(\xi)$ depends only on the single parameter $\omega$ which is itself a function of $\xi$. As a consequence, for each $k=1,\ldots,d+1$,
$$
\frac{\partial W}{\partial \xi_k}=\frac{\partial W}{\partial \omega}\frac{\partial \omega}{\partial \xi_k}=\left(-\frac{1}{2}\int_{\R^d}\vb u_\omega(x)\vb^2\diff x\right)\frac{\partial \omega}{\partial \xi_k}
$$
and, writing $f(\omega)=\left(-\frac{1}{2}\int_{\R^d}\vb u_\omega(x)\vb^2\diff x\right)$,
$$
\frac{\partial^2 W}{\partial \xi_\ell\partial \xi_k}=\frac{\partial f}{\partial \omega}\frac{\partial \omega}{\partial \xi_\ell}\frac{\partial \omega}{\partial \xi_k}+f(\omega)\frac{\partial^2 \omega}{\partial \xi_\ell\partial \xi_k}
$$
for any $\ell=1,\ldots,d+1$. Recalling $\omega(\xi)=\xi_1+\frac{|\hat \xi|^2}{4}$, this gives
\begin{equation*}
\left\{
\begin{aligned}
&\frac{\partial^2 W}{\partial \xi_1^2}=\frac{\partial f}{\partial \omega}\\
&\frac{\partial^2 W}{\partial \xi_1\partial \xi_k}=\frac{\xi_k}{2}\frac{\partial f}{\partial \omega} &\text{ for } k=2\ldots,d+1\\ 
&\frac{\partial^2 W}{\partial \xi_\ell\partial \xi_k}=\frac{\xi_\ell}{2}\frac{\xi_k}{2}\frac{\partial f}{\partial \omega}+\frac{1}{2}\delta_{k\ell}f(\omega) &\text{ for } \ell, k=2\ldots,d+1
\end{aligned}
\right.
\end{equation*}
and
\begin{equation*}
D^2_\xi W=\begin{pmatrix} \frac{\partial f}{\partial \omega}& \frac{\hat \xi}{2} \frac{\partial f}{\partial \omega}\\
 \frac{\hat \xi^T}{2} \frac{\partial f}{\partial \omega} & \frac{\hat \xi^T}{2}  \frac{\hat \xi}{2} \frac{\partial f}{\partial \omega}+\frac{1}{2}f(\omega)\bbone_{d\times d}
\end{pmatrix}.
\end{equation*}
Hence
\begin{equation*}
\mathrm{det} (D^2_\xi W)=\mathrm{det}\begin{pmatrix} \frac{\partial f}{\partial \omega}& \frac{\hat \xi}{2} \frac{\partial f}{\partial \omega}\\
0_{d\times 1} & \frac{1}{2}f(\omega)\bbone_{d\times d}
\end{pmatrix}
\end{equation*} 
so that
\begin{equation*}
\det(D^2_\xi W)=\left(\frac 12 f(\omega)\right)^d\frac{\partial f}{\partial \omega}
\text{ and } 
\mathrm{sign}(\det(D^2_\xi W))=(-1)^d\mathrm{sign}\left(\frac{\partial f}{\partial \omega}\right).
\end{equation*}

As a consequence, 
$
\mathrm{sign}\left(\frac{\partial f}{\partial \omega}\right)=\mathrm{sign}(\lambda_0).
$
This implies that $p(D^2_\xi W)=1$ if and only if $\frac{\partial f}{\partial \omega}>0$. Using the definition of $f(
\omega)$, we can conclude that $p(D^2_\xi W)=1$ if and only if
\begin{equation}
	\label{condnormschr}
	\frac{\partial }{\partial \omega}\int_{\R^d}|u_\omega|^2\diff x<0.
\end{equation}
Condition~\eqref{condnormschr} can be rewritten as
\begin{align*}
\frac{\partial }{\partial \omega}\int_{\R^d}|u_\omega|^2\diff x=2\int_{\R^d}u_\omega\frac{\partial u_{\omega}}{\partial \omega}\diff x=2\int_{\R^d}u_\omega L_+^{-1}u_\omega\diff x <0
\end{align*}
with $L_+$ defined above. Hence, let $S$ be the scaling operator $S=x\cdot \nabla +\frac{2}{p-1}$. A straightforward calculation gives $L_+ Su_\omega=2\omega u_\omega$. Indeed, if $u_\omega$ is a solution to~\eqref{eqsoliton}, then  $u_{\omega,\lambda}(x):=u_\omega(\lambda x)$ satisfies
\begin{equation*}
	\Delta u_{\omega,\lambda}+\lambda^2 |u_{\omega,\lambda}|^{p-1}u_{\omega,\lambda}=-\omega u_{\omega,\lambda}
\end{equation*}
for all $\lambda\in \R_+\smallsetminus\{0\}$. Hence, by taking the derivative of this equation with respect to $\lambda$ and choosing $\lambda=1$, we obtain $L_+ Su_\omega=2\omega u_\omega$. As a consequence, 
\begin{align*}
2\int_{\R^d}u_\omega L_+^{-1}u_\omega\diff x=\frac{1}{\omega}\int_{\R^d}u_\omega Su_\omega\diff x=\frac{1}{\omega}\int_{\R^d}u_\omega\left(x\cdot \nabla +\frac{2}{p-1}\right)u_\omega\diff x=\frac{1}{\omega}\left(-\frac d2 +\frac{2}{p-1}\right)\int_{\R^d}|u_\omega|^2\diff x
\end{align*}
which is strictly negative if and only if $p<1+\frac 4d$. 

As a consequence, if $1<p<1+\frac{4}{d}$, Theorem~\ref{thm:localcoercivitygen} applies and gives the local coercivity of $D^2_{u_\xi}\mathcal L_\xi$. We then have:
\begin{theorem}
	\label{thmstabilitySc} Let $d=1,2,3$ and $1<p<1+\frac 4d$. The solitary wave $u_{\omega,c}$, defined as in~\eqref{eqsolitarywaves} is an orbitally stable relative equilibrium. 
\end{theorem}
When $d=1$ and $3\leq p <5$, this follows from Theorem~\ref{thm:localcoercivitygen} together with Proposition~\ref{prop:hessianestimate} and the results of~\cite{debgenrot15}. When $d=1,2,3$ and $1<p<3$, the nonlinearity is not sufficiently smooth to ensure the ``propagation of the regularity" for initial conditions in $\Dcal=H^3(\R^d)$, as required in~\cite{debgenrot15} (see \cite{cazenave2003}). Hence, the results of~\cite{debgenrot15} cannot be directly applied in this case. Nevertheless, to prove the orbital stability once one has the coercivity of $\mathcal L_{\xi}$, we can use Theorem~10 of~\cite{debgenrot15} the proof of which can be easily adapted in the case of the Schr\"odinger equation with a power nonlinearity.

\begin{remark}\mbox{}
\begin{enumerate}
\item As announced at the end of Section~\ref{s:maintheoremhil}, $\nabla^2\Lcal_\xi$ is an unbounded partial differential operator and we are in the setting of Theorem~\ref{thm:localcoercivitygen}, not of Theorem~\ref{thm:localcoercivity}, nor of Theorem~\ref{thm:gssII} below, which comes from~\cite{gssII}.\\
\item A proof of the orbital stability of the soliton of the focusing NLSE for $1<p<1+\frac 4d$, $d=1,2,3$ was given originally using concentration-compactness arguments in~\cite{cazlions} and with a variational method in \cite{weinstein86}. Finally, in~\cite{gssII}, some of the spectral arguments we used to control $\nabla^2\Lcal_\xi$ are provided, but a complete proof of orbital stability is lacking for reasons further explained in Section~\ref{s:comparegss}.
\end{enumerate}
\end{remark}

%
%
\subsection{Stability of solitons for a system of  coupled nonlinear Schrödinger equations}\label{ss:cnlssoliton}

We consider the system of two coupled nonlinear Schrödinger equations given by
\begin{equation}
	\label{eqnlssys}
	\left\{
	\begin{aligned}
		&i\partial_t u_1(t,x) +\Delta u_1(t,x) +(\alpha |u_1(t,x)|^2+\delta |u_2(t,x)|^2) u_1(t,x)=0\\
		&i\partial_t u_2(t,x) +\Delta u_2(t,x) +(\delta |u_1(t,x)|^2+\gamma |u_2(t,x)|^2) u_2(t,x)=0\\
		&u(0,x)=u(x)
	\end{aligned}
	\right.
\end{equation}
with $u(t,x)=\begin{pmatrix}u_1(t,x)\\u_2(t,x)\end{pmatrix}:\R\times\R^d\to \C^2$ and $d=1,2,3$. Here, $\alpha, \gamma\in \R_+$ and $\delta \in \R_+\smallsetminus\{0\}$ are parameters of the model.

Coupled NLSEs have been used to model nonlinear wave propagation in a variety of physical systems. In nonlinear optics, they describe light propagation in birefringent fibers~\cite{agrawal}. In the study of ocean waves, they have been proposed as a model for the generation of rogue waves in crossing sea states : these are two-component wave systems with different directions of propagation (See~\cite{onorato} and references therein). They also appear in the study of two-component Bose-Einstein condensates~\cite{antoineduboscq, pitstr2003}. A central topic in each of these situations is the stability or instability of solitions and plane wave solutions of those equations.  We consider solitons in this subsection, and plane waves in the next one.

In dimension $d=1,2,3$, the Cauchy problem~\eqref{eqnlssys} is locally well posed in $H^1(\R^d,\C^2)$ \cite{caz-96}. Moreover, it has been proved in~\cite{fanmon-07} that, in dimension $d=1$, \eqref{eqnlssys} is globally well posed in $H^1(\R^d,\C^2)$.

Equation \eqref{eqnlssys} is the Hamiltonian differential equation associated to the function $H$ defined by 
\begin{align}
	\label{eqenergysys}
	H(u)=&\,\frac{1}{2}\int_{\R^d} \left(\vb\nabla u_1(x)\vb^2+\vb\nabla u_2(x)\vb^2\right)\diff x\nonumber\\
	&-\frac{1}{4}\int_{\R^d}\left( \alpha\vb u_1(x)\vb^4+2\delta|u_1(x)|^2|u_2(x)|^2+\gamma \vb u_2(x)\vb^4\right)\diff x.
\end{align}
Let $G=\R\times \R\times \R^d$ and define its action on $E=H^1(\R^d,\C^2)$ via
		\begin{equation}\label{eqmaninvariance}
		\forall u\in H^1(\R^d,\C^2),\quad \left(\Phi_{\gamma_1,\gamma_2,\tau}(u)\right)(x)=\begin{pmatrix}e^{-i\gamma_1} u_1(x-\tau)\\ e^{-i\gamma_2} u_2(x-\tau) \end{pmatrix}.
	\end{equation}
The group $G$ is an invariance group for the dynamics and the quantities 
\begin{align}
	\label{eqL2norm1}
	&F_1(u)=\frac{1}{2}\int_{\R^d}\vb u_1(x)\vb^2\diff x\\
	\label{eqL2norm2}
	&F_2(u)=\frac{1}{2}\int_{\R^d}\vb u_2(x)\vb^2\diff x\\
	\label{eqmomentumman}
	&F_{2+j}(u)
	=\frac12\int_{\R^d}u_1^*\cdot\left(\frac{1}{i}\partial_{x_j}\right) u_1\diff x+\frac12\int_{\R^d}u_2^*\cdot\left(\frac{1}{i}\partial_{x_j}\right) u_2\diff x
\end{align}
for $j=1,\dots,d$, are the corresponding constants of the motion. 
The family of solitary waves
\begin{equation}
 	\label{eqsolitarywavessys}
 	u_{\omega_1,\omega_2,c}(x)=e^{i \frac{c}{2}\cdot x}\Phi_{\omega_1,\omega_2}\\
\end{equation} 
with $c\in \R^d$ and $\Phi_{\omega_1,\omega_2}=\begin{pmatrix}\Phi_1\\\Phi_2\end{pmatrix}$ a solution to
\begin{equation}
	\label{eqsolsys}
	\left\{
	\begin{aligned}
		&-\Delta \Phi_1 -\omega_1 \Phi_1=(\alpha |\Phi_1|^2+\delta |\Phi_2|^2)\Phi_1\\
		&-\Delta \Phi_2 -\omega_2 \Phi_2= (\delta |\Phi_1|^2+\gamma |\Phi_2|^2) \Phi_2
	\end{aligned}
	\right.
\end{equation}
with $\omega_1,\omega_2\in \R$, $\omega_1,\omega_2<0$, are $G$-relative equilibria of \eqref{eqnlssys}. Indeed, if we define $\mathcal L_\xi$ by 
\begin{equation}
	\label{defL}
	\mathcal L_\xi(u)=H(u)-\xi_1 F_1(u)-\xi_2 F_2(u)-\sum_{j=1}^d \xi_{j+2} F_j(u),
\end{equation} 
we can easily verify that $u_{c,\omega_1,\omega_2}$ is a solution to the stationary equation $D_{u}\mathcal L_\xi=0$ with $\xi=(\omega_1-\frac{|c|^2}{4},\omega_2-\frac{|c|^2}{4},c)$. 

In particular, if $\omega_1=\omega_2=\omega_*<0$, then 
\begin{equation}
	\label{solpartsys}
\Phi_{\omega_*,\omega_*}(x)=u_{\omega_*}(x)\begin{pmatrix}\zeta_1\\ \zeta_2\end{pmatrix}
\end{equation}
with
\begin{equation}
	\label{conditiondelta}
	\alpha\zeta_1^2+\delta\zeta_2^2=1,\qquad
	\delta\zeta_1^2+\gamma\zeta_2^2=1
\end{equation}
and $u_{\omega_*}$ the unique positive solution to
\begin{equation}
	\label{eqsolsys2}
	-\Delta \varphi-\omega_* \varphi=|\varphi|^2\varphi
\end{equation}
is a solution to \eqref{eqsolsys}. As a consequence, $u_{\xi_*}(x)=e^{i\frac{\hat\xi_*}{2}\cdot x}\Phi_{\omega_*,\omega_*}(x)$ with $\xi_*=(\omega_*-\frac{|c|^2}{4},\omega_*-\frac{|c|^2}{4},c)$ is a $G$-relative equilibrium of \eqref{eqnlssys} and our goal is to investigate its orbital stability by means of Theorem~\ref{thm:localcoercivitygen}.
 Note that this kind of solution exists only if $\delta\notin[\min(\alpha,\gamma),\max(\alpha,\gamma)]$ or if $\delta=\alpha=\gamma$ which corresponds to the integrable case. In what follows, we will assume $\delta <\min(\alpha,\gamma)$ or $\delta>\max(\alpha,\gamma)$. 

First of all, to apply Theorem~\ref{thm:localcoercivitygen}, we have to show the existence of a family of solutions to $D_{u_\xi} \mathcal L_\xi=0$ for each $\xi$ in a neighbourhood of $\xi_*$.
The existence of such a family of $G$-relative equilibria is obtained using Theorem~\ref{thm:persistenceb}.

As before, Hypotheses~A (by taking again $\Dcal=H^3(\R^d,\C^2)$ and the $L^2$-scalar product on $E=H^1(\R^d,\C^2)$) and~B are clearly satisfied.

A straightforward calculation gives
\begin{align*}
D^2_{u}\mathcal L_{\xi}(v,w)=&\,\re\left[\int_{\R^d}(-\Delta w_1)v_1^*+\int_{\R^d}(-\Delta w_2)v_2^*\right.\\
&\left.-\int_{\R^d}(\alpha|u_1|^2+\delta|u_2|^2)w_1v_1^*-\int_{\R^d}(\delta|u_1|^2+\gamma|u_2|^2)w_2v_2^*\right.\\
&\left.-\int_{\R^d}\left(\alpha(u_1w_1^*+u_1^*w_1)+\delta(u_2w_2^*+u_2^*w_2)\right)u_1v_1^*\right.\\
&\left.-\int_{\R^d}\left(\delta(u_1w_1^*+u_1^*w_1)+\gamma(u_2w_2^*+u_2^*w_2)\right)u_2v_2^*\right.\\
&\left.-\xi_1\int_{\R^d}w_1 v_1^*-\xi_2\int_{\R^d}w_2 v_2^*-\sum_{j=1}^d\xi_{j+2}\int_{\R^d}\left(\frac{1}{i}\partial_{x_j}w_1 \right) v_1^*+\left(\frac{1}{i}\partial_{x_j}w_2 \right) v_2^*\right].
\end{align*}
Writing $v(x)=e^{i\frac{c}{2}\cdot x}\tilde v(x)$ and $w(x)=e^{i\frac{c}{2}\cdot x}\tilde w(x)$ and using the particular form of $u_\xi$, we obtain 
\begin{align*}
D^2_{u_\xi}\mathcal L_{\xi}(v,w)=&\,\re\left[\int_{\R^d}(-\Delta \tilde w_1)\tilde v_1^*+\int_{\R^d}(-\Delta \tilde w_2)\tilde v_2^*\right.\\
&\left.-\int_{\R^d}(\alpha|\Phi_1|^2+\delta|\Phi_2|^2)\tilde w_1\tilde v_1^*-\int_{\R^d}(\delta|\Phi_1|^2+\gamma|\Phi_2|^2)\tilde w_2\tilde v_2^*\right.\\
&\left.-\int_{\R^d}\left(\alpha(\Phi_1\tilde w_1^*+\Phi_1^*\tilde w_1)+\delta(\Phi_2\tilde w_2^*+\Phi_2^*\tilde w_2)\right)\Phi_1\tilde v_1^*\right.\\
&\left.-\int_{\R^d}\left(\delta(\Phi_1\tilde w_1^*+\Phi_1^*\tilde w_1)+\gamma(\Phi_2\tilde w_2^*+\Phi_2^*\tilde w_2)\right)\Phi_2\tilde v_2^*\right.\\
&\left.-\omega_1\int_{\R^d}\tilde w_1 \tilde v_1^*-\omega_2\int_{\R^d}\tilde w_2 \tilde v_2^*\right].
\end{align*}
Hence, by using the definition of $u_{\xi_*}$, we obtain
\begin{align*}
D^2_{u_{\xi_*}}\mathcal L_{\xi_*}(v,w)=&\,\re\left[\int_{\R^d}(-\Delta \tilde w_1)\tilde v_1^*+\int_{\R^d}(-\Delta \tilde w_2)\tilde v_2^*-\int_{\R^d}u_{\omega_*}^2\tilde w_1\tilde v_1^*-\int_{\R^d}u_{\omega_*}^2\tilde w_2\tilde v_2^*\right.\\
&\left.-\int_{\R^d}\left(\alpha\zeta_1(\tilde w_1^*+\tilde w_1)+\delta\zeta_2(\tilde w_2^*+\tilde w_2)\right)\zeta_1u_{\omega_*}^2\tilde v_1^*\right.\\
&\left.-\int_{\R^d}\left(\delta\zeta_1(\tilde w_1^*+\tilde w_1)+\gamma\zeta_2(\tilde w_2^*+\tilde w_2)\right)\zeta_2u_{\omega_*}^2\tilde v_2^*\right.\\
&\left.-\omega_*\int_{\R^d}\tilde w_1\tilde v_1^*-\omega_*\int_{\R^d}\tilde w_2 \tilde v_2^*\right]=\, \langle \mathbbm{L} \tilde w, \tilde v\rangle=\langle \nabla^2 \mathcal L_{\xi_*}(u_{\xi_*}) w,  v\rangle
\end{align*}
with
\begin{equation}
	\label{defnabla2sys} 
	 \mathbbm{L}\tilde w= \begin{pmatrix} \mathbbm L_+ & {0}_{2\times 2}\\
	0_{2\times 2}& \mathbbm L_-\end{pmatrix} \tilde w,
\end{equation}
$\mathbbm{L}=U^*\nabla^2 \mathcal L_{\xi_*}(u_{\xi_*})U$ and $U$ a unitary matrix, 
and $\mathbbm L_+$, $\mathbbm L_-$ given by
\begin{align}
	\label{defLplus} 
	\mathbbm L_+&=\begin{pmatrix} -\Delta - u_{\omega_*}^2 -{\omega_*} - 2\alpha\zeta_1^2 u_{\omega_*}^2 & - 2\delta\zeta_1\zeta_2 u_{\omega_*}^2\\
	- 2\delta\zeta_1\zeta_2 u_{\omega_*}^2 & -\Delta - u_{\omega_*}^2 -{\omega_*} - 2\gamma\zeta_2^2 u_{\omega_*}^2\\
	\end{pmatrix},\\
	\label{defLminus} 
	\mathbbm L_-&= \begin{pmatrix} -\Delta -u_{\omega_*}^2-\omega_*&0\\  0 & -\Delta - u_{\omega_*}^2-\omega_*  \end{pmatrix}. 
\end{align}
Note that $\mathcal D(\nabla^2 \mathcal L_{\xi_*}(u_{\xi_*}))=H^2(\R^d)$. Next, since $u_{\omega_*}$ is the unique positive solution of \eqref{eqsolsys}, $L_-= -\Delta -u_{\omega_*}^2-\omega_*$ is a nonnegative operator and $\Ker(L_-)=\mathrm{span}\{u_{\omega_*}\}$. Hence, $$\Ker(\mathbbm L_-)=\mathrm{span}\left\{\begin{pmatrix}u_{\omega_*}\\0\end{pmatrix},\begin{pmatrix}0\\u_{\omega_*}\end{pmatrix}\right\}.$$

Next, to analyze the spectrum of $\mathbbm L_+$, it is convient to perform the orthogonal transformation defined by 
\begin{equation}
	\label{orthotrans}
	P=\frac{1}{\sqrt{\zeta_1^2+\zeta_2^2}}\begin{pmatrix}\zeta_1& \zeta_2\\ \zeta_2& -\zeta_1\end{pmatrix}
\end{equation}
which leads to
$$
P\mathbbm L_+P=\begin{pmatrix} -\Delta - 3u_{\omega_*}^2 -{\omega_*}  & 0\\
	0 & -\Delta - (3-2\delta(\zeta_1^2+\zeta_2^2))u_{\omega_*}^2 -{\omega_*}\\
	\end{pmatrix}:=\begin{pmatrix} L_+  & 0\\
	0 & L_{\delta}\\
	\end{pmatrix}.
$$
We know that the operator $L_+=-\nabla-3u_{\omega_*}^2-\omega_*$ has exactly one negative eigenvalue and that $\Ker(L_+)=\mathrm{span}\{\partial_{x_1}u_{\omega_*},\ldots, \partial_{x_d}u_{\omega_*}\}$. Moreover, from \cite[Lemma 4.1]{wei-96}, we can deduce that $\Ker(L_\delta)=\{0\}$ except if $\delta= 0$ or $\delta=\alpha=\gamma$. As a consequence, if $\delta \in \R_+^*\smallsetminus [\min(\alpha,\gamma),\max(\alpha,\gamma)]$, 
\begin{equation}
	\Ker(\mathbbm L_+)=\mathrm{span}\left\{P\begin{pmatrix}\partial_{x_j}u_{\omega_*}\\0\end{pmatrix},j=1,\ldots,d\right\}=\mathrm{span}\left\{\partial_{x_j}u_{\omega_*}\begin{pmatrix}\zeta_1\\\zeta_2\end{pmatrix},j=1,\ldots,d\right\}.
\end{equation}
This implies $\Ker(D^2_{u_{\xi_*}}\mathcal L_{\xi_*})=T_{u_{\xi_*}}\mathcal O_{u_{\xi_*}}$.

Now we have to count the negative eigenvalues of $L_\delta$. First of all, we remark that $$L_\delta=L_- - 2(1-\delta(\zeta_1^2+\zeta_2^2))u_{\omega_*}^2.$$
If $1-\delta(\zeta_1^2+\zeta_2^2)<0$, $L_\delta$ is clearly a positive operator. Then $n(L_\delta)=0$ and $n(D^2_{u_{\xi_*}}\mathcal L_{\xi_*})=1$. This corresponds to the case $\delta >\max(\alpha,\gamma)$.
If $\delta<\min{(\alpha,\gamma)}$, then $1-\delta(\delta_1^2+\delta_2^2)>0$ and
$$
\langle L_\delta u_{\omega_*},u_{\omega_*}\rangle=-2(1-\delta(\zeta_1^2+\zeta_2^2))\int_{\R^d}|u_{\omega_*}|^4<0
$$
so that $n(L_\delta)\ge1$. Since $L_\delta=L_+ +2\delta(\zeta_1^2+\zeta_2^2)u^2_{\omega_*}$, it is clear, by means of a min-max type argument, that $n(L_\delta)\le1$. Hence, $n(L_\delta)=1$ and $n(D^2_{u_{\xi_*}}\mathcal L_{\xi_*})=2$. In both cases  $n(D^2_{u_{\xi_*}}\mathcal L_{\xi_*})<+\infty$ and Theorem~\ref{thm:persistenceb} leads to the existence of a family of $G_{\mu_\xi}$-relative equilibria.

Next, to apply Theorem~\ref{thm:localcoercivitygen}, we have to show that $n(D^2_{u_{\xi_*}}\mathcal L_{\xi_*})=p(D^2_{\xi_*} W)$ where $D^2_{\xi_*} W$ is the notation for the Hessian of $W$ evaluated at $\xi_*$. 

A straightforward calculation gives 
\begin{align}
	\label{functW}
	W(\xi) =&\,H(u_\xi)-\xi_1 F_1(u_\xi)-\xi_2 F_2(u_\xi)-\sum_{j=1}^d \xi_{j+1}F_{j+2}(u_\xi)\nonumber\\
	=&\,\frac{1}{2}\int_{\R^d} \left(\vb\nabla \Phi_1(x)\vb^2+\vb\nabla \Phi_2(x)\vb^2\right)\diff x
	\nonumber\\
	&-\frac{1}{4}\int_{\R^d}\left( \alpha |\Phi_1(x)|^4+2\delta|\Phi_1(x)|^2|\Phi_2(x)|^2+\gamma |\Phi_2(x)|^4\right)\diff x\nonumber\\
	& -\frac{\omega_1}{2}\int_{\R^d}\vb \Phi_1(x)\vb^2\diff x-\frac{\omega_2}{2}\int_{\R^d}\vb \Phi_2(x)\vb^2\diff x.
\end{align}

Hence, using the fact that $\Phi_{\omega_1,\omega_2}$ is a solution to~\eqref{eqsolsys}, for each $k=1,\ldots,d+2$, we obtain
\begin{align*}
\frac{\partial W}{\partial \xi_k}(\xi)&=-\frac{\partial \omega_1}{\partial \xi_k}(\xi)\frac{1}{2}\int_{\R^d}\vb \Phi_1(x)\vb^2\diff x -\frac{\partial \omega_2}{\partial \xi_k}(\xi)\frac{1}{2}\int_{\R^d}\vb \Phi_2(x)\vb^2\diff x\\
&= -\frac{\partial \omega_1}{\partial \xi_k}(\xi)F_1(\Phi_{\omega_1,\omega_2}) -\frac{\partial \omega_2}{\partial \xi_k}(\xi)F_2(\Phi_{\omega_1,\omega_2}).
\end{align*}

Recalling $\omega_1(\xi)=\xi_1+\frac{|\hat \xi|^2}{4}$ and $\omega_2(\xi)=\xi_2+\frac{|\hat \xi|^2}{4}$, we have

\begin{equation*}
\left\{
\begin{aligned}
&\frac{\partial W}{\partial \xi_1}(\xi)=-F_1(\Phi_{\omega_1,\omega_2})\\
&\frac{\partial W}{\partial \xi_2}(\xi)=-F_2(\Phi_{\omega_1,\omega_2})\\ 
&\frac{\partial W}{\partial \xi_k}(\xi)=-\frac{\xi_k}{2}(F_1(\Phi_{\omega_1,\omega_2})+F_2(\Phi_{\omega_1,\omega_2}))&\text{for } k=3,\ldots,d+2\end{aligned}
\right..
\end{equation*}
Next, a straightforward computation gives, for all $\ell=1,\ldots,d+2$,
$$
\left\{
\begin{aligned}
&\frac{\partial^2 W}{\partial \xi_\ell\partial \xi_1}(\xi)=-\frac{\partial}{\partial \xi_\ell}F_1(\Phi_{\omega_1,\omega_2})\\
&\frac{\partial^2 W}{\partial \xi_\ell\partial \xi_2}(\xi)=-\frac{\partial}{\partial \xi_\ell}F_2(\Phi_{\omega_1,\omega_2})
\end{aligned}
\right.
$$
and
$$
\frac{\partial^2 W}{\partial \xi_l\partial\xi_k}(\xi)=-\frac{\xi_k}{2}\left(\frac{\partial}{\partial \xi_l}F_1(\Phi_{\omega_1,\omega_2})+\frac{\partial}{\partial \xi_l}F_2(\Phi_{\omega_1,\omega_2})\right)-\frac{\delta_{kl}}{2}(F_1(\Phi_{\omega_1,\omega_2})+F_2(\Phi_{\omega_1,\omega_2}))
$$
for $k=3,\ldots,d+2$.

Note that for all $k=3,\ldots,d+2$ and for all $\ell=1,\ldots,d+2$,
$$
\frac{\partial^2 W}{\partial \xi_{\ell}\partial \xi_k}(\xi)-\frac{\xi_k}{2}\left(\frac{\partial^2 W}{\partial \xi_{\ell}\partial \xi_1}(\xi)+\frac{\partial^2 W}{\partial \xi_{\ell}\partial \xi_2}(\xi)\right)=-\frac{\delta_{k\ell}}{2}\left(F_1(\Phi_{\omega_1,\omega_2})+F_2(\Phi_{\omega_1,\omega_2})\right).
$$
Hence, for all $\lambda \in \R$, 
\begin{align*}
\det(D^2_\xi W-\lambda\bbone)=&\,\left(-\frac{1}{2}\left(F_1(\Phi_{\omega_1,\omega_2})+F_2(\Phi_{\omega_1,\omega_2})\right)-\lambda\right)^d\det\left(\begin{pmatrix}\frac{\partial^2 W}{\partial \xi_{1}^2}&\frac{\partial^2 W}{\partial \xi_{1}\partial \xi_2}\\\frac{\partial^2 W}{\partial \xi_{2}\partial \xi_1}&\frac{\partial^2 W}{\partial \xi_{2}^2}\end{pmatrix}-\lambda\bbone\right)\\
\end{align*}
This means that $D^2_\xi W$ has at least $d$ negatives eigenvalues and $p(D^2_\xi W)\le 2$. More precisely, the two remaining eigenvalues are the eigenvalues of the matrix
$$
M=\begin{pmatrix}\frac{\partial^2 W}{\partial \xi_{1}^2}&\frac{\partial^2 W}{\partial \xi_{1}\partial \xi_2}\\\frac{\partial^2 W}{\partial \xi_{2}\partial \xi_1}&\frac{\partial^2 W}{\partial \xi_{2}^2}\end{pmatrix}=\begin{pmatrix}-\frac{\partial F_1}{\partial \omega_{1}}&-\frac{\partial F_1}{\partial \omega_{2}}\\-\frac{\partial F_2}{\partial \omega_{1}}&-\frac{\partial F_2}{\partial \omega_{2}}\end{pmatrix}
$$

On the one hand, if $\frac{\partial F_1}{\partial \omega_1}\frac{\partial F_2}{\partial \omega_2}-\frac{\partial F_1}{\partial \omega_2}\frac{\partial F_2}{\partial \omega_1}<0$, the matrix $M$ has exactly one positive eigenvalue and $p(D^2_\xi W)= 1$. On the other hand, if $\frac{\partial F_1}{\partial \omega_1}\frac{\partial F_2}{\partial \omega_2}-\frac{\partial F_1}{\partial \omega_2}\frac{\partial F_2}{\partial \omega_1}>0$ and $\frac{\partial F_1}{\partial \omega_1}+\frac{\partial F_2}{\partial \omega_2}<0$, the matrix $M$ has two positive eigenvalues and and $p(D^2_\xi W)= 2$.  

By applying Theorem~\ref{thm:localcoercivitygen}, we obtain to the following result.

\begin{theorem}
\begin{enumerate}
	\item Let $\delta>\max(\alpha,\gamma)$. If $$\frac{\partial F_1}{\partial \omega_1}(\Phi_{\omega_*,\omega_*})\frac{\partial F_2}{\partial \omega_2}(\Phi_{\omega_*,\omega_*})-\frac{\partial F_1}{\partial \omega_2}(\Phi_{\omega_*,\omega_*})\frac{\partial F_2}{\partial \omega_1}(\Phi_{\omega_*,\omega_*})<0$$ 
then $u_{\omega_*,\omega_*,c}$ is such that the local coercivity estimate~\eqref{eq:localcoercivitygen} is satisfied.
	 \item Let $\delta<\min(\alpha,\gamma)$. If $$\frac{\partial F_1}{\partial \omega_1}(\Phi_{\omega_*,\omega_*})\frac{\partial F_2}{\partial \omega_2}(\Phi_{\omega_*,\omega_*})-\frac{\partial F_1}{\partial \omega_2}(\Phi_{\omega_*,\omega_*})\frac{\partial F_2}{\partial \omega_1}(\Phi_{\omega_*,\omega_*})>0 \text{ and } \frac{\partial F_1}{\partial \omega_1}(\Phi_{\omega_*,\omega_*})+\frac{\partial F_2}{\partial \omega_2}(\Phi_{\omega_*,\omega_*})<0$$ then $u_{\omega_*,\omega_*,c}$ is is such that the local coercivity estimate~\eqref{eq:localcoercivitygen} is satisfied.
	 \end{enumerate}
\end{theorem}

A lengthly but straightforward computation gives 
\begin{align*}
	\frac{\partial F_1}{\partial \omega_1}&(\Phi_{\omega_*,\omega_*})\frac{\partial F_2}{\partial \omega_2}(\Phi_{\omega_*,\omega_*})-\frac{\partial F_1}{\partial \omega_2}(\Phi_{\omega_*,\omega_*})\frac{\partial F_2}{\partial \omega_1}(\Phi_{\omega_*,\omega_*})=\frac{\delta_1^2\delta_2^2}{2\omega_*}\left(1-\frac d2\right)\left(\int_{\R^d}u_{\omega_*}^2\right)\int_{\R^d}u_{\omega_*}L_{\delta}^{-1}u_{\omega_*}
\end{align*} 
and
\begin{align*}
\frac{\partial F_1}{\partial \omega_1}(\Phi_{\omega_*,\omega_*})+\frac{\partial F_2}{\partial \omega_2}(\Phi_{\omega_*,\omega_*})=\frac{\zeta_1^4+\zeta_2^4}{\zeta_1^2+\zeta_2^2}\frac{\left(1-\frac d2\right)}{{2\omega_*}}\int_{\R^d}u_{\omega_*}^2+\frac{2\zeta_1^2\zeta_2^2}{\zeta_1^2+\zeta_2^2}\int_{\R^d}u_{\omega_*}L_{\delta}^{-1}u_{\omega_*}.	
\end{align*}
Recalling that if $\delta>\max(\alpha,\gamma)$ then $L_\delta$ is a positive operator, together with Proposition~\ref{prop:hessianestimate} and the results of~\cite{debgenrot15}, we obtain the following corollary.

\begin{theorem} Let $d=1$.
\begin{enumerate}
	\item If $\delta>\max(\alpha,\gamma)$ then $u_{\omega_*,\omega_*,c}$ is an orbitally stable relative equilibrium.
	 \item If $\delta<\min(\alpha,\gamma)$ and $\int_{\R^d}u_{\omega_*}L_{\delta}^{-1}u_{\omega_*}<0$ then $u_{\omega_*,\omega_*,c}$ is an orbitally stable relative equilibrium.
\end{enumerate}
\end{theorem}

Note that a proof of orbital of $u_{\omega_*,\omega_*,0}$ in dimension $d=1$ and for $\alpha=\gamma=1$ was given in~\cite{otha-96} using concentration-compactness arguments.

\subsection{Stability of plane waves for a system of  coupled nonlinear Schrödinger equations}\label{ss:cnlsplanewaves}
We consider a system of two coupled nonlinear Schrödinger equations on a one-dimensional torus $\T=\T^1$ is the one-dimensional torus of length $L>0$. The system is given by
\begin{equation}
	\label{eq:nlsmanakovbis}
	\left\{
	\begin{aligned}
		&i\partial_t u_1(t,x) +\beta\Delta u_1(t,x) +(\alpha|u_1(t,x)|^2+\delta |u_2(t,x)|^2) u_1(t,x)=0\\
		&i\partial_t u_2(t,x) +\beta\Delta u_2(t,x) +(\delta |u_1(t,x)|^2+\gamma |u_2(t,x)|^2) u_2(t,x)=0\\
		&u(0,x)=u(x)
	\end{aligned}
	\right.
\end{equation}
with $u(t,x)=\begin{pmatrix}u_1(t,x)\\u_2(t,x)\end{pmatrix}:\R\times\T\to \C^2$. The constants $\alpha,\gamma,\delta\in \R$ and $\beta\in \R_{+}^*$ are parameters of the model. 

As already mentioned, this system is of relevance in nonlinear optics. Linear instability is in this context referred to as \emph{modulational instability} and was studied for various parameter ranges in~\cite{agrawal, trilloandcy, FMMW}, among others. We will use the methods exposed in this paper to show that, in the parameter regimes where linear stability can be established, orbital stability also holds.

 
The four-parameter family of plane waves
\begin{equation}
 	\label{eq:manplanewaves1}
 	\tilde u_{\en}(t,x)=\begin{pmatrix}\zeta_1 e^{ik_1\cdot x}e^{-i\en_1 t}\\ \zeta_2 e^{ik_2\cdot x}e^{-i\en_2 t}\end{pmatrix}
\end{equation} 
with $\en=(\en_1,\en_2)\in \R^2$, $(k_1,k_2)\in \frac{2\pi}{L}\Z^2$ and $\alpha=(\alpha_1,\alpha_2) \in \R^2$ such that
\begin{equation}
	\label{eq:disprelmanbis}
	\left\{
	\begin{aligned}
		\en_1 &=\beta k_1^2-( \alpha \zeta_1^2+\delta \zeta_2^2)\\
		\en_2 &=\beta k_2^2-( \delta \zeta_1^2+\gamma \zeta_2^2).
	\end{aligned}
	\right.
\end{equation}
are solution to the equation \eqref{eq:nlsmanakovbis} and we are interested in study their orbital stability. Using Galilean invariance of the equation (see~\cite{debgenrot15}), the stability of these plane waves is seen to be equivalent to that of 
\begin{equation}
 	\label{eq:manplanewaves2}
 	\tilde u_{\en}(t,x)=\begin{pmatrix}\zeta_1 e^{ik x}e^{-i\en_1 t}\\ \zeta_2 e^{-ik x}e^{-i\en_2 t}\end{pmatrix}
\end{equation} 
with $k=k_1-k_2$. Furthermore, we can easily remark that $\tilde u_{\en}(t,x)$ can be written in the form 
\begin{equation*}
 	\tilde u_{\en}(t,x)=\begin{pmatrix}e^{ik x}&0\\ 0&e^{-ik x}\end{pmatrix}u_{\en}(t,x)
\end{equation*} 
with 
\begin{equation}
 	\label{eq:manplanewaves3}
 	u_{\en}(t,x)=\begin{pmatrix}\zeta_1 e^{-i\en_1 t}\\ \zeta_2 e^{-i\en_2 t}\end{pmatrix}
\end{equation} 
a solution to the system of coupled nonlinear Schrödinger equations
\begin{equation}
	\label{eq:nlsmanakov}
	\left\{
	\begin{aligned}
		&i\partial_t u_1 +\beta\Delta u_1+2\beta ik\nabla u_1 +(\alpha|u_1|^2+\delta |u_2|^2) u_1-\beta k^2u_1=0\\
		&i\partial_t u_2 +\beta\Delta u_2-2\beta ik\nabla u_2 +(\delta |u_1|^2+\gamma |u_2|^2) u_2-\beta k^2u_2=0\\
		&u(0,x)=u(x)
	\end{aligned}
	\right.
\end{equation}

It is easy to show that the Cauchy problem \eqref{eq:nlsmanakov} is globally well-posed in $H^{1}(\T,\C^2)$ 
(since we consider here only the dimension $d=1$).

Equation \eqref{eq:nlsmanakov} is the Hamiltonian differential equation associated to the function $H$ defined by 
\begin{align}
	\label{eq:energyman}
	H(u)=&\,\frac{\beta}{2}\int_0^L \left(\vb(\nabla+ik) u_1(x)\vb^2+\vb(\nabla-ik) u_2(x)\vb^2\right)\diff x\nonumber\\
	&\,-\frac{1}{4}\int_0^L\left( \alpha\vb u_1(x)\vb^4+2\delta|u_1(x)|^2|u_2(x)|^2+\gamma \vb u_2(x)\vb^4\right)\diff x.
\end{align}

Let $G=\R\times \R$ and define its action on $E=H^{1}(\T,\C^2)$ via
		\begin{equation}\label{eq:maninvariance}
		\forall u=\begin{pmatrix}u_1\\u_2\end{pmatrix}\in H^1(\T,\C^2),\quad \left(\Phi_{\gamma_1,\gamma_2}(u)\right)(x)=\begin{pmatrix} e^{-i\gamma_1} u_1(x)\\  e^{-i\gamma_2} u_2(x) \end{pmatrix}.
	\end{equation}
The group $G$ is an invariance group for the dynamics and the quantities 
\begin{align*}
	F_1(u)=\frac{1}{2}\int_{0}^L\vb u_1(x)\vb^2\diff x,\quad
	F_2(u)=\frac{1}{2}\int_{0}^L\vb u_2(x)\vb^2\diff x
\end{align*}
are the corresponding constants of the motion.

The two-parameter family of plane waves
\begin{equation}
 	\label{eq:manplanewaves}
 	u_{\en}(x)=\begin{pmatrix}\zeta_1 \\ \zeta_2 \end{pmatrix}
\end{equation}
with $\en=(\en_1,\en_2)\in \R^2$ and $\zeta_1,\zeta_2 \in \R\smallsetminus\{0\}$ such that
\begin{equation}
	\label{eq:disprelman}
	\left\{
	\begin{aligned}
		\en_1 &=\beta k^2-( \alpha \zeta_1^2+\delta \zeta_2^2)\\
		\en_2 &=\beta k^2-( \delta \zeta_1^2+\gamma \zeta_2^2).
	\end{aligned}
	\right.
\end{equation}
 are solutions to the stationary equation \eqref{eq:defrelequilibria}. As a consequence, $u_{\en}$ is $G_{\mu_\xi}$-relative equilibria of \eqref{eq:nlsmanakov} and our goal is to investigate the orbital stability of these plane wave solutions by applying Theorem~\ref{thm:localcoercivitygen}. 
 
From now, assume that $\alpha\gamma\neq \delta^2.$ This is the necessary and sufficient condition for the map 
$$
\mu=(\frac L2|\zeta_1|^2, \frac L2|\zeta_2|^2)\to (\xi_1,\xi_2)
$$
defined in~\eqref{eq:disprelman} to be invertible. Its inverse is $\hat F$, which is a diffeomorphism.
Note that this condition corresponds to the case in which the system is not completely integrable.

As before, hypotheses A (by taking the $L^2$- scalar product on $E=H^1(\T,\C^2)$), B and C are clearly satisfied.
By using the dispersion relation~\eqref{eq:disprelmanbis}, we have 
\begin{align}\label{eq:defWMan}
W(\en)=&\,H(u_\en)-\en_1F_1(u_\en)-\en_2F_2(u_\en)\nonumber\\
=&\, \frac{L}{4(\alpha\gamma-\delta^2)}\left[\gamma(\en_1-\beta k^2)^2-2\delta(\en_1-\beta k^2)(\en_2-\beta k^2)+\alpha(\en_2-\beta k^2)^2\right]
\end{align}
As a consequence, 
\begin{align}\label{eq:DWMan}
D_\en^2W=\frac{L}{2(\alpha\gamma-\delta^2)}\begin{pmatrix} \gamma &-\delta\\ -\delta &\alpha
\end{pmatrix}
\end{align}

It is clear that $D_\en^2W$ is non-degenerate. Moreover a straightforward calculation shows that
\begin{enumerate}
\item If $\alpha\gamma-\delta^2>0$ and $\min(\alpha,\gamma)>0$, then $p(D^2_\en W)=2$;
\item If $\alpha\gamma-\delta^2>0$ and $\max(\alpha,\gamma)<0$, then $p(D^2_\en W)=0$;
\item If $\alpha\gamma-\delta^2<0$, then $p(D^2_\en W)=1$.
\end{enumerate}

Next, we have to compute $D^2_{u_\xi}\mathcal L_\xi(v,v)$ with $\mathcal L_\xi(u)=H(u)-\xi_1 F_1(u)-\xi_2 F_2 (u)$. A straightforward calculation gives $D^2_{u_\xi}\mathcal L_\xi (v,v)=\langle \nabla^2\mathcal L_\xi ({u_\xi}) v,v\rangle$
with 
\begin{equation}\label{eq:Hessiantorus}
\nabla^2\mathcal L_\xi ({u_\xi})=
\begin{pmatrix}
-\beta \Delta -2\alpha\zeta_1^2& -2\delta\zeta_1\zeta_2&2\beta k\nabla  &0\\
-2\delta\zeta_1\zeta_2 & -\beta \Delta -2\gamma\zeta_2^2 &0 & -2\beta k\nabla\\
-2\beta k\nabla& 0 &-\beta \Delta & 0\\
0 &2\beta k\nabla &0 &-\beta\Delta
\end{pmatrix}.
\end{equation}
In particular, in this functional space setting the hypotheses of Lemma~\ref{lem:existenceHessian} are clearly satisfied by $D^2_{u_\xi}\mathcal L_\xi (\cdot,\cdot)$. 

Using Fourier series, we can show that the eigenvalues of $\nabla^2\mathcal L_\xi ({u_\xi})$, for all $n\in\N$, are of the form
\begin{align}
\label{eq:evD2Lkplus}
\lambda^+_{\pm,n}=\beta\left(\frac{2\pi}{L}n\right)^2+\frac{1}{2}\left(C_+\pm\sqrt{C_+^2+16\beta^2k^2\left(\frac{2\pi}{L}n\right)^2}\right)\\
\label{eq:evD2Lkminus}
\lambda^-_{\pm,n}=\beta\left(\frac{2\pi}{L}n\right)^2+\frac{1}{2}\left(C_-\pm\sqrt{C_-^2+16\beta^2k^2\left(\frac{2\pi}{L}n\right)^2}\right)
\end{align}
with 
\begin{align*}
C_{\pm}=-(\alpha\zeta_1^2+\gamma\zeta_2^2)\pm\sqrt{(\alpha\zeta_1^2+\gamma\zeta_2^2)^2-4\zeta_1^2\zeta_2^2(\alpha\gamma-\delta^2)},
\end{align*}
By analyzing the sign of the eigenvalues for $n=0$, we obtain the following situation 
\begin{enumerate}
\item  If $\alpha\gamma-\delta^2>0$ and $\min(\alpha,\gamma)>0$, then 
$\lambda^+_{-,0}$ and $\lambda^-_{-,0}$ are both negative and 
$n(D^2_{u_\xi}\mathcal L_\xi)\ge2$. 
\item If $\alpha\gamma-\delta^2>0$ and $\min(\alpha,\gamma)<0$, then $\lambda^+_{+,0}$ and $\lambda^-_{+,0}$ are both positive. This implies $\lambda^+_{+,n}>0$ and $\lambda^-_{+,n}>0$ for all $n\in \N$.
\item  If $\alpha\gamma-\delta^2<0$, then $\lambda^-_{-,0}<0$, $\lambda^+_{+,0}>0$ and $n(D^2_{u_\xi}\mathcal L_\xi)\ge1$. As a consequence $\lambda^+_{+,n}>0$ for all $n\in \N$.
\end{enumerate}
In all the cases, the two remaining eigenvalues are both $0$ with purely imaginary eigenvectors which implies that
$$
T_{u_\xi}\mathcal O_{u_\xi}=\mathrm{span}\left\{\begin{pmatrix}i\\0\end{pmatrix}, \begin{pmatrix}0\\i\end{pmatrix}\right\}\subset \Ker(D^2_{u_\xi}\mathcal L_\xi).
$$
and $\lambda_{+,n}^{\pm}>0$ for all $n\in \N^*$.

Next, a straightforward calculation shows that if we assume 
\begin{align}
\label{eq:condstabmank1}
C_{\pm}^2+16\beta^2k^2\left(\frac{2\pi}{L}\right)^2>16\beta^2 k^4
\end{align}
then $\lambda^{\pm}_{-,n}$ is increasing as a function of $n$ for all $n\in \N^*$. Hence, it enough to suppose that $\lambda^{\pm}_{-,1}>0$, that is
\begin{align}
\label{eq:condstabmank2}
\beta\left(\frac{2\pi}{L}\right)^2+C_{\pm}>4\beta k^2,
\end{align}
to conclude that $p(D^2_\en W)=n(D^2_{u_\xi}\mathcal L_\xi)$.
Note that condition~\eqref{eq:condstabmank2} implies condition~\eqref{eq:condstabmank1}. Moreover, since $C_+\ge C_-$, it is enough to assume that 
$$
\beta\left(\frac{2\pi}{L}\right)^2+C_{-}>4\beta k^2.
$$
Moreover, it is then clear that 
$$
\inf(\sigma(\nabla^2\mathcal L_\xi ({u_\xi}))\cap (0,+\infty))>0.
$$
and
$$
\Ker(D^2_{u_\xi}\mathcal L_\xi)=\Ker(\nabla^2\mathcal L_\xi ({u_\xi}))=\mathrm{span}\left\{\begin{pmatrix}i\\0\end{pmatrix}, \begin{pmatrix}0\\i\end{pmatrix}\right\}=T_{u_\xi}\mathcal O_{u_\xi}.
$$

Hence Theorem~\ref{thm:localcoercivitygen} applies and, together with Proposition~\ref{prop:hessianestimate}  and the results of~\cite{debgenrot15}, leads to the following result.

\begin{proposition} Let $k\in \frac{2\pi}{L}\Z$, $\zeta_1,\zeta_2\in \R^*$ and  $\gamma_1,\gamma_2\in\R, \gamma_{12}\in \R^*$ such that $\alpha\gamma\neq \delta^2$. If
\begin{align}
\label{eq:condstabmanpropk}
\beta\left(\frac{2\pi}{L}\right)^2+C_{-}>4\beta k^2
\end{align} 
then $u_\xi=\begin{pmatrix}\zeta_1\\ \zeta_2\end{pmatrix}$ is an orbitally stable $G_{\mu_\xi}$- relative equilibrium.
\end{proposition}
We note that in the present example, one does not strictly need to use the Vakhitov-Kolokolov condition since a direct study of the Hessian of $D^2_{u_\xi}\Lcal_\xi$ restricted to $T_{u_\xi}\Sigma_{\mu_\xi}\cap (T_{u_\xi}\mathcal O_{u_\xi})^\perp$ could also be performed starting from~\eqref{eq:Hessiantorus}.  

\subsubsection{Case $k=0$} 
If $k=0$, condition~\eqref{eq:condstabmanpropk} reads as 
\begin{align}
\label{eq:condstabmanprop}
\beta\left(\frac{2\pi}{L}\right)^2-(\alpha\zeta_1^2+\gamma\zeta_2^2)-\sqrt{(\alpha\zeta_1^2+\gamma\zeta_2^2)^2-4\zeta_1^2\zeta_2^2(\alpha\gamma-\delta^2)}>0
\end{align} 
We remark that condition~\eqref{eq:condstabmanprop} is a necessary and sufficient condition for linear stability. To see this, note that the linearization of~\eqref{eq:nlsmanakov} around $u_\xi=\begin{pmatrix}\alpha_1\\ \alpha_2\end{pmatrix}$ is given by the system 
\begin{equation}
	\partial_t \begin{pmatrix} \mathrm{Re}(v_1)\\
	\mathrm{Re}(v_2)\\
	\mathrm{Im}(v_1)\\
	\mathrm{Im}(v_2)
	\end{pmatrix}
	=\mathbb{L}  \begin{pmatrix} \mathrm{Re}(v_1)\\
	\mathrm{Re}(v_2)\\
	\mathrm{Im}(v_1)\\
	\mathrm{Im}(v_2)
	\end{pmatrix},\quad 
	\label{eq:nlsmanakovlinmatrix}
	\mathbb{L}=\begin{pmatrix}
	-2\beta k\nabla& 0 &-\beta \Delta & 0\\
0 &2\beta k\nabla &0 &-\beta\Delta\\
	\beta \Delta +2\alpha\zeta_1^2& 2 \delta\zeta_1\zeta_2&-2\beta k\nabla  &0\\
2 \delta\zeta_1\zeta_2 & \beta \Delta +2\gamma\zeta_2^2 &0 & 2\beta k\nabla
	\end{pmatrix}.
\end{equation}
A solution to~\eqref{eq:nlsmanakov} is said to be linearly stable if all the eigenvalues of $\mathbb{L}$ are purely imaginary. By using Fourier series, the eigenvalues of~\eqref{eq:nlsmanakovlinmatrix} can be seen to be the zeros of the characteristic polynomial 
\begin{align*}
P_n(\lambda)=&\, \lambda^4-2\lambda^2 \beta n_L^2(-\beta n_L^2+(\alpha\zeta^2_1+\gamma\zeta_2^2)-4\beta k^2)+i \lambda 8 \beta n_L^2 \beta k n_L(\alpha\zeta^2_1-\gamma\zeta_2^2)\\
&+(\beta n_L^2)^3  (\beta n_L^2-2 (\alpha\zeta^2_1+\gamma\zeta_2^2)) + 4 (\beta n_L^2)^2 \zeta_1^2\zeta_2^2(\alpha\gamma-\delta^2)\\
&+8  \beta n_L^2 \beta^2 k^2 n_L^2 (-\beta n_L^2+(\alpha\zeta^2_1+\gamma\zeta_2^2)+2\beta k^2)
\end{align*}
with $n_L=\left(\frac{2\pi}{L}n\right)$. So, whenever $k=0$, $P_n(\lambda)$ reduces to
\begin{align*}
P_n(\lambda)=&\, \lambda^4-2\lambda^2 \beta n_L^2(-\beta n_L^2+(\alpha\zeta^2_1+\gamma\zeta_2^2))\\
&+(\beta n_L^2)^3  (\beta n_L^2-2 (\alpha\zeta^2_1+\gamma\zeta_2^2)) + 4 (\beta n_L^2)^2 \zeta_1^2\zeta_2^2(\alpha\gamma-\delta^2).
\end{align*} and, for all $n\in \N$ the eigenvalues of $\mathbb{L}$ are 
\begin{align*}
\tilde\lambda^2_{\pm,n}=&\,\beta n_L^2\left(-\beta n_L^2+(\alpha\zeta_1^2+\gamma\zeta_2^2)\pm\sqrt{(\alpha\zeta_1^2-\gamma\zeta_2^2)^2+4\zeta_1^2\zeta_2^2\delta^2}\right).
\end{align*}
Now it is clear that  $\tilde\lambda^2_{+,1}<0$ if and only if condition~\eqref{eq:condstabmanprop} holds. In that case, for all $n\in\N^*$, $\tilde\lambda^2_{+,n}\le \tilde\lambda^2_{+,1}$.  Moreover, $\tilde\lambda^2_{-,n}\le \tilde\lambda^2_{+,n}$ for all $n\in\N^*$. As a consequence, 	all the eigenvalues of $\mathbb{L}$ are purely imaginary and the corresponding plane wave is linearly stable if and only if~\eqref{eq:condstabmanprop} holds. 
 
\subsubsection{Case $k\neq 0$} 
If $k\neq0$,
in the particular case $\alpha\zeta_1^2=\gamma\zeta_2^2$ which is a generalization of the set of parameters treated in \cite{trilloandcy}, we can show that condition~\eqref{eq:condstabmanpropk} is a necessary and sufficient condition for linear stability. 

Indeed, in this case, $P_n(\lambda)$  the characteristic polynomial of~\eqref{eq:nlsmanakovlinmatrix} reduces to
\begin{align*}
P_n(\lambda)=&\, \lambda^4-2\lambda^2 \beta n_L^2(-\beta n_L^2+(\alpha\zeta^2_1+\gamma\zeta_2^2)-4\beta k^2)+i \lambda 8 \beta n_L^2 \beta k n_L(\alpha\zeta^2_1-\gamma\zeta_2^2)\\
&+(\beta n_L^2)^3  (\beta n_L^2-2 (\alpha\zeta^2_1+\gamma\zeta_2^2)) + 4 (\beta n_L^2)^2 \zeta_1^2\zeta_2^2(\alpha\gamma-\delta^2)\\
&+8  (\beta n_L^2)^2 \beta k^2 (-\beta n_L^2+(\alpha\zeta^2_1+\gamma\zeta_2^2)+2\beta k^2)
\end{align*} and, for all $n\in \N$, the eigenvalues are given by 
\begin{align*}
\tilde\lambda^2_{\pm,n}=&\,\beta n_L^2\Big[-\beta n_L^2+(\alpha\zeta_1^2+\gamma\zeta_2^2)-4\beta k^2\pm\sqrt{4\zeta_1^2\zeta_2^2\delta^2+16\beta k^2(\beta n_L^2-(\alpha\zeta_1^2+\gamma\zeta_2^2))}\Big].
\end{align*}

Now, a tedious but straightforward calculation shows that $\tilde\lambda^2_{+,n}<0$ for all $n\in \N^*$ if and only if condition~\eqref{eq:condstabmanpropk} holds. Moreover, $\tilde\lambda^2_{-,n}\le \tilde\lambda^2_{+,n}$ for all $n\in\N^*$. As a consequence, 	all the eigenvalues of $\mathbb{L}$ are purely imaginary and the corresponding plane wave is linearly stable if and only if~\eqref{eq:condstabmanpropk} holds.

\subsubsection{Physical interpretation} To sum up, we can conclude that, given $k=k_1-k_2\in \frac{2\pi}{L}\Z$, $\zeta_1,\zeta_2\in \R^*$ and  $\alpha,\gamma, \delta\in \R^*$ such that $\alpha\gamma\neq \delta^2$, if 
\begin{align}
\label{eq:condconcl}
\beta\left(\frac{2\pi}{L}\right)^2-(\alpha\zeta_1^2+\gamma\zeta_2^2)-\sqrt{(\alpha\zeta_1^2+\gamma\zeta_2^2)^2-4\zeta_1^2\zeta_2^2(\alpha\gamma-\delta^2)}>4\beta k^2
\end{align} 
then the plane waves given by~\eqref{eq:manplanewaves1}
are orbitally stable $G_{\mu_\xi}$- relative equilibria. Moreover, for $k=0$, if this condition is not satisfied the plane wave is unstable (at least linearly). For $k\neq 0$ this remains true whenever $\alpha\zeta_1^2=\gamma\zeta^2$. 

We know from \cite{debgenrot15} that plane waves solutions to a cubic defocusing nonlinear Schr\"odinger equation on the one-dimensional torus are orbitally stable. This means that whenever, $\delta=0$, $\alpha<0$ and $\gamma<0$, all the plane waves of the form 
\begin{equation*}
 	u_{\en}(t,x)=\begin{pmatrix}\zeta_1 e^{ik_1\cdot x}e^{-i\en_1 t}\\ \zeta_2 e^{ik_2\cdot x}e^{-i\en_2 t}\end{pmatrix}
\end{equation*}
are orbitally stable. It is natural to ask what happens if $|\delta|\neq 0$. We have two different situations: $k=0$ (\textit{i.e.} the plane waves have the same wave number $k_1=k_2$) and $k\neq 0$.
If $k=0$ and $\delta^2<\alpha\gamma$, which means that the coupling is weak, then $C_->0$ and condition~\eqref{eq:condconcl} remains true. This means that the plane waves with $k=0$ then remain stable.
If $k=0$ and $\delta^2>\alpha\gamma$, which means that the coupling is strong, then $C_-<0$ and condition~\eqref{eq:condconcl} fails at least if $L$ is large enough. Then the plane waves considered become unstable.

In the case $k\not=0$, note that condition~\eqref{eq:condconcl} can be satisfied only in the case~(2) above, namely when $\alpha\gamma-\delta^2>0$, and $\max(\alpha, \gamma)<0$, since otherwise $C_-<0$. This corresponds to a relatively small perturbation of two uncoupled defocusing Schr\"odinger equations with orbitally stable plane wave solutions. Condition~\eqref{eq:condconcl} can then be satisfied for a finite number of values of $k$, provided $C_-$ is large enough, but it fails for larger ones. The size of $C_-$ depends in particular on the ``power'' of the plane wave, determined by $|\zeta_1|$ and $|\zeta_2|$. For larger values of $k$, the plane wave becomes linearly unstable, on the other hand, even at weak coupling. In other words, high $k$ plane waves show modulational instability, even at arbitrarily low $\delta$. 

\section{On the link with Grillakis-Shatah-Strauss}\label{s:comparegss}

We will now compare the results in this paper to~\cite{gssII}. As we have already pointed out, in~\cite{gssII} a proof of orbital stability is proposed with respect to an a priori different subgroup of $G$ and under similar but nevertheless different conditions. Both in order to understand the general structure of the theory and with an eye towards further applications, it is important to understand the relations between the two approaches.

\subsection{The main coercivity estimate of Grillakis-Shatah-Strauss}
Since in~\cite{gssII} the phase space $E$ on which the dynamics takes place is taken to be a Hilbert space, we place ourselves for this discussion in the Hilbert space setting of Section~\ref{s:maintheoremhil} and consider the situation described by~\eqref{eq:deftildeu}-\eqref{eq:spectraldec}. 

To state the coercivity estimate of~\cite{gssII} which is the analog of our Theorem~\ref{thm:localcoercivity},  we need some additional notation. We define
\begin{equation}
\tilde W:\Omega\cap \frak g_\xi \to \mathcal L_\xi(u_\xi)\in\R,
\end{equation}
which is the restriction of the W-function~\eqref{eq:WhatF} to the sub-Lie-algebra~$\frak g_\xi$ of $\frak g$, defined in~\eqref{eq:gxi}. Also
\begin{equation}\label{eq:tildeoxi}
\tilde{\calO}_{u_\xi}=\Phi_{G_\xi}(u_\xi),
\end{equation}
is the $G_\xi$ orbit through $u_\xi$. Since a priori $G_\xi$ differs from $G_{\mu_\xi}$, one should not confuse $\tilde\Ocal_{u_\xi}$ with $\Ocal_{u_\xi}$, which is the $G_{\mu_\xi}$-orbit through $u_\xi$. We introduce furthermore
\begin{equation}
\tilde\Sigma_\xi=\{v\in \Ban\mid \eta\cdot F(v)= \eta\cdot\mu_\xi, \forall \eta\in\frak g_\xi\}.
\end{equation}
In other words, $\tilde\Sigma_\xi$ is the constraint surface corresponding to the constants of the motion $\eta\cdot F$ for $\eta\in\frak g_\xi$. Note that $\Sigma_{\mu_\xi}\subset \tilde\Sigma_\xi$. In fact, when the moment map is regular at $\mu_\xi$, then $\tilde\Sigma_\xi$ is a submanifold of $\Ban$ of co-dimension $\text{dim} \frak g_\xi$ which contains the submanifold $\Sigma_{\mu_\xi}$, itself of codimension~$\text{dim} \frak g$. 
The following theorem, which is the analog of Theorem~\ref{thm:localcoercivity} above, can be inferred from the proof of Theorem~4.1 in~\cite{gssII}. 
\begin{theorem}
	\label{thm:gssII}
	Suppose Hypotheses~A and~C hold. Let $\xi\in \Omega$ and suppose
	\begin{enumerate}
		\item[(i)] $D^2_\xi \tilde W$ is non-degenerate, \emph{i.e.} $\Ker(D_\xi^2\tilde W)=\{0\}$,
		\item[(ii)]  $\mathrm{Ker}D^2_{u_\xi}\calL_\xi= Z_\xi$, \label{hyp:kernel}
		\item[(iii)]  $\inf (\sigma(\nabla^2\calL_\xi (u_\xi))\cap (0,+\infty))>0$,
		\item[(iv)]  $p(D^2_\xi \tilde W)=n(D^2_{u_\xi}\Lcal_\xi)$.
	\end{enumerate}
	Then there exists $\delta>0$ such that
	\begin{equation}
		\label{eq:localcoercivitygssII}
		\forall v\in T_{u_\xi}\tilde\Sigma_{\xi}\cap \left(T_{u_\xi}\tilde\Ocal_{u_\xi}\right)^{\perp},\ D^2_{u_\xi}\calL_{\xi}(v,v)\ge \delta \| v\|^2.
	\end{equation}
\end{theorem}
It is clear that, when  the invariance group $G$ is one-dimensional, \emph{i.e.} $\text{dim}\ \frak g=1$, this theorem is identical to Theorem~\ref{thm:localcoercivity}. Indeed, then $G=G_\xi=G_{\mu_\xi}$ and hence $W=\tilde W$ so that both the assumptions and the conclusions of both theorems are identical. This is the situation studied in~\cite{gssI} and~\cite{stuart08}. The same conclusions hold true more generally when the group $G$ is abelian, since then again, $G_\xi=G_{\mu_\xi}=G$. In general, however, the groups $G_\xi$ and $G_{\mu_\xi}$ may be distinct, and so may therefore be the orbits $\tilde\Ocal_{u_\xi}$ and $\Ocal_{u_\xi}$. Hence,  a priori, the two approaches could yield different coercivity estimates and hence different stability results. Their comparison therefore needs to be done with care, a task we turn to in the next subsection. 

\begin{remark}\mbox{}
\begin{enumerate}
\item A proof of Theorem~\ref{thm:gssII} can be given along the same lines as the proof of Theorem~\ref{thm:localcoercivity} in Section~\ref{s:proofmaintheorem} and we don't reproduce it here. 
We point out that in fact only the bound $D^2_{u_\xi}\calL_{\xi}(v,v)\ge0$ is shown in~\cite{gssII}; the argument leading from that bound to~\eqref{eq:localcoercivitygssII} is the same as in the proof of Theorem~\ref{thm:localcoercivity} above. 
\item The theorem actually only necessitates a slightly weakened version of Hypotheses~A and~C. Indeed, once a $\xi\in\frak g$ is found satisfying the stationary equation~\eqref{eq:statequation}, only the subgroup $G_\xi$ of $G$ is still of relevance to its assumptions, its statement and its proof. In particular, it is sufficient to establish persistence of the relative equilibrium on an open subset $\Omega$ of $\frak g_\xi$, for the same fixed value of $\xi$. Of course, whenever the persistence result of Section~\ref{s:persistence} applies, this is not a real gain.  
\item The proof of Theorem~4.1 in~\cite{gssII} uses Theorem~3.1 of that same paper. We point out that the latter necessitates the unstated assumption that $\xi'\in\Omega\cap \frak g_\xi \to u_{\xi'}$ has an injective derivative at $\xi$. In Theorem~\ref{thm:gssII} above, this assumption follows from hypothesis~(i) and Lemma~\ref{lem:transversality}. Note also that, if persistence of the relative equilibrium is shown as in Section~\ref{s:persistence}, the assumption follows from the construction (Theorem~\ref{thm:persistenceb}).
\end{enumerate}
\end{remark}

\subsection{Comparing Theorem~\ref{thm:localcoercivity} to Theorem~\ref{thm:gssII}}
Let us first compare the respective conclusions~\eqref{eq:localcoercivity} and~\eqref{eq:localcoercivitygssII} as follows. Writing
\begin{equation}\label{eq:negcone}
\mathcal C_-=\{u\in\Ban\mid \HessL(u,u)<0\}
\end{equation}
for the negative cone of $\HessL_\xi$, we see that they imply that
\begin{equation}\label{eq:coneintersection}
T_{u_\xi}\Sigma_{\mu_\xi}\cap \mathcal C_-=\emptyset,\quad\text{respectively}\quad T_{u_\xi}\tilde\Sigma_{\xi}\cap \mathcal C_-=\emptyset,
\end{equation}
meaning that $T_{u_\xi}\Sigma$, respectively $T_{u_\xi}\tilde\Sigma_\xi$ are positive subspaces of $E$ for $\HessL$. Since $T_{u_\xi}\Sigma_{\mu_\xi}\subset T_{u_\xi}\tilde\Sigma_\xi$ the second of these statements implies the first and should in general be harder to obtain. Indeed, the cone $\calC_-$ may avoid $T_{u_\xi}\Sigma_{\mu_\xi}$ but have a non-trivial intersection with $T_{u_\xi}\tilde \Sigma_\xi$. This is further reflected in the fact that~\eqref{eq:coneintersection} implies that
$$
\text{codim}(T_{u_\xi}\Sigma_{\mu_\xi})\geq n(\HessL_\xi)\quad \text{respectively}\quad \text{codim}(T_{u_\xi}\tilde \Sigma_\xi)\geq n(\HessL_\xi).
$$
When $u_\xi$ is a regular relative equilibrium, one has $\text{dim}\frak g=\text{codim}(T_{u_\xi}\Sigma_{\mu_\xi})\geq \text{dim}\frak g_{\xi}=\text{codim}(T_{u_\xi}\tilde \Sigma_\xi)$.

To understand how the stronger conclusion comes about, one may note that condition~(iv) of Theorem~\ref{thm:gssII} has a more limited range of applicability than condition~(iv) of Theorem~\ref{thm:localcoercivity} since in general
\begin{equation}\label{eq:inequalities2}
p(D^2_\xi\tilde W)\leq p(D_\xi^2W)\leq n(\HessL_\xi).
\end{equation}
In particular, condition~(iv) of Theorem~\ref{thm:gssII} cannot be satisfied when $p(D^2_\xi\tilde W)< p(D_\xi^2W)$. To illustrate this phenomenon, we will give below a simple finite dimensional example where indeed
\begin{equation*}
p(D^2_\xi\tilde W)<p(D_\xi^2W)= n(\HessL_\xi),
\end{equation*}
so that Theorem~\ref{thm:localcoercivity} applies, but Theorem~\ref{thm:gssII} does not.

The following corollary further clarifies the link between the two results.
\begin{corollary}\label{cor:gssIIcor}
Suppose the hypotheses of Theorem~\ref{thm:gssII} are satisfied. Then 
$$
T_{u_\xi}\tilde\Ocal_{u_\xi}=T_{u_\xi}\Ocal_{u_\xi}
$$ 
so that
there exists $\delta>0$ such that
\begin{equation}
	\label{eq:localcoercivitygssIIbis}
\forall v\in T_{u_\xi}\tilde\Sigma_{\xi}\cap \left(T_{u_\xi}\Ocal_{u_\xi}\right)^{\perp},\ D^2_{u_\xi}\calL_{\xi}(v,v)\ge \delta \| v\|^2.
\end{equation}
Moreover hypotheses~(ii), (iii) and (iv) of Theorem~\ref{thm:localcoercivity} are satisfied. If, in addition, $u_\xi$ is a regular relative equilibrium, then $\frak g_{\mu_\xi}=\frak g_\xi$. 
\end{corollary}
We can conclude from the previous discussion and the corollary that, under the non-degeneracy hypothesis $\Ker(D_\xi^2W)=\{0\}$, Theorem~\ref{thm:localcoercivity} provides the desired coercivity estimate~\eqref{eq:localcoercivity} under weaker conditions than Theorem~\ref{thm:gssII}. As a result, to find a situation where Theorem~\ref{thm:gssII} does apply, whereas Theorem~\ref{thm:localcoercivity} does not, one has to suppose $\Ker(D_\xi^2W)\not=\{0\}$, whereas $\Ker(D^2_\xi \tilde W)=\{0\}$.  We did not find an example of such a situation. 
\begin{proof}
It follows from Lemma~\ref{lem:prophessWhessL}~(6) that $T_{u_\xi}\tilde\Ocal_{u_\xi}\subset T_{u_\xi}\Ocal_\xi$.  Also, Lemma~\ref{lem:prophessWhessL}~(4) implies that $T_{u_\xi}\Ocal_{u_\xi}\subset \Ker(\HessL_\xi\mid T_{u_\xi}\Sigma_{\mu_\xi})$. So, if $v\in T_{u_\xi}\Ocal_{u_\xi}\subset T_{u_\xi}\Sigma_{\mu_\xi}\subset T_{u_\xi}\tilde\Sigma_{\xi}$, then $\HessL_\xi(v,v)=0$. Writing 
$$
v=v_\parallel +v_\perp, \quad v_\parallel\in T_{u_\xi}\tilde\Ocal_{u_\xi}, \quad v_\perp \in (T_{u_\xi}\tilde\Ocal_{u_\xi})^\perp
$$
we have, 
\begin{eqnarray*}
0&=&\HessL(v, v)= \HessL_\xi(v_\parallel,  v_\parallel) +\HessL_\xi(v_\perp, v_\perp)+ 2 \HessL_\xi(v_\parallel, v_\perp)\\
&=&\HessL_\xi(v_\perp, v_\perp),
\end{eqnarray*}
since $v_\parallel\in \Ker(\HessL_\xi)$. It follows from~\eqref{eq:localcoercivitygssII} that $v_\perp=0$ so that $v\in T_{u_\xi}\tilde\Ocal_{u_\xi}$. We conclude that $T_{u_\xi}\tilde \Ocal_{u_\xi}=T_{u_\xi}\Ocal_{u_\xi}$ and hence~\eqref{eq:localcoercivitygssIIbis} follows from~\eqref{eq:localcoercivitygssII}. It also follows that hypothesis~(ii) of Theorem~\ref{thm:localcoercivity} is satisfied. Hypothesis~(iii) is the same in both theorems and hypothesis~(iv) of Theorem~\ref{thm:gssII}, together with~\eqref{eq:inequalities2}, implies hypothesis (iv) of Theorem~\ref{thm:localcoercivity}.

To prove the last statement, recall from Lemma~\ref{lem:prophessWhessL}~(6)  that $\frak g_\xi\subset \frak g_{\mu_\xi}$. Now let $\eta\in\frak g_{\mu_\xi}$. Since  $T_{u_\xi}\tilde \Ocal_{u_\xi}=T_{u_\xi}\Ocal_{u_\xi}$ there exists $\eta'\in \frak g_{\xi}$ so that $X_\eta(u_\xi)=X_{\eta'}(u_\xi)$. Since $F$ is regular at $u_\xi$, this implies $\eta'=\eta$, so that $\eta\in \frak g_\xi$, proving the result.
\end{proof}

To complete our comparative analysis of those two theorems, we further analyse the conditions on the kernel of $\HessL_\xi$ they impose. Similarly to the non-degeneracy condition~(i), those conditions are also not in a clear logical relation, in particular because they refer to two a priori different subgroups of $G$, namely $G_\xi$ and $G_{\mu_\xi}$. The following lemma sheds further light on the situation.
\begin{lemma}\label{lem:compare}
Suppose Hypotheses~A and~C are satisfied. Let $\xi\in\Omega$ and suppose $\Ker(D_\xi^2W)=\{0\}$. Then
\begin{equation}\label{eq:kernel=kernelrestrict}
\Ker(\HessL_\xi)=\Ker( \HessL_\xi\mid \Ker D_{u_\xi}F).
\end{equation}
In addition, the following two statements are equivalent:
\begin{itemize}
\item[(i)] $\Ker(\HessL_\xi)=T_{u_\xi}\tilde\Ocal_\xi$;
\item[(ii)] $\Ker(\HessL_\xi)=T_{u_\xi}\Ocal_\xi$ and $\frak g_\xi=\frak g_{\mu_\xi}$.
\end{itemize}
\end{lemma}
The lemma shows that, if $u_\xi$ is 
a non-degenerate (see definition at page 10), and hence regular, relative equilibrium,
then the condition on the kernel of $\HessL_\xi$ of Theorem~\ref{thm:gssII} implies not only the kernel condition in Theorem~\ref{thm:localcoercivity}, but in addition that $\frak g_\xi=\frak g_{\mu_\xi}$. This statement is independent of the other spectral conditons of these theorems on $D^2_{u_\xi}\Lcal_\xi$. 
\begin{proof} It follows from $\Ker(D^2_\xi W)=\{0\}$, together with Lemma~\ref{lem:prophessWhessL}~(1) and Lemma~\ref{lem:transversality} that $\Ban=\Ucal_\xi\oplus T_{u_\xi}\Sigma_{\mu_\xi}$. Hence~\eqref{eq:ortho} implies that
\begin{equation}\label{eq:KerHessLdecomposition}
\Ker(\HessL_\xi)= \Ker(\HessL_\xi\mid \Ucalxi)\oplus\Ker( \HessL_\xi\mid \Ker D_{u_\xi}F)
\end{equation}
On the other hand, it follows from~\eqref{eq:hessWhessL} that
$$
\Ker(\HessL_\xi\mid \Ucalxi)=D_\xi\tilde u(\Ker(D^2_\xi W))=\{0\},
$$
so that the first statement of the Lemma follows.\\
$(i)\Rightarrow (ii)$ From Lemma~\ref{lem:prophessWhessL}~(4), together with~\eqref{eq:kernel=kernelrestrict} and hypothesis~(i), we conclude that 
$$
T_{u_\xi}\Ocal_{u_\xi}\subset\Ker( \HessL_\xi\mid \Ker D_{u_\xi}F)=\Ker(\HessL_\xi) = T_{u_\xi}\tilde\Ocal_{u_\xi}\subset T_{u_\xi}\Sigma_{\mu_\xi}.
$$
On the other hand, as in the proof of Corollary~\ref{cor:gssIIcor}, we have $T_{u_\xi}\tilde\Ocal_{u_\xi}\subset T_{u_\xi}\Ocal_{u_\xi}$. Hence $T_{u_\xi}\tilde\Ocal_{u_\xi}= T_{u_\xi}\Ocal_{u_\xi}$ and the first statement of~(ii) follows.
The second statement is now proven as in the proof of Corollary~\ref{cor:gssIIcor}. \\
$(ii)\Rightarrow (i)$ This is obvious.
\end{proof}

\subsection{Proving orbital stability}
The above coercivity estimate~\eqref{eq:localcoercivitygssII} (or, equivalently~\eqref{eq:localcoercivitygssIIbis}) is used in~\cite{gssII} as an essential input to show the $G_\xi$-orbital stability of $u_\xi$. Note that this distinguishes their approach from the rest of the literature on orbital stability, including this paper and~\cite{debgenrot15}, where instead $G_{\mu_\xi}$-stability is proven. The argument given in~\cite{gssII}  (and also in~\cite{gssI}) leading from the above coercivity estimate to $G_\xi$-orbital stability of $u_\xi$ is however based on an implicit assumption on~$F$, referred to as Hypothesis~F in~\cite{debgenrot15}. It was explained in~\cite{debgenrot15} how, starting from a coercivity estimate, this condition is used in the cited works to obtain orbital stability for general perturbations of the relative equilibrium from orbital stability for perturbations within the constraint surface $\tilde \Sigma_\mu$: see Section~8.3 and Theorem~9 of~\cite{debgenrot15} 
However, it was also explained in that last paper that Hypothesis~F is typically not satisfied when $F$ takes values in $\R^m$, with $m>1$: it is therefore insufficient to deal with the situations under consideration in this paper as well as in~\cite{gssII}.  As recalled in Section~2 above, it is instead possible to use arguments provided in~\cite{debgenrot15} to prove $G_{\mu_\xi}$-relative stability, using~\eqref{eq:localcoercivity} (or, a fortiori~\eqref{eq:localcoercivitygssIIbis}) as a starting point.


\subsection{An example}\label{s:example}
We end this section with a simple  but illustrative example in $\Ban=\R^6$ where Theorem~\ref{thm:localcoercivity} applies, but Theorem~\ref{thm:gssII} does not. As already pointed out, $G$ must be non-commutative for this to happen.  Consider the SO$(3)$-invariant Hamiltonian 
\begin{equation}
H_{\alpha}(q,p)=H_0(q,p)-\alpha F^2(q,p),\quad \text{with}\quad H_0(q,p)=\frac{\|p\|^2}{2}+V(\|q\|), 
\end{equation}	
and $F(q,p)=q\wedge p$. Note that $H_0$ is the Hamiltonian of a particle in a central potential $V$ and that the components of the angular momentum vector $F$ generate rotations. Since $H_0$ Poisson commutes with $F^2$ and since $F^2$ generates rotations about $q\wedge p$, it is easy to see that the circular orbits of $H_0$ are also flow lines of $H_\alpha$ and that they are relative equilibria. These are the ones whose orbital stability we shall study. Consider for that purpose the stationary equation, with $\xi\in\R^3\simeq \frak{so}(3),$
$$
D_{u_\xi}H_\alpha-\xi\cdot D_{u_\xi}F=0,
$$
where $u_{\xi}=(q_\xi, p_\xi)$. A simple computation shows that any solution $u_\xi$ is of the form 
$u_{\xi}=(q_\xi, p_\xi)=(\rho_\xi\hat q_\xi, \sigma_\xi\hat p_\xi)$ with
$$
\sigma^2_\xi=\rho_\xi V'(\rho_\xi),\quad\hat q_\xi\cdot \hat p_\xi=0.
$$ 
Here $\rho_\xi>0, \sigma_\xi>0$ and we write $\hat a=a/\|a\|$ for each $a\in\R^3$.  We have 
$$
\mu_\xi=F(u_\xi)=q_\xi\wedge p_\xi=\rho_\xi\sigma_\xi \hat q_\xi\wedge \hat p_\xi, \text{with}\ \xi=\eta_{\alpha,\xi}\mu_\xi,\ \eta_{\alpha,\xi}=\frac{1-2\alpha\rho_{\xi}^2}{\rho_\xi^2}.
$$ 
Clearly, in this situation $G_\xi=G_{\mu_\xi}\simeq\mathrm{SO}(2)$. One has
$$
\hat F(\xi)=\rho_\xi\sigma_\xi \hat\xi =\rho_\xi^{3/2}\left[V'(\rho_\xi)\right]^{1/2}\hat\xi 
$$
$$
\|\xi\|=\rho_\xi\sigma_\xi \frac{|1-2\alpha\rho_{\xi}^2|}{\rho_\xi^2}=\rho_\xi^{3/2}\left[V'(\rho_\xi)\right]^{1/2}\frac{|1-2\alpha\rho_{\xi}^2|}{\rho_\xi^2}=\left[\rho_\xi^{-1}V'(\rho_\xi)\right]^{1/2}|1-2\alpha\rho_{\xi}^2|.
$$
For illustrative purposes, it is sufficient to consider $V(\rho_\xi)=\frac12\omega\rho_\xi^2$. Then it is clear that, provided $2\alpha> \rho_\xi^{-2}$, $\hat F$ is a local diffeomorphism and
$$
\|\xi\|=\omega^{1/2}(2\alpha\rho_\xi^2-1).
$$
 A simple computation then yields
$$
W(\xi)=H_\alpha(u_\xi)-\xi\cdot F(u_\xi)=\alpha\omega\rho_\xi^4=\frac14\frac{\omega}{\alpha}\left(1+\frac{\|\xi\|}{\sqrt\omega}\right)^2,
$$
and furthermore that
$$
D_\xi^2 W(v, v)=\frac{1}{2\alpha}\left[1+\frac{\sqrt\omega}{\|\xi\|}\right]v^2-\frac{\sqrt{\omega}}{2\alpha\|\xi\|}(v\cdot\hat\xi)^2\geq
\frac{v^2}{2\alpha},
$$
which is positive definite so that $p(D^2_\xi W)=3$. Hence, $n(D^2\Lcal_\xi)\geq 3$.  A further lengthy but straightforward computation shows that 
\begin{eqnarray*}
n(D^2\Lcal_\xi)=3,
\end{eqnarray*}
and that $\Ker(D^2_{u_\xi}\Lcal_\xi)=T_{u_\xi}\Ocal_{u_\xi}$. Hence Theorem~\ref{thm:localcoercivity} implies that the circular orbits are orbitally stable.

On the other hand the assumptions of \cite{gssII} are too strong to apply in this simple example. Indeed, in \cite{gssII} the authors consider the Hessian of function $W$ restricted to $\frak{g}_\xi$. The main hypothesis of their stability theorem is that $p(D^2_\xi\tilde W)=n(D_{u_\xi}^2 \Lcal_{\alpha,\xi})$. 
In the present situation this condition is not satisfied. Indeed,  since $\frak g_\xi=\text{so}(2)$, $p(D^2_\xi\tilde W)\le 1$. In fact, it is easy to see it is equal to $1$. As a consequence, $p(D^2_\xi\tilde W)<n(D_{u_\xi}^2 \Lcal_{\alpha,\xi})=3$ and so the hypotheses of Theorem~\ref{thm:gssII} are not satisfied.

\vskip 1cm

\noindent{\bf Conflict of Interest:} The authors declare that they have
no conflict of interest.





\end{document}